\documentclass[12pt]{article}
\usepackage{mathrsfs,amsthm,graphicx,color,verbatim,bbm,amsmath,amsfonts,amssymb,newclude,nicefrac,graphicx,enumerate,bm,geometry,mathabx}
\usepackage[hidelinks,colorlinks=true,linktoc=true]{hyperref}
\usepackage[T1]{fontenc}
\usepackage[capitalise]{cleveref} 

\crefname{subsection}{Subsection}{Subsections}
\crefname{enumitem}{item}{items}

\theoremstyle{plain}
\newtheorem{theorem}{Theorem}[section]
\newtheorem{lemma}[theorem]{Lemma}
\newtheorem{prop}[theorem]{Proposition}
\newtheorem{cor}[theorem]{Corollary}

\newtheorem{setting}[theorem]{Setting}

\theoremstyle{remark}

\theoremstyle{definition}

\newcommand{\A}{\mathbb{A}}
\newcommand{\E}{\mathbb{E}}

\newcommand{\I}{\mathbb{I}}
\renewcommand{\P}{\mathbb{P}}

\newcommand{\R}{\mathbb{R}}
\newcommand{\N}{\mathbb{N}}

\newcommand{\U}{\mathbb{U}}
\newcommand{\V}{\mathbb{V}}

\newcommand{\Y}{\mathbb{Y}}
\newcommand{\Z}{\mathbb{Z}}
\newcommand{\Borel}{\mathcal{B}}
\newcommand{\smallsum}{\textstyle\sum}

\newcommand{\Exp}[1]{ \E \! \left[ #1 \right]}

\newcommand{\EXP}[1]{ \E  [ #1 ]}
\newcommand{\EXPP}[1]{ \E \big[ #1 \big]}
\newcommand{\EXPPP}[1]{ \E \Big[ #1 \Big]}
\newcommand{\EXPPPP}[1]{ \E \bigg[ #1 \bigg]}

\newcommand{\norm}[1]{ \left\| #1 \right\| }

\newcommand{\qandq}{\qquad\text{and}\qquad}

\newcommand{\var}[1]{ \operatorname{Var}\!\left[ #1 \right]}

\newcommand{\VAR}[1]{ \operatorname{Var} \!\big[ #1 \big]}
\newcommand{\VARRR}[1]{ \operatorname{Var} \!\Big[ #1 \Big]}
\newcommand{\VARRRR}[1]{ \operatorname{Var} \!\bigg[ #1 \bigg]}
\newcommand{\VARRRRR}[1]{ \operatorname{Var} \!\Bigg[ #1 \Bigg]}

\newcommand{\Forall}{\forall\,}
\newcommand{\mc}{\mathcal}
\newcommand{\mf}{\mathfrak}
\newcommand{\is}{\curvearrowleft}

\makeatletter
\newcommand{\vast}{\bBigg@{3.5}}
\newcommand{\Vast}{\bBigg@{4}}
\makeatother

\newcounter{AuthorCount}
\stepcounter{AuthorCount}

\begin{document}


\title{
	Nonlinear Monte Carlo methods with \\
	polynomial runtime for high-dimensional \\
	iterated nested expectations
}

\author{
	Christian Beck$^{\arabic{AuthorCount}}	\stepcounter{AuthorCount}$,
	Arnulf Jentzen$^{\arabic{AuthorCount}\stepcounter{AuthorCount},\arabic{AuthorCount}\stepcounter{AuthorCount}}$, 
	and 
	Thomas Kruse$^{\arabic{AuthorCount}\stepcounter{AuthorCount}}$
	\bigskip
	\setcounter{AuthorCount}{1}
	\\
	\small{$^{\arabic{AuthorCount}
			\stepcounter{AuthorCount}}$ 
		Department of Mathematics, 
		ETH Zurich, 
		Z\"urich,}\\
	\small{Switzerland, 
		e-mail: christian.beck@math.ethz.ch}
	\\
	\small{$^{\arabic{AuthorCount}
			\stepcounter{AuthorCount}}$ 
		Department of Mathematics, 
		ETH Zurich, 
		Z\"urich,}\\
	\small{Switzerland, 
		e-mail: arnulf.jentzen@sam.math.ethz.ch} 
	\\
	\small{$^{\arabic{AuthorCount}\stepcounter{AuthorCount}}$ 
		Faculty of Mathematics and Computer Science, 
		University of M\"unster, }\\
	\small{M\"unster, Germany, e-mail: ajentzen@uni-muenster.de}\\
	\small{$^{\arabic{AuthorCount}\stepcounter{AuthorCount}}$ Institute of Mathematics, University of Gie{\ss}en, Gie{\ss}en,} \\
	\small{Germany, e-mail: thomas.kruse@math.uni-giessen.de}	
}

\maketitle

\begin{abstract}
The approximative calculation of iterated nested expectations is a recurring challenging problem in applications. 
Nested expectations appear, for example, in the numerical approximation of solutions of backward stochastic differential equations (BSDEs), in the numerical approximation of solutions of semilinear parabolic partial differential equations (PDEs), in statistical physics, in optimal stopping problems such as the approximative pricing of American or Bermudan options, in risk measure estimation in mathematical finance, or in decision-making under uncertainty. 
Nested expectations which arise in the above named applications often consist of a large number of nestings.  
However, the computational effort of standard nested Monte Carlo approximations for iterated nested expectations grows exponentially in the number of nestings and it remained an open question whether it is possible to approximately calculate multiply iterated high-dimensional nested expectations in polynomial time.
In this article we tackle this problem by proposing and studying a new class of full-history recursive multilevel Picard (MLP) approximation schemes for iterated nested expectations.
In particular, we prove under suitable assumptions that these MLP approximation schemes can approximately calculate multiply iterated nested expectations with a computational effort growing at most polynomially in the number of nestings $ K \in \N = \{1, 2, 3, \ldots \} $, in the problem dimension $ d \in \N $, and in the reciprocal $\nicefrac{1}{\varepsilon}$ of the desired approximation accuracy $ \varepsilon \in (0, \infty) $. 

\end{abstract}
 
\newpage

\tableofcontents

\section{Introduction}
\label{sect:intro}

The approximative calculation of iterated nested expectations is a recurring challenging problem in applications. 
Iterated nested expectations can, for instance, arise in the numerical approximation of solutions of backward stochastic differential equations (BSDEs) (see, e.g., \cite{BenderGaertnerSchweizer2018PathwiseDynamicProgramming,BouchardTouzi2004ApproximationOfBSDEs,GobetLemorWarin2005_ARegressionBasedMonteCarloMethodForBSDEs,PardouxPeng1990AdaptionSolutionOfBSDEs}), in the numerical approximation of solutons of semilinear parabolic partial differential equations (PDEs) (see, e.g., \cite{ElKarouiKapoudjianPardouxPeng1997ReflectedBSDEsAndObstacleProblems,PardouxPeng1992}), 
in  statistical physics (see, e.g., \cite{DimitsEtAl2013CoulombCollisions,HajiAliEtAl2018MCMcKeanVlasov,RosinRicketsonDimits2014MultilevelMCForCoulombCollisions,SzpruchTanTse2017IterativeParticleApproximationForMcKeanVlasov}), 
in optimal stopping problems such as the approximative pricing of American or Bermudan options (see, e.g., \cite{Andersen1999SimpleApproachToPricingBermudanSwaptions,Belomestny2011PricingBermudanOptions,BenderSchweizerZhuo2017PrimalDualAlgorithmForBSDEs,BroadieGlasserman1997PricingAmericans,Carriere1996ValuationEarlyExercise,Egloff2005_MonteCarloForOptimalStoppingAndStatisticalLearning,GoldbergChen2018BeatingTheCurseOfDimensionalityInOptionsPricingAndOptimalStopping,LongstaffSchwartz2001ValuingAmericanOptions,TsitsiklisVanRoy1999OptimalStoppingOfMarkovProcesses,TsitsiklisVanRoy2001RegressionMethodForAmericanOptions}), 
in risk measure estimation in mathematical finance (see, e.g., \cite{BroadieDuMoallemi2011EfficientRiskEstimation,BroadieDuMoallemi2015RiskRegression,GlassermanHeidelbergerShahabudding2000VarianceReductionForVaR,GlassermanHeidelbergerShahabuddin2002PortfolioVaRHeavyTailedRisk,GordyJuneja2010NestedSimulationPortfolioRisk,KornKornKroisandt2010MonteCarloMethodsAndModels,Peng2004NonlinearExpectations}), 
or in decision-making under uncertainty (see, e.g., \cite{AdesLuClaxton2004ExpectedValueOfSampleInfo,BratvoldBickelLohne2009ValueOfInformationOilGas,BrennanEtAl2007CalculatingPartialEVPI,NakayasuGodaTanaka2016EvaluatingValueSinglePointData}). 
In general, explicit solutions in closed form are not available for iterated nested expectations in the above mentioned problems and this entails a high demand for numerical approximation methods for the approximative calculation of iterated nested expectations. 
In the scientific literature we refer to \cite{BujokHamblyReisinger2015MultilevelSimulation,Giles2018MLMCNestedExpectations,GilesGoda2018DecisionMakingUnderUncertainty,GilesHaji-Ali2018MultilevelNestedSimulationForRiskEstimation,GodaHironakaIwamoto2020MultilevelMonteCarloForExpectedInformationGain,GodaMurakamiTanakaSato2018DecisionTheoreticSensitivityAnalysis,hironaka2019multilevel} for results on the numerical approximation of two nested expectations. 
Nested expectations which appear in the numerical approximation of solutions of BSDEs, in the numerical approximation of semilinear PDEs, or in the approximative pricing of American or Bermudan options usually do not only consist of two nested expectations but do instead consist of a large number of nestings. 
However, the computational effort of standard nested Monte Carlo approximations for iterated nested expectations grows exponentially in the number of nestings and it remained an open question whether it is possible to approximately calculate multiply iterated high-dimensional nested expectations in polynomial time.

It is the subject of this article to attack this problem by proposing and studying a new class of full-history recursive multilevel Picard approximation schemes for iterated nested expectations (in the following we abbreviate full-history recursive multilevel Picard by MLP). 
MLP approximation schemes were introduced in \cite{hutzenthaler2016multilevel,Overcoming} for the numerical approximation of solutions of semilinear PDEs and have been shown in \cite{MLPElliptic,AllenCahnApproximation2019,Overcoming,hutzenthaler2019overcoming} to overcome the curse of dimensionality in the numerical approximation of certain semilinear PDEs (cf.\ also \cite{BeckerEtAl2020MLPSimulations,EHutzenthalerJentzenKruse2019MLP,GilesWeltiJentzen2019GeneralisedMLP,HutzenthalerJentzenKruse2019MLPGradient,HutzenthalerKruse2020MLPApproximationsOfHighDimensionalSemilinearPDEsWithGradientDependentNonlinearities} for numerical simulations and further mathematical results for MLP approximation schemes). 
In this article we propose a new class of time-discrete MLP approximation schemes and we prove under suitable assumptions that these MLP approximation schemes can approximately calculate multiply iterated nested expectations with a computational effort growing at most polynomially in the number of nestings $ K \in \N =\{1, 2, 3, \ldots \} $, in the problem dimension $ d \in \N $, and in the reciprocal $ \nicefrac{1}{\varepsilon} $ of the desired approximation accuracy $ \varepsilon \in (0, \infty) $. 
To illustrate this article's main result, \cref{cor:complexity_general_dynamics} in \cref{subsec:complexity} below, we now present in \cref{intro_thm} a special case of \cref{cor:complexity_general_dynamics}. 

\begin{theorem} \label{intro_thm} 
	Let	$ L,T,\delta \in ( 0, \infty ) $, 
		$ \Theta = \bigcup_{ k \in \N } \! \Z^{ k } $, 
		$ f \in C( \R, \R ) $ 
	satisfy for all 
		$ a,b \in \R $ 
	that 
		$| f ( a ) - f ( b ) | \leq L | a - b |$, 
	let $( \Omega, \mc F, \P ) $ be a probability space, 
	let $ R^{ \theta } \colon \Omega \to ( 0, 1 ) $, $ \theta \in \Theta $, be i.i.d.\,random variables, 
	assume for all 
		$ t \in (0, 1) $
	that 
		$ \P ( R^{ 0 } \leq t ) = t $, 
	let $ W^{ d, \theta } \colon [0,T] \times \Omega \to \R^d $, $ d \in \N $, $ \theta \in \Theta $, be i.i.d.\ standard Brownian motions, 
	assume that 
	$ ( R^{ \theta } )_{ \theta \in \Theta } $ 
	and 
	$ ( W^{ d, \theta } )_{ ( d, \theta ) \in \N \times \Theta } $ 
	are independent, 
	let $ v^{ d, K }_{ k } \in C( \R^d, \R ) $, $ k \in \{ 0, 1, \ldots, K \} $, $ d, K \in \N $, satisfy for all 
		$ d, K \in \N $, 
		$ k \in \{ 1, 2, \ldots, K \} $, 
		$ x \in \R^d $ 
	that  
		$ | v^{ d, K }_{ 0 } ( x ) | \leq L $
	and 
		\begin{equation} \label{intro_thm:exponential_Euler_scheme}
		v^{ d, K }_{ k } ( x ) 
		= \EXPP{ v^{ d, K }_{ k - 1 } ( x + W^{ d, 0 }_{ T \slash K } ) + \tfrac{T}{K} f \big( v^{ d, K }_{ k - 1 } ( x + W^{ d, 0 }_{ T \slash K } ) \big) }, 
		\end{equation}
	let $ V^{ d, K, \theta }_{ k, n, M } \colon \R^d \times \Omega \to \R $, $ \theta \in \Theta $, $ k \in \{ 0, 1, \ldots, K \} $, $ n \in \N_0 $, $ d,K,M \in \N $, 
	satisfy\footnote{Note that for all $ t \in \R $ it holds that $ \lfloor t \rfloor = \max( (-\infty, t] \cap \Z ) $.}  for all
		$ d,K,M \in \N $, 
		$ n \in \N_0 $, 
		$ k \in \{ 0, 1, \ldots, K \} $, 
		$ \theta \in \Theta $, 
		$ x \in \R^d $ 
	that
	\begin{multline}\label{intro_thm:mlp_scheme}
		V^{ d, K, \theta }_{ k, n, M } ( x ) 
		= 
		\frac{ \mathbbm{1}_{\N}(n) }{ M^{ n } } \sum_{ m = 1 }^{ M^{ n } } 
		\left[ v^{ d, K }_{ 0 } ( x + W^{ d, ( \theta, 0, -m ) }_{ k T \slash K } ) + \tfrac{ k T }{ K } f ( 0 ) \right]  
		\\
		+ 
		\sum_{ j = 1 }^{ n - 1 } 
		\frac{ kT }{ K M^{ n - j } }
		\sum_{ m = 1 }^{ M^{ n - j } } \bigg[  
		f \big( V^{ d, K, ( \theta, j, m ) }_{ \lfloor k R^{ ( \theta, j, m ) } \rfloor, j, M } ( x + W^{ d, ( \theta, j, m ) }_{ k T \slash K } - W^{ d, ( \theta, j, m ) }_{ \lfloor k R^{ ( \theta, j, m ) } \rfloor T \slash K } ) \big) 
		\\
		- f \big( V^{ d, K, ( \theta, j, -m ) }_{ \lfloor k R^{ ( \theta, j, m ) } \rfloor, j - 1, M } ( x + W^{ d, ( \theta, j, m ) }_{ kT \slash K } - W^{ d, ( \theta, j, m ) }_{ \lfloor k R^{ ( \theta, j, m ) } \rfloor T \slash K } ) \big) \bigg], 
		\end{multline}
	and for every 
		$ d, K, M \in \N $, 
		$ n \in \N_0 $ 
	let $ \mf C^{ d, K }_{ M, n } \in \N_0 $ be the number of realizations of scalar standard normal random variables which are used to compute one realization of $V^{d, K, 0}_{K, n, M}(0) \colon \Omega \to \R$ (cf.\ \eqref{cor:complexity_mlp_exponential_euler:cost_inequality} for a precise definition). 
	Then there exist $ \mf N = ( \mf N_{ \varepsilon } )_{ \varepsilon \in (0,1] } \colon ( 0, 1 ] \to \N $ and $ \mf c \in \R $ such that for all 
		$ \varepsilon \in (0,1] $,
		$ d,K \in \N $
	it holds that 
		$ \mf C^{ d, K }_{ \mf N_{ \varepsilon }, \mf N_{ \varepsilon } } \leq \mf c d \varepsilon^{ -(2+\delta) } $ 
	and 
	$ \big( \EXP{ | V^{ d, K, 0 }_{ K, \mf N_{ \varepsilon }, \mf N_{ \varepsilon } } ( 0 ) - v^{ d, K }_{ K } ( 0 ) |^2 } \big)^{\!\nicefrac12 } \leq \varepsilon $. 
\end{theorem}

\cref{intro_thm} is an immediate consequence of \cref{cor:complexity_mlp_exponential_euler} in \cref{subsec:application} below. \cref{cor:complexity_mlp_exponential_euler}, in turn, follows from \cref{cor:complexity_general_dynamics}, the main result of this article. 
Roughly speaking, \cref{intro_thm} shows for every arbitrarily small real number $\delta \in (0,\infty)$ that MLP approximations can approximately calculate the function values $v^{d,K}_K(0)$, $d,K \in \N$, which are given recursively by the iterated nested expectations in \eqref{intro_thm:exponential_Euler_scheme} above with the root mean square error bounded by $ \varepsilon \in (0,\infty)$ and the computational cost bounded by a multiple of the product of $d\in\N$ and $\varepsilon^{-(2+\delta)}$. 
The real number $ T \in (0, \infty) $ and the function $f \in C(\R,\R)$ in \cref{intro_thm} above describe the time horizon and the nonlinearity in the nested expectations in \eqref{intro_thm:exponential_Euler_scheme} in \cref{intro_thm}. 
The nonlinearity $f$ in \cref{intro_thm} is assumed to be Lipschitz continuous with the Lipschitz constant bounded by $L\in (0,\infty)$. 
The real number $L\in (0,\infty)$ is also used to express a uniform boundedness assumption for the functions $v^{d,K}_0\colon \R^d\to\R$, $d,K\in\N$, in \cref{intro_thm}. 
The functions $v^{d,K}_K\colon \R^d \to \R$, $d,K\in\N$, are calculated approximately by means of the MLP approximation scheme in \eqref{intro_thm:mlp_scheme} in \cref{intro_thm}. 
The computational effort for evaluating the MLP approximations in \eqref{intro_thm:mlp_scheme} is measured by means of the natural numbers $\mf C^{d,K}_{M,n} \in \N$, $d,K,M\in\N$, $n\in\N_0$, in \cref{intro_thm} above. 
The specific recursion of the iterated nested expectations in \eqref{intro_thm:exponential_Euler_scheme} in \cref{intro_thm} above describes just one class of examples which can be calculated approximately by the MLP approximations proposed in this paper and we refer to  \cref{cor:complexity_general_dynamics} in \cref{subsec:complexity} below for our more general approximation result for iterated nested expectations.

The remainder of this article is organized as follows. 
In \cref{sec:MLP_approximations} we introduce in \cref{subsec:MLP_scheme} MLP approximation schemes for the approximation of iterated nested expectations (see \eqref{setting:mlp_scheme} in \cref{setting} in \cref{subsec:MLP_scheme}) and we establish in \cref{subsec:measurability_properties,subsec:distribution_properties,subsec:integrability_properties} basic measurability, integrability, and distribution properties for the proposed MLP approximations. 
In \cref{sec:analysis_of_MLP_approximations} we establish in \cref{subsec:bias_and_variance_estimates} bias and variance estimates for MLP approximations  and we provide in \cref{subsec:full_error_analysis} a full error analysis for the proposed MLP approximation schemes. 
In \cref{sec:complexity_analysis} we combine the error analysis for the proposed MLP approximation schemes in \cref{subsec:full_error_analysis} with a computational cost analysis for the proposed MLP approximation schemes to obtain a complexity analysis for the proposed MLP approximation schemes. 

\section[Full-history recursive multilevel Picard (MLP) approximations]{Full-history recursive multilevel Picard (MLP) approximations for iterated nested expectations}
\label{sec:MLP_approximations}

In this section we introduce in \cref{setting} in \cref{subsec:MLP_scheme} below MLP approximations for iterated nested expectations, we establish in \cref{lem:measurability} in \cref{subsec:measurability_properties} basic measurability properties for MLP approximations, we establish in \cref{lem:equidistribution,cor:equidistribution_for_mean_lemma} in \cref{subsec:distribution_properties} basic distribution properties for MLP approximations, and we establish in \cref{lem:integrability} in \cref{subsec:integrability_properties} below basic integrability properties for MLP approximations. 
Our proofs of \cref{lem:equidistribution}, \cref{cor:equidistribution_for_mean_lemma}, and \cref{lem:integrability} are based on some general and essentially well-known distribution properties for random fields which we establish in \cref{lem:elementary_measurability}, \cref{lem:equidistribution_random_fields}, \cref{lem:equidistributed_vectors}, and \cref{cor:equidistributed_sums} in \cref{subsec:random_fields} below (cf., e.g., also Hutzenthaler et al.~\cite[Subsection 2.2]{Overcoming} and Hutzenthaler et al.~\cite[Subsection 3.3]{hutzenthaler2019overcoming}).
The MLP approximations for iterated nested expectations introduced in \cref{subsec:MLP_scheme} below are inspired by the MLP approximations for semilinear PDEs introduced in Hutzenthaler et al.~\cite[Section 3]{Overcoming}.

\subsection{MLP approximation schemes for nested expectations}
\label{subsec:MLP_scheme}

\begin{setting}\label{setting} 
	Let $ d,K,M \in \N $, 
		$ \Theta = \bigcup_{ n \in \N } \! \Z^n $, 
		$ \mc O \in \Borel( \R^d )\setminus\{\emptyset\} $, 
	let $ f_{ k } \colon \mc O \times \R \to \R $, $ k \in \{ 0, 1, \ldots, K \} $,
	be $ \Borel( \mc O \times \R ) $/$ \Borel( \R ) $-measurable, 	
	let $ g \colon \mc O \to \R $ be $ \Borel( \mc O ) $/$ \Borel(\R) $-measurable, 	
	let $ ( S, \mc S ) $ be a measurable space, 
	let $ \phi_{ k } \colon \mc O\times S \to \mc O $, $ k \in \{ 0, 1, \ldots, K \} $, be $ ( \Borel( \mc O ) \otimes \mc S ) $/$ \Borel( \mc O ) $-measurable, 
	let $ ( \Omega, \mc F, \P ) $ be a probability space, 
	let $ W^{ \theta }_{ k } \colon \Omega \to S $, $ \theta \in \Theta $, $ k \in \{0, 1, \ldots, K\} $, be independent random variables, 
	assume for every 
		$ k \in \{0, 1, \ldots, K\} $ 
	that $ W^{ \theta }_{ k } \colon \Omega \to S $, $ \theta \in \Theta $, are identically distributed, 
	let $ \mc{R}^{ \theta } = (\mc{R}^{\theta}_{k})_{k \in \{0, 1, \ldots, K\}} \colon \{0, 1, \ldots, K\} \times \Omega \to \N_0 $, $ \theta \in \Theta $, be i.i.d.\,stochastic processes, 
	assume for every 
		$ k \in \{1, 2, \ldots, K\} $ 
	that $ \mc{R}^{ \theta }_{ k } \leq k-1 $,  
	let $ \mf p_{ k, l } \in ( 0, \infty ) $, $ k \in \{ 1, 2, \ldots, K \} $, $ l \in \N_0 $, satisfy for all 
		$ k \in \{ 1, 2, \ldots, K \} $, 
		$ l \in \N_{ 0 } $ 
	that 
		$ \mf p_{ k, l } \P( \mc{R}^{ 0 }_{ k } = l ) = |\P ( \mc{R}^{ 0 }_{ k } = l )|^2 $, 
	assume that $ ( \mc{R}^{ \theta })_{ \theta \in \Theta } $ and $ ( W^{ \theta }_{ k } )_{ ( \theta, k ) \in \Theta \times \{ 0, 1, \ldots, K \}} $ are independent, 
	let 
		$ X^{\theta,k} 
		= (X^{\theta,k,x}_l)_{(l,x)\in \{0,1,\ldots,k\}\times\mc O}\colon \{0,1,\ldots,k\}\times\mc O\times\Omega\to\mc O $, 
		$ k \in \{0,1,\ldots,K\} $,
		$ \theta \in \Theta $,  		 
	satisfy for all 
		$ \theta \in \Theta $,  
		$ k \in \{0,1,\ldots,K\} $,
		$ l \in \{0,1,\ldots,k\} $, 
		$ x \in \mc O $ 
	that 
		\begin{equation}\label{setting:dynamics}  
		X^{\theta,k,x}_l
		= 
		\begin{cases}
		x  & \colon l=k \\
		\phi_{l}( X^{\theta,k,x}_{l+1}, W^{\theta}_{l}) & \colon l<k, 
		\end{cases}
		\end{equation}
	and	let 
		$ V^{\theta}_{k,n} \colon \mc O \times \Omega \to \R $, $ n \in \N_0 $, $ k \in \{0,1,\ldots,K\} $, $ \theta\in\Theta $, 
	satisfy for all 
		$ \theta \in \Theta$, 
		$ k \in \{ 0, 1, \ldots, K \} $, 
		$ n \in \N_{ 0 } $, 
		$ x\in \mc O $ 
	that 
		\begin{equation} \label{setting:mlp_scheme}
		\begin{split}
		V^{ \theta }_{ k, n } ( x ) 
		& = 
		\frac{ \mathbbm{ 1 }_{ \N }( n ) }{ M^n } \sum_{ m = 1 }^{ M^{ n } }
		\left[ g( X^{ ( \theta, 0, -m ), k, x }_{ 0 } ) + \sum_{ l = 0 }^{ k - 1 } f_{ l }( X^{ ( \theta, 0, m ), k, x }_{ l }, 0 ) \right] 
		+ 
		\sum_{ j = 1 }^{ n - 1 } \frac{ 1 }{ M^{ n - j } } \sum_{ m = 1 }^{ M^{ n - j } } 
		\frac{ \mathbbm{ 1 }_{ \N }( k ) }{ \mf p_{ k, \mc{R}^{ ( \theta, j, m ) }_{ k } } } 
		\\
		& \cdot 
		\bigg[ \Big( f_{ \mc{R}^{ ( \theta, j, m ) }_{ k } } \big( X^{ ( \theta, j, m ), k, x }_{ \mc{R}^{ ( \theta, j, m ) }_{ k } }, V^{ ( \theta, j, m ) }_{ \mc{R}^{ ( \theta, j, m ) }_{ k }, j } ( X^{ ( \theta, j, m ), k, x }_{ \mc{R}^{ ( \theta, j, m ) }_{ k } } ) \big) 
		-
		V^{(\theta,j,m)}_{{\mc{R}^{(\theta,j,m)}_k},j}(X^{(\theta,j,m),k,x}_{\mc{R}^{(\theta,j,m)}_k})
		\Big) 
		\\
		& 
		- 
		\Big(f_{\mc{R}^{(\theta,j,m)}_k} \big( X^{(\theta,j,m),k,x}_{\mc{R}^{(\theta,j,m)}_k},  V^{(\theta,j,-m)}_{{\mc{R}^{(\theta,j,m)}_k},j-1}(X^{(\theta,j,m),k,x}_{\mc{R}^{(\theta,j,m)}_k})
		\big) 
		-
		V^{ ( \theta, j, -m ) }_{ \mc{R}^{ ( \theta, j, m ) }_{ k } , j - 1 }
		( X^{ ( \theta, j, m ), k, x }_{ \mc{R}^{ ( \theta, j, m ) }_{ k } } ) 
		\Big)
		\bigg].
		\end{split}
		\end{equation}
	\end{setting}

\subsection{Measurability properties for MLP approximations}
\label{subsec:measurability_properties}

\begin{lemma} \label{lem:elementary_measurability_X_processes} 
	Let $ d,K \in \N $, 
		$ \mc O \in \Borel(\R^d) \setminus \{ \emptyset \} $, 
	let $ (S,\mc S) $ be a measurable space, 
	let $ \phi_k \colon \mc O \times S \to \mc O $, $ k \in \{0,1,\ldots,K\} $, be $ (\Borel(\mc O)\otimes\mc S) $/$ \Borel(\mc O) $-measurable, 
	let $ (\Omega,\mc F,\P) $ be a probability space, 
	let $ \mc A \subseteq \mc F $ be a sigma-algebra on $ \Omega $, 
	let $ W_k \colon \Omega \to S $, $ k \in \{0,1,\ldots,K\} $, be random variables, 
	and let 
		$ X^k = (X^{k,x}_l)_{(l,x) \in \{0,1,\ldots,k\}\times\mc O } \colon \{0,1,\ldots,k\} \times \mc O \times \Omega \to \mc O $, $ k \in \{0,1,\ldots,K\} $,  
	satisfy for all 
		$ k \in \{0,1,\ldots,K\} $,  
		$ l \in \{0,1,\ldots,k\} $, 
		$ x \in \mc O $ 
	that 
		\begin{equation} \label{elementary_measurability_X_processes:dynamics}
		X^{k,x}_l = 
		\begin{cases}
		x & \colon l=k \\
		\phi_{l}(X^{k,x}_{l+1},W_{l}) & \colon l<k. 
		\end{cases}
		\end{equation} 
	Then 
	\begin{enumerate}[(i)] 
		\item \label{elementary_measurability_X_processes:item1}
		it holds for all 
			$ k,l \in \N_0 $
		with 
			$ l \leq k \leq K $ 
		that 	
			$ \mc O \times \Omega \ni (x, \omega) \mapsto X^{k,x}_l(\omega) \in \mc O $ 
		is 
			$ ( \Borel(\mc O)\otimes\sigma( (W_m)_{m\in [l,k)\cap\N_0} ) ) $/$ \Borel(\mc O) $-measurable 
		and 
		\item \label{elementary_measurability_X_processes:item2} 
		it holds for all 
			$ k,l \in \N_0 $ 
		with 
			$ l \leq k \leq K $ 
		that 
			$ \mc O \times \Omega \ni (x,\omega) \mapsto (X^{k,x}_l(\omega),\omega) \in \mc O \times \Omega $ 
		is 
			$ ( \Borel(\mc O) \otimes \sigma( \sigma((W_m)_{ m \in [l,k) \cap \N_0 } ) \cup \mc A ) ) $/$ ( \Borel(\mc O) \otimes \mc A ) $-measurable.  
	\end{enumerate}
\end{lemma} 

\begin{proof}[Proof of \cref{lem:elementary_measurability_X_processes}] 
	Throughout this proof let 
		$ \A_{k} \subseteq \N_0 $, $ k\in \{0,1,\ldots,K\} $, 
	satisfy for all 
		$ k \in \{ 0,1,\ldots,K \} $ 
	that 
		\begin{equation} 
		\A_{ k } 
		= 
		\left\{ 
			l \in \{ 0, 1, \ldots, k \} \colon 
			[ \Forall A \in \Borel(\mc O) \colon 
			(X^{k}_l)^{-1}(A) \in ( \Borel( \mc O ) \otimes \sigma( (W_m)_{m\in [l,k)\cap\N_0} ) ) ]
		\right\}\!.
		\end{equation} 
	Observe that the fact that for all 
		$ k \in \{ 0, 1, \ldots, K-1 \} $ 
	it holds that 
		$ W_{ k } \colon \Omega \to S $ 
	is $ \sigma( W_{ k } ) $/$ \mc S $-measurable 
	implies that for all 
		$ k,l \in \{0,1,\ldots,K\} $ 
	with 
		$ l+1 \in \A_k $ 
	it holds that 
		$ \mc O \times \Omega \ni (x,\omega) \mapsto ( X^{k,x}_{l+1}(\omega), W_l(\omega) ) \in \mc O \times S $ 
	is 
		$ (\Borel( \mc O ) \otimes \sigma( (W_m)_{ m\in [l,k)\cap\N_0} ) ) $/$ (\Borel( \mc O ) \otimes \mc S ) $-measurable. 
	The fact that for all  
		$ k \in \{ 0, 1, \ldots, K-1 \} $ 
	it holds that 	
		$ \phi_k \colon \mc O \times S \to \mc O $ 
	is $ ( \Borel( \mc O ) \otimes \mc S ) $/$ \Borel( \mc O ) $-measurable and \eqref{elementary_measurability_X_processes:dynamics} therefore ensure that for all 
		$ k,l \in \{0,1,\ldots,K\} $
	with $ l + 1 \in \A_{k} $ it holds that 
		$X^{k}_{l} \colon \mc O \times \Omega \to \mc O $ 
	is 
		$ ( \Borel( \mc O ) \otimes \sigma( (W_m)_{m\in [l,k)\cap \N_0 } ) ) $/$ \Borel( \mc O ) $-measurable. 
	Hence, we obtain that for all 
		$ k,l \in \{0,1,\ldots,K\} $ 
	with  
		$ l+1 \in \A_{k} $ 
	it holds that 
		$ l \in \A_{k} $. 
	Moreover, observe that \eqref{elementary_measurability_X_processes:dynamics} ensures for all 
		$ k \in \{0,1,\ldots,K\} $ 
	that 
		$ k \in \A_{k} $. 
	Induction and the fact that for all 
		$ k,l \in \{0,1,\ldots,K\} $ with $ l+1 \in \A_{k} $ 
	it holds that 
		$ l \in \A_{k} $  
	therefore yield that for all 
		$ k \in \{ 0, 1, \ldots, K \} $
	it holds that 
		$ \A_{k} = \{ 0, 1, \ldots, k \} $. 
	This establishes item \eqref{elementary_measurability_X_processes:item1}. 
	Next note that item~\eqref{elementary_measurability_X_processes:item1} ensures that for all 
		$ k,l \in \N_0 $, 
		$ B \in \Borel(\mc O) $
	with 
		$ l \leq k \leq K $ 
	it holds that 
		\begin{equation} \label{elementary_measurability_X_processes:first_projection}
		\begin{split} 
		\big\{ (x,\omega) \in \mc O \times \Omega \colon X^{k,x}_l(\omega) \in B \big\} 
		& \in ( \Borel( \mc O ) \otimes \sigma( (W_m)_{m\in [l,k)\cap\N_0}) ) 
		\\
		& \subseteq ( \Borel(\mc O )\otimes\sigma(\sigma( (W_m)_{m\in [l,k)\cap\N_0}) \cup \mc A) ). 
		\end{split} 
		\end{equation}
	Furthermore, observe that for all 
		$ k,l \in \N_0 $, 
		$ A \in \mc A $ 
	with 
		$ l \leq k \leq K $ 
	it holds that 
		\begin{equation} 
		\left\{ (x,\omega) \in \mc O \times\Omega \colon \omega \in A \right\} 
		= ( \mc O \times A ) \in ( \Borel( \mc O ) \otimes \mc A ) \subseteq ( \Borel( \mc O )\otimes\sigma(\sigma((W_m)_{m\in [l,k)\cap\N_0}) \cup \mc A) ). 
		\end{equation} 
	This and \eqref{elementary_measurability_X_processes:first_projection} imply that for all 
		$ k,l \in \N_0 $, 
		$ B \in \Borel( \mc O ) $, 
		$ A \in \mc A $ 
	with 
		$ l \leq k \leq K $ 
	it  holds that 
		\begin{equation} 
		\begin{split}
			& 
			\big\{ (x,\omega) \in \mc O \times \Omega \colon (X^{k,x}_l(\omega),\omega) \in B \times A \big\} 
			\\
			& = 
			\big( \big\{ (x,\omega) \in \mc O\times\Omega \colon X^{k,x}_l(\omega) \in B \big\} \cap 
			(\mc O \times A) \big) 
			\in \big( \Borel(\mc O)\otimes\sigma( \sigma((W_m)_{m\in [l,k)\cap\N_0}) \cup \mc A ) \big). 
		\end{split}
		\end{equation} 
	This establishes item~\eqref{elementary_measurability_X_processes:item2}. 
	This completes the proof of \cref{lem:elementary_measurability_X_processes}. 
\end{proof} 

\begin{lemma} \label{lem:measurability}
Assume 
	\cref{setting}. 
Then
	\begin{enumerate}[(i)] 
	\item \label{measurability:item1} 
		it holds for all 
			$ \theta \in \Theta $, 
			$ n \in \N_0 $ 
		that 
			$ \{ 0, 1, \ldots, K \} \times \mc O \times \Omega \ni ( k, x, \omega ) \mapsto V^{\theta}_{k,n}( x, \omega ) \in \R $ 
		is 		
		$ ( {\mf 2^{ \{0, 1, \ldots, K \}} \mathbbm 2}^{ \{0,1,\ldots,K\} } \otimes 
			\Borel( \mc O ) \otimes 
			\sigma( (\mc{R}^{ ( \theta, \vartheta ) }_{ s } )_{ ( s, \vartheta) \in \{1, 2, \ldots, K\} \times \Theta }, ( W^{ ( \theta, \vartheta ) }_s )_{ ( s, \vartheta ) \in \{ 0, 1, \ldots, K-1 \} \times \Theta } ) ) $/$ \Borel( \R ) $-measurable
		and 
	\item \label{measurability:item2} 
		it holds for all 
			$ \theta,\vartheta \in \Theta $, 
			$ j \in \N $, 
			$ k \in \{1, 2, \ldots, K\} $ 
		that 
			\begin{equation}
			\mc O \times \Omega \ni ( x, \omega ) 
			\mapsto 
			\left[ \tfrac{ 
			f_{ \mc{R}^{ \theta }_{ k } ( \omega ) } ( X^{ \theta, k, x }_{ \mc{R}^{ \theta }_{ k } ( \omega ) }( \omega ), V^{ \vartheta }_{ \mc{R}^{ \theta }_{ k } ( \omega ), j } ( X^{ \theta, k, x }_{ \mc{R}^{ \theta }_{ k } ( \omega ) }( \omega ), \omega ) )
			- 
		V^{ \vartheta }_{ \mc{R}^{ \theta }_{ k } ( \omega ), j } ( X^{ \theta, k, x }_{ \mc{R}^{ \theta }_{ k } ( \omega ) }( \omega ), \omega ) }{ \mf p_{ k, \mc{R}^{ \theta }_{ k }( \omega ) } } \right]
			\in \R
			\end{equation} 
		is 
			$ ( \Borel( \mc O ) \otimes \sigma( ( \mc{R}^{ \theta }_{ s } )_{ s \in \{1, 2, \ldots, K\} }, ( \mc{R}^{ ( \vartheta, \eta ) }_{ s } )_{ (s, \eta) \in \{1, 2, \ldots, K\} \times \Theta }, ( W^{ \theta }_{ s } )_{ s \in \{ 0, 1, \ldots, K-1 \} }, ( W^{ ( \vartheta, \eta ) }_{ s } )_{ ( s, \eta ) \in \{ 0, 1, \ldots, K-1 \} \times \Theta } ) ) $\allowbreak/$ \Borel( \R ) $-measurable.	
	\end{enumerate}
\end{lemma} 

\begin{proof}[Proof of \cref{lem:measurability}] 
	Throughout this proof let 
		$ \A \subseteq \N $ 
	satisfy
		\begin{multline} 
		\A =
		\Bigg\{ 
		n \in \N \colon 
		\bigg[ 
		\Forall \theta \in \Theta, k \in \{ 0, 1, \ldots, K \}, j \in \{ 0, 1, \ldots, n-1 \}, A \in \Borel(\R) \colon \\
		(V^{ \theta }_{ k, j })^{-1}(A) \in ( \Borel( \mc O ) \otimes 
		\sigma ( ( \mc{R}^{ ( \theta, \vartheta ) }_{ s } )_{ (s,\vartheta) \in \{1, 2, \ldots, K\} \times \Theta }, ( W^{ ( \theta, \vartheta ) }_{ s } )_{ ( s, \vartheta ) \in \{ 0, 1, \ldots, K-1 \} \times \Theta } ) ) 
		\bigg] 
		\Bigg\}.
		\end{multline} 
	Observe that the assumption that for all 
		$ \theta \in \Theta $, 
		$ k \in \{ 0, 1, \ldots, K \} $, 
		$ x \in \mc O $ 
	it holds that 
		$ V^{ \theta }_{ k, 0 } ( x ) = 0 $ 
	implies that 
		$ 1 \in \A $. 
	In the next step we note that item~\eqref{elementary_measurability_X_processes:item2} of \cref{lem:elementary_measurability_X_processes} ensures that for all 
		$ n \in \A $, 
		$ \theta,\vartheta \in \Theta $, 
		$ j \in \{ 0, 1, \ldots, n-1 \} $,
		$ k,m \in \{ 0, 1, \ldots, K \} $, 
		$ l \in \{ 0, 1, \ldots, k \} $
	it holds that 
		\begin{equation} \label{measurability:composition_of_V_and_X}	
		\mc O \times \Omega \ni ( x, \omega ) \mapsto V^{ \vartheta }_{ m, j }( X^{ \theta, k, x }_{ l }( \omega ), \omega ) \in \R
		\end{equation} 
	is $ ( \Borel( \mc O ) \otimes \sigma( ( W^{ \theta }_{ s } )_{ s \in \{ 0, 1, \ldots, K-1 \} }, ( \mc{R}^{ ( \vartheta, \eta ) }_{ s } )_{ (s, \eta) \in \{1, 2, \ldots, K\} \times  \Theta }, ( W^{ ( \vartheta, \eta ) }_{ s })_{ ( s, \eta ) \in \{ 0, 1, \ldots, K-1 \} \times \Theta } ) ) $/$ \Borel(\R) $-measurable. 
	The fact that 
		$ \Borel(\mc O) \otimes \Borel(\R) = \Borel( \mc O \times \R) $ 
	and item~\eqref{elementary_measurability_X_processes:item1} of \cref{lem:elementary_measurability_X_processes} 
	hence ensure that for all 
		$ n \in \A $, 
		$ \theta,\vartheta \in \Theta $, 
		$ j \in \{0,1,\ldots,n-1\} $,
		$ k,m \in \{0,1,\ldots,K\} $, 
		$ l \in \{0,1,\ldots,k\} $
	it holds that 	
		\begin{equation} \label{measurability:measurability_display_01}
		 \mc O \times \Omega \ni ( x, \omega ) \mapsto ( X^{ \theta, k, x }_{ l } ( \omega ), V^{ \vartheta }_{ m, j } ( X^{ \theta, k, x }_{ l } ( \omega ), \omega ) ) \in \mc O \times \R 
		\end{equation} 
	is $ (\Borel( \mc O ) \otimes \sigma ( 
	( W^{ \theta }_{ s } )_{ s \in \{ 0, 1, \ldots, K-1 \} }, ( \mc{R}^{ ( \vartheta, \eta ) }_{ s } )_{ ( s, \eta ) \in \{1, 2, \ldots, K \} \times \Theta }, ( W^{ ( \vartheta, \eta ) }_{ s } )_{ ( s, \eta ) \in \{ 0, 1, \ldots, K-1 \} \times \Theta } ) ) $/$ ( \Borel( \mc O \times \R ) ) $-{mea\-su\-rable}.
	Combining this with the assumption that for all 
		$ k \in \{ 0, 1, \ldots, K-1 \} $ 
	it holds that 
		$ f_k \colon \mc O \times \R \to \R $ is $\Borel( \mc O \times \R ) $/$ \Borel(\R) $-measurable 
	and the fact that for all 
		$ k,j \in \N_0 $, 
		$ \theta,\vartheta\in\Theta	$, 
		$ x \in \mc O $, 
		$ \omega \in \Omega $ 
	with 
		$ 1 \leq k \leq K $ 
	it holds that 
		\begin{equation} 
		\begin{split} 
		& 
		\frac{ f_{ \mc{R}^{ \theta }_{ k } ( \omega ) } 
		\big( X^{ \theta, k, x }_{ \mc{R}^{ \theta }_{ k } ( \omega ) }( \omega ), V^{ \vartheta }_{ \mc{R}^{ \theta }_{ k } ( \omega ), j }( X^{ \theta, k, x }_{ \mc{R}^{ \theta }_{ k } ( \omega ) }( \omega ) ) \big) }{ \mf p_{ k, \mc{R}^{ \theta }_{ k }( \omega ) } } 
		= 
		\sum_{ l = 0 }^{ k - 1 } \mathbbm{1}_{ \{ \mc{R}^{ \theta }_{ k } = l \}}(\omega) 
		\left[ \frac{ f_{ l } \big( X^{\theta,k,x}_{l}(\omega), V^{\vartheta}_{l,j}( X^{\theta,k,x}_{l}(\omega) ) \big)  }{ \mf p_{ k, l } } \right]
		\end{split} 
		\end{equation} 
	ensures that for all 
		$ n \in \A $, 
		$ \theta,\vartheta \in \Theta $, 
		$ k \in \{0,1,\ldots,K\} $, 
		$ j \in \{0,1,\ldots,n-1\} $
	it holds that 
		\begin{equation} \label{measurability:measurability_display_02}
		\mc O \times \Omega \ni ( x, \omega ) \mapsto f_{ \mc{R}^{ \theta }_{ k } ( \omega ) }( X^{ \theta, k, x }_{ \mc{R}^{ \theta }_{ k } ( \omega ) }( \omega ), V^{ \vartheta }_{ \mc{R}^{ \theta }_{ k }( \omega ), j }( X^{ \theta, k, x }_{ \mc{R}^{ \theta }_{ k }( \omega ) }( \omega ) ) ) \in \R 
		\end{equation} 
	is $ (\Borel(\mc O) \otimes \sigma( 
		( \mc{R}^{ \theta }_{ s } )_{ s \in \{1, 2, \ldots, K\} } ),
		( W^{ \theta }_{ s } )_{ s \in \{ 0, 1, \ldots, K-1 \} }, ( \mc{R}^{ ( \vartheta, \eta ) }_{ s } )_{ (s, \eta ) \in \{1, 2, \ldots, K \} \times \Theta }, ( W^{ ( \vartheta, \eta ) }_{ s } )_{ ( s, \eta ) \in \{ 0, 1, \ldots, K-1 \} \times \Theta } ) 
		) $/\allowbreak$ \Borel(\R) $-{mea\-su\-rable}.
	The assumption that for all 
		$ k \in \{ 0, 1, \ldots, K-1 \} $ 
	it holds that $ f_{ k } \colon \mc O \times \R \to \R $ is $ \Borel( \mc O \times \R ) $/$ \Borel( \R ) $-measurable, 
	the assumption that $ g \colon \mc O \to \R $ is $ \Borel( \mc O ) $/$ \Borel(\R) $-measurable, 
	\eqref{setting:mlp_scheme},
	\eqref{measurability:composition_of_V_and_X}, 
	\eqref{measurability:measurability_display_01}, 
	and item~\eqref{elementary_measurability_X_processes:item1} of \cref{lem:elementary_measurability_X_processes} therefore 
	guarantee that for all 
		$ n \in \A $,
		$ \theta \in \Theta $, 
		$ k \in \{ 0, 1, \ldots, K \} $ 
	it holds that 
		$ V^{ \theta }_{ k, n } \colon \mc O \times \Omega \to \R $ 
	is 
		$ (\Borel( \mc O ) \otimes \sigma ( ( \mc{R}^{ ( \theta, \vartheta ) }_{ s } )_{ ( s, \vartheta ) \in \{1, 2, \ldots, K\} \times \Theta }, ( W^{ \theta, \vartheta }_{ s } )_{ ( s, \vartheta ) \in \{ 0, 1, \ldots, K-1 \} \times \Theta } ) ) $/$ \Borel( \R ) $-measurable.
	Hence, we obtain that for all 
		$ n \in \A $ 
	it holds that  
		$ n+1 \in \A $. 
	Induction and the fact that $ 1 \in \A $ therefore demonstrate that $ \A = \N $. 
	This establishes item~\eqref{measurability:item1}. 
	In the next step we note that item~\eqref{measurability:item1} and \eqref{measurability:measurability_display_02} establish item~\eqref{measurability:item2}. 
	This completes the proof of \cref{lem:measurability}. 
\end{proof}

\subsection{Distribution properties for general random fields}
\label{subsec:random_fields}

\begin{lemma} \label{lem:elementary_measurability}
	Let $ (\Omega, \mc F) $ be a measurable space, 
	let $ \mc G \subseteq \mc F $ and $ \mc A \subseteq \mc F $ be sigma-algebras on $ \Omega $, 
	let $ ( E, \mc E ) $, $ ( S, \mc S ) $, and $ ( T, \mc T ) $ be measurable spaces, 
	let $ U \colon S \times \Omega \to T $ be $ ( \mc S \otimes \mc G ) $/$ \mc T $-measurable, 
	and let $ Y \colon E \times \Omega \to S $ be $ ( \mc E \otimes \mc A ) $/$ \mc S $-measurable. 
	Then it holds that $ U(Y) = ( E \times \Omega \ni (e,\omega) \mapsto U(Y(e,\omega),\omega) \in T ) $ 
	is $ ( \mc E \otimes \mc F ) $/$ \mc T $-measurable. 
\end{lemma} 

\begin{proof}[Proof of \cref{lem:elementary_measurability}]
	First, note that the assumption that $ \mc A \subseteq \mc F $ and the assumption that $ Y $ is $ ( \mc E \otimes \mc A ) $/$ \mc S $-measurable ensures that 
	$ Y \colon E \times \Omega \to S $ 
	is $ ( \mc E \otimes \mc F ) $/$ \mc S $-measurable. 
	This and the assumption that $ \mc G \subseteq \mc F $ imply that $ E \times \Omega \ni (e, \omega) \mapsto ( Y(e, \omega), \omega ) \in S \times \Omega $ is $ ( \mc E \otimes \mc F ) $/$ ( \mc S \otimes \mc G ) $-measurable. 
	The assumption that $ U $ is $ ( \mc S \otimes \mc G ) $/$ \mc T $-measurable hence demonstrates that $ E \times \Omega \ni (e,\omega) \mapsto U(Y(e,\omega),\omega) \in T $ is $ ( \mc E \otimes \mc F ) $/$ \mc T $-measurable. 	
	This completes the proof of \cref{lem:elementary_measurability}. 
\end{proof} 

\begin{lemma} \label{lem:equidistribution_random_fields}
	Let $ ( \Omega , \mc F , \P ) $ be a probability space, 
	let $ \mc G_i \subseteq \mc F $, $ i \in \{ 1, 2 \} $, and $ \mc A_i \subseteq \mc F $, $ i \in \{ 1, 2 \} $, be sigma-algebras on $ \Omega $, 
	let $ ( E, \mc E ) $, $ ( S, \mc S ) $, and $ ( T, \mc T ) $  be measurable spaces, 
	let $ U_i \colon S \times \Omega \to T $, $ i \in \{ 1, 2 \} $, be identically distributed random fields, 
	let $ Y_i \colon E \times \Omega \to S $, $ i \in \{ 1, 2 \} $, be identically distributed random fields, 
	assume for all 
	$ i \in \{1,2\} $ 
	that $ U_i $ is $ ( \mc S \otimes \mc G_i ) $/$ \mc T $-measurable, 
	assume for all 
	$ i \in \{1,2\} $ 
	that $ Y_i $ is $ ( \mc E \otimes \mc A_i ) $/$ \mc S $-measurable, 
	and	assume for all 
	$ i \in \{1,2\} $ 
	that $ \mc G_i $ and $ \mc A_i $ are independent. 
	Then
	\begin{enumerate}[(i)]
		\item \label{equidistribution_random_fields:item1}
		it holds for all $ i \in \{1,2\} $ that 
		$ U_i(Y_i) = ( E \times \Omega \ni (e,\omega) \mapsto U_i(Y_i(e,\omega),\omega) \in T ) $ 
		is $ ( \mc E \otimes \mc F ) $/$ \mc T $-measurable and 
		\item \label{equidistribution_random_fields:item2}
		it holds that 
		$ U_1(Y_1) $ 
		and $ U_2(Y_2) $ 
		are identically distributed random fields. 
	\end{enumerate}
\end{lemma}

\begin{proof}[Proof of \cref{lem:equidistribution_random_fields}] 
	First, observe that \cref{lem:elementary_measurability} establishes item~\eqref{equidistribution_random_fields:item1}. 
	Next we prove item~\eqref{equidistribution_random_fields:item2}. 
	For this let 
	$ n \in \N $, 
	$ e_1, e_2, \ldots, e_n \in E $, 
	$ B_1, B_2, \ldots, B_n \in \mc T $, 
	let $ \U_i \colon S^n \times \Omega \to T^n $, $ i \in \{1,2\} $, satisfy for all 
	$ i \in \{1,2\} $, 
	$ s = ( s_1, s_2, \ldots, s_n ) \in S^n $, 
	$ \omega \in \Omega $
	that 
	\begin{equation} 
	\U_i(s,\omega) = (U_i(s_1,\omega), U_i(s_2,\omega), \ldots, U_i(s_n,\omega)), 
	\end{equation} 
	let 
	$ \Y_i \colon \Omega \to S^n $, $ i \in \{1,2\} $, 
	satisfy for all 
	$ i \in \{1,2\} $, $ \omega \in \Omega $ 
	that 
	\begin{equation} 
	\Y_i(\omega) = (Y_i(e_1,\omega),Y_i(e_2,\omega),\ldots,Y_i(e_n,\omega)), 
	\end{equation}
	and let 
	$ \I \colon T^n \to [0,\infty) $ 
	satisfy for all 
	$ t = (t_1,t_2,\ldots,t_n) \in T^n $ 
	that 
	\begin{equation}
			\I(t) = \mathbbm{1}_{B_1 \times B_2 \times \ldots \times B_n}(t_1, t_2, \ldots, t_n).
	\end{equation} 
	Note that the assumption that for all 
	$ i \in \{1,2\} $ 
	it holds that 
	$ Y_i $ 
	is 
	$ (\mc E \otimes \mc A_i ) $/$ \mc S $-measurable
	ensures that for all 
	$ i \in \{1,2\} $, 
	$ j \in \{1,2,\ldots,n\} $
	it holds that 
	$ \Omega \ni \omega \mapsto Y_i(e_j,\omega) \in S $ 
	is 
	$ \mc A_i $/$ \mc S $-measurable. 	
	This implies that for all 
	$ i \in \{1,2\} $ 
	it holds that 
	$ \Y_i $ 
	is 
	$ \mc A_i $/$ \mc S^{\otimes n} $-measurable. 
	Moreover, observe that the assumption that for all 
	$ i \in \{1,2\} $ 
	it holds that 
	$ U_i $ 
	is 
	$ (\mc S \otimes \mc G_i) $/$ \mc T $-measurable 
	guarantees that for all 
	$ i \in \{1,2\} $, 
	$ j \in \{1,2,\ldots,n\} $ 
	it holds that 
	$ S^n\times\Omega \ni (s_1,s_2,\ldots,s_n,\omega) \mapsto U_i(s_j,\omega) \in \mc T $ 
	is 
	$ (\mc S^{\otimes n}\otimes \mc G_i) $/$ \mc T $-measurable. 
	Therefore, we obtain for all 
	$ i \in \{1,2\} $ 
	that 
	$ \U_i $ 
	is 
	$ (\mc S^{\otimes n} \otimes \mc G_i) $/$ \mc T^{\otimes n} $-measurable. 
	The fact that 
	$ \I $ is $ \mc T^{\otimes n}$/$\Borel([0,\infty)) $-measurable 
	hence implies that for all 
	$ i \in \{1,2\} $ 
	it holds that 
	$ \I \circ \U_i $ 
	is 	
	$ (\mc S^{\otimes n} \otimes \mc G_i) $/$ \Borel([0,\infty)) $-measurable. 
	Combining this, the fact that for all 
	$ i \in \{1,2\} $ 
	it holds that 
	$ \Y_i $ 
	is 
	$ \mc A_i $/$ \mc S^{\otimes n} $-measurable, 
	the hypothesis that for every 
	$ i \in \{1,2\} $ 
	it holds that 
	$ \mc G_i $ and $ \mc A_i $ 
	are independent, and Hutzenthaler et al.~\cite[Lemma 2.2]{Overcoming} (applied with 
	$ (\Omega, \mc F, \P) \is (\Omega,\mc F, \P) $, 
	$ \mc G \is \mc G_i $, 
	$ (S, \mc S) \is (S^n,\mc S^{\otimes n}) $, 
	$ U \is \I\circ\U_i $, 
	$ Y \is \Y_i $ 
	for $ i \in \{1,2\} $
	in the notation of Hutzenthaler et al.~\cite[Lemma 2.2]{Overcoming}) assure that for all 
	$ i \in \{1,2\} $ 
	it holds that 
	\begin{equation} \label{U_refined_measurability:calculation}
	\begin{split}
	&\P\big(U_i(Y_i(e_1))\in B_1, U_i(Y_i(e_2)) \in B_2, \ldots, U_i(Y_i(e_n))\in B_n\big) 
	\\[1ex]
	&= 
	\Exp{ \mathbbm{1}_{B_1 \times B_2 \times \ldots \times B_n} ( U_i(Y_i(e_1)),  U_i(Y_i(e_2)), \ldots, U_i(Y_i(e_n)) )  }
	= 
	\Exp{ \I(\U_i(\Y_i)) } 
	\\
	&= 
	\Exp{ (\I\circ\U_i)(\Y_i) } 
	= 
	\int_{S^n} \Exp{ (\I\circ \U_i)(s_1,s_2,\ldots,s_n) }\!\,(\Y_i(\P))_{\mc S^{\otimes n}}(d(s_1,s_2,\ldots,s_n)). 
	\end{split}
	\end{equation} 
	Moreover, note that the hypothesis that $ U_1 $ and $ U_2 $ are identically distributed yields that for all 
	$ s_1, s_2, \ldots, s_n \in S $ 
	it holds that 
	\begin{equation} \label{U_refined_measurability:expectations}
	\begin{split}
	& \Exp{(\I\circ\U_1)(s_1,s_2,\ldots,s_n)} 
	= 
	\P( U_1(s_1) \in B_1, U_1(s_2) \in B_2, \ldots, U_1(s_n) \in B_n) 
	\\
	& = 
	\P( U_2(s_1) \in B_1, U_2(s_2) \in B_2, \ldots, U_2(s_n) \in B_n) 
	= 
	\Exp{(\I\circ\U_2)(s_1,s_2,\ldots,s_n)}\!. 
	\end{split}
	\end{equation} 
	In addition, observe that the hypothesis that $ Y_1 $ and $ Y_2 $ are identically distributed ensures that $ (\Y_1(\P))_{\mc S^{\otimes n}} = (\Y_2(\P))_{\mc S^{\otimes n}} $. 
	This, \eqref{U_refined_measurability:expectations}, and \eqref{U_refined_measurability:calculation} ensure that 
	\begin{equation}
	\begin{split}
	& \P(U_1(Y_1(e_1))\in B_1, U_1(Y_1(e_2)) \in B_2, \ldots, U_1(Y_1(e_n))\in B_n) 
	\\
	& = 
	\int_{S^n} \Exp{ (\I\circ \U_1)(s_1, s_2,\ldots, s_n) }\!\,(\Y_1(\P))_{\mc S^{\otimes n}}(d(s_1, s_2, \ldots, s_n))
	\\
	& = 
	\int_{S^n} \Exp{ (\I\circ \U_2)(s_1, s_2, \ldots, s_n) }\!\,(\Y_2(\P))_{\mc S^{\otimes n}}(d(s_1, s_2, \ldots, s_n))
	\\		
	& = 
	\P(U_2(Y_2(e_1))\in B_1, U_2(Y_2(e_2))\in B_2, \ldots, U_2(Y_2(e_n))\in B_n).  
	\end{split}
	\end{equation} 
	Hence, we obtain that $ U_1(Y_1) $ and $ U_2(Y_2) $ are identically distributed random fields. 
	This establishes item~\eqref{equidistribution_random_fields:item2}. 
	This completes the proof of \cref{lem:equidistribution_random_fields}.
\end{proof}

\begin{lemma} \label{lem:equidistributed_vectors} 
	Let $ N \in \N $, 
	let $ ( \Omega, \mc F, \P ) $ be a probability space,  
	let $ ( E, \mc E ) $ and $ ( S, \mc S ) $ be measurable spaces, 
	let $ \mc G_i \subseteq \mc F $, $ i \in \{ 1, 2, \ldots, N \} $, be independent sigma-algebras on $ \Omega $, 
	let $ \mc A_i \subseteq \mc F $, $ i \in \{ 1, 2, \ldots, N \} $, be independent sigma-algebras on $ \Omega $,  
	let $ U_i \colon S \times \Omega \to E $, $ i \in \{1,2,\ldots,N\} $, 
	satisfy for all 
	$ i \in \{1,2,\ldots,N\} $ 
	that $ U_i $ is $ (\mc S \otimes \mc G_i) $/$ \mc E $-measurable, 
	let $ V_i \colon S \times \Omega \to E $, $ i \in \{1,2,\ldots,N\} $, 
	satisfy for all 
	$ i \in \{1,2,\ldots,N\} $ 
	that $ V_i $ is $ (\mc S \otimes \mc A_i) $/$ \mc E $-measurable, 
	and assume for all 
	$ i \in \{1,2,\ldots,N\} $ 
	that $ U_i $ and $ V_i $ are identically distributed random fields. 
	Then it holds that 
	$ (U_1,U_2,\ldots,U_N) = ( S \times \Omega \ni (s,\omega) \mapsto (U_1(s,\omega),U_2(s,\omega),\ldots,U_N(s,\omega)) \in E^N ) $ 
	and 
	$ (V_1,V_2,\ldots,V_N) = ( S \times \Omega \ni (s,\omega) \mapsto (V_1(s,\omega),V_2(s,\omega),\ldots,V_N(s,\omega)) \in E^N ) $  
	are identically distributed random fields. 
\end{lemma}

\begin{proof}[Proof of \cref{lem:equidistributed_vectors}]
	Throughout this proof let 
	$ n \in \N $, 
	$ s_1, s_2,\ldots,s_n \in S $, 
	let 
	$ B_{i,j} \in \mc E $, $ i\in\{1,2,\ldots,N\} $, $ j \in \{1,2,\ldots,n\}$, 
	and let 
	$ \U \colon S \times \Omega \to E^N $ and $ \V \colon S \times \Omega \to E^N $
	satisfy for all 
	$ s \in S $, 
	$ \omega \in \Omega $ 
	that 
	$ \U(s,\omega) = (U_1(s,\omega), U_2(s,\omega),\ldots, U_N(s,\omega)) $ 
	and 
	$ \V(s,\omega) = (V_1(s,\omega), V_2(s,\omega),\ldots, V_N(s,\omega)) $. 
	Observe that the assumption that for all 
	$ i \in \{1,2,\ldots,N\} $ 
	it holds that 
	$ U_i \colon S\times\Omega \to E $ 
	is $ (\mc S\otimes\mc G_i) $/$ \mc E $-measurable implies that for all 
	$ i \in \{1,2,\ldots,N\} $, 
	$ j \in \{1,2,\ldots,n\} $ 
	it holds that 
	$ \Omega \ni \omega \mapsto U_i(s_j,\omega) \in E $ 
	is 
	$ \mc G_i $/$ \mc E $-measurable. 
	This implies that for all 
	$ i \in \{1,2,\ldots,N\} $ 
	it holds that 
	$ \big( \bigcap_{j=1}^n \{ U_i(s_j) \in B_{i,j} \} \big) \in \mc G_i  $.  
	The fact that $ \mc G_i $, $ i \in \{1,2,\ldots,N\} $, are independent sigma-algebras on $ \Omega $ therefore ensures that 
	\begin{equation}
	\begin{split}
	& 
	\P\!\left( 
	\U(s_1) \in \left[ \bigtimes_{i=1}^N B_{i,1} \right]\!, 
	\U(s_2) \in \left[ \bigtimes_{i=1}^N B_{i,2} \right]\!, 
	\ldots, 
	\U(s_n) \in \left[ \bigtimes_{i=1}^N B_{i,n} \right] 
	\right)
	\\
	& 
	= 
	\P\!\left( \bigcap_{i=1}^N \bigcap_{j=1}^n \left\{ U_i(s_j) \in B_{i,j} \right\} \right)
	= 
	\prod_{i=1}^N \P\!\left(\bigcap_{j=1}^n \left\{ U_i(s_j) \in B_{i,j} \right\} \right)\!. 
	\end{split}
	\end{equation}
	The assumption that for all 
	$ i \in \{ 1, 2, \ldots, N \} $ 
	it holds that $ U_i $ and $ V_i $ are identically distributed hence ensures that 
	\begin{equation} 
	\begin{split}
	& 
	\P\!\left( 
	\U(s_1) \in \left[ \bigtimes_{i=1}^N B_{i,1} \right]\!, \U(s_2) \in \left[ \bigtimes_{i=1}^N B_{i,2} \right]\!, \ldots, 
	\U(s_n) \in \left[ \bigtimes_{i=1}^N B_{i,n} \right] \right)
	\\
	& 
	= 
	\prod\limits_{i=1}^N \P\!\left(\bigcap_{j=1}^n \left\{ V_i(s_j) \in B_{i,j} \right\} \right)\!. 
	\end{split}
	\end{equation} 
	Combining this with the fact that for all 
		$ i \in \{1,2,\ldots,N\} $ 
	it holds that $ \big(\bigcap_{j=1}^n \{ V_i(s_j) \in B_{i,j} \}\big) \in \mc A_i $ and the assumption that $ \mc A_i $, $ i\in\{1,2,\ldots,N\} $, are independent sigma-algebras on $ \Omega $ yields that 
	\begin{equation} 
	\begin{split}
	& 
	\P\!\left( 
	\U(s_1) \in \left[ \bigtimes_{i=1}^N B_{i,1}\right]\!, \U(s_2) \in \left[ \bigtimes_{i=1}^N B_{i,2}\right]\!, \ldots, 
	\U(s_n) \in \left[ \bigtimes_{i=1}^N B_{i,n}\right] \right)
	\\
	& =
	\P\!\left( \bigcap_{i=1}^N \bigcap_{j=1}^n \left\{ V_i(s_j) \in B_{i,j} \right\} \right) 
	\\
	& =
	\P\!\left( 
	\V(s_1) \in \left[ \bigtimes_{i=1}^N B_{i,1} \right]\!,
	\V(s_2) \in \left[ \bigtimes_{i=1}^N B_{i,2} \right]\!,
	\ldots, 
	\V(s_n) \in \left[ \bigtimes_{i=1}^N B_{i,n} \right]
	\right)\!. 
	\end{split} 
	\end{equation} 
	Klenke~\cite[Lemma 1.42]{Klenke2014} hence ensures that $\U$ and $\V$ are identically distributed random fields. 
	This completes the proof of \cref{lem:equidistributed_vectors}. 
\end{proof}

\begin{cor} \label{cor:equidistributed_sums}
	Let $ K,N \in \N $, 
	let $ ( \Omega, \mc F, \P ) $ be a probability space, 
	let $ ( S, \mc S ) $ be a measurable space, 
	let $ \mc G_i \subseteq \mc F $, $i\in\{1,2,\ldots,N\}$, be independent sigma-algebras on $\Omega$, 
	let $ \mc A_i \subseteq \mc F $, $i\in\{1,2,\ldots,N\}$, be independent sigma-algebras on $\Omega$,
	let $ U_i \colon S \times \Omega \to \R^K $, $ i \in \{1,2,\ldots,N\} $, satisfy for all 
	$ i \in \{1,2,\ldots,N\} $ 
	that $ U_i $ is $ (\mc S \otimes \mc G_i) $/$ \Borel(\R^K) $-measurable, 
	let $ V_i \colon S \times \Omega \to \R^K $, $ i \in \{1,2,\ldots,N\} $, satisfy for all 
	$ i \in \{1,2,\ldots,N\} $ 
	that $ V_i $ is $ (\mc S \otimes \mc A_i) $/$ \Borel(\R^K) $-measurable, 
	and assume for all
	$ i \in \{1,2,\ldots,N\} $ 
	that $ U_i $ and $ V_i $ are identically distributed random fields. 
	Then it holds that 
	$ S\times\Omega\ni(s,\omega)\mapsto\sum_{i=1}^N U_i(s,\omega) \in\R^K $
	and 
	$ S\times\Omega\ni(s,\omega)\mapsto\sum_{i=1}^N V_i(s,\omega) \in\R^K $
	are identically distributed random fields. 
\end{cor}

\begin{proof}[Proof of \cref{cor:equidistributed_sums}]
	Observe that \cref{lem:equidistributed_vectors} 
	(applied with 
		$ (E,\mc E) \is (\R^K,\Borel(\R^K)) $
	in the notation of \cref{lem:equidistributed_vectors}) proves that $ (U_1,U_2,\ldots,U_N) \colon S \times \Omega \to (\R^K)^N $ and $ (V_1,V_2,\ldots,V_N) \colon S \times \Omega \to (\R^K)^N $ are identically distributed random fields. 
	Item~\eqref{equidistribution_random_fields:item2} of \cref{lem:equidistribution_random_fields} 
	hence implies that 
	$ S\times\Omega\ni(s,\omega)\mapsto\sum_{i=1}^N U_i(s,\omega) \in\R^K $ and $ S\times\Omega\ni(s,\omega)\mapsto\sum_{i=1}^N V_i(s,\omega) \in\R^K $ 
	are identically distributed random fields. 
	This completes the proof of \cref{cor:equidistributed_sums}. 
\end{proof}

\subsection{Distribution properties for MLP approximations}
\label{subsec:distribution_properties}

\begin{lemma} \label{lem:elementary_equidistribution_X_processes} 
	Assume \cref{setting}. Then 
		\begin{enumerate}[(i)]
		\item \label{elementary_equidistribution_X_processes:item1}
		it holds for all 
			$ \theta \in \Theta $
		that 
			$ \{ (m,n) \in \N_0\times\N_0 \colon n \leq m \leq K \} \times \mc O \times \Omega \ni ( k, l, x, \omega ) \mapsto X^{\theta,k,x}_l(\omega) \in \mc O  $
		and 
			$ \{ (m,n) \in \N_0\times\N_0 \colon n \leq m \leq K \} \times \mc O \times \Omega \ni ( k, l, x, \omega ) \mapsto X^{0,k,x}_l(\omega) \in \mc O  $
		are identically distributed random fields and 
		\item \label{elementary_equidistribution_X_processes:item2}
		it holds for all 
			$ \theta,\vartheta \in \Theta $, 
			$ k,l,m \in \N_0 $
		with 
			$ m \leq l \leq k \leq K $ 
		that 
			$ \mc O \times \Omega \ni ( x, \omega ) \mapsto X^{ \vartheta, l, X^{ \theta, k, x }_l(\omega) }_m(\omega) \in \mc O $
		and 
			$ \mc O \times \Omega \ni ( x, \omega ) \mapsto X^{0,k,x}_m( \omega ) \in \mc O $
		are identically distributed random fields.  
		\end{enumerate} 
\end{lemma}

\begin{proof}[Proof of \cref{lem:elementary_equidistribution_X_processes}]
	Throughout this proof let 
		$ F^{ k }_{ l } \colon S^{ K } \times \mc O \to \mc O $, $ l \in \{ 0, 1, \ldots, k \} $, $ k \in \{ 0, 1, \ldots, \allowbreak K \} $, 
	satisfy for all 
		$ k \in \{ 0, 1, \ldots, K \} $, 
		$ l \in \{ 0, 1, \ldots, k \} $, 
		$ s = ( s_{ 1 }, s_{ 2 }, \ldots, s_{ K } ) \in S^{ K } $, 
		$ x \in \mc O $ 
	that 
		\begin{equation} \label{elementary_equidistribution_X_processes:definition_functionals}
		F^{ k }_{ l }( s, x ) = 
		\begin{cases} 
			x & \colon l = k \\
			\phi_{ l }( F^{ k }_{ l + 1 }( s, x ), s_{ l + 1 } ) & \colon l < k. 
		\end{cases} 
		\end{equation} 
	Observe that induction shows that for all 
		$ k \in \{ 0, 1, \ldots, K \} $, 
		$ l \in \{ 0, 1, \ldots, k \} $ 
	it holds that 
		$ F^{ k }_{ l } \colon S^{ K } \times \mc O \to \mc O $ 
	is 
		$ ( \mc S^{ \otimes K } \otimes \Borel( \mc O ) ) $/$ \Borel( \mc O ) $-measurable. 
	Moreover, note that induction, \eqref{elementary_equidistribution_X_processes:definition_functionals}, and \eqref{setting:dynamics} ensure that for all 
		$ k \in \{ 0, 1, \ldots, K \} $, 
		$ l \in \{ 0, 1, \ldots, k \} $, 
		$ \theta \in \Theta $, 
		$ x \in \mc O $ 
	it holds that 
		\begin{equation} \label{elementary_equidistribution_X_processes:functional_dependence}
		X^{ \theta, k, x }_{ l } = F^{ k }_{ l }( W^{ \theta }_{ 0 }, \ldots, W^{ \theta }_{ K - 1 }, x ) . 
		\end{equation} 
	Next observe that the assumption that  
		$ W^{ \theta }_{ k } \colon \Omega \to S $, $ \theta \in \Theta $, $ k \in \{ 0, 1, \ldots, K - 1 \} $, 
	are independent random variables and the assumption that for all 
		$ k \in \{ 0, 1, \ldots, K-1 \} $
	it holds that 
		$ W^{ \theta }_{ k } \colon \Omega \to S $, $ \theta \in \Theta $, 
	are identically distributed implies that for all 
		$ \theta \in \Theta $
	it holds that 
		$ ( W^{ \theta }_{ l } )_{ l \in \{ 0, 1, \ldots, K-1 \} } $ 
	and 
		$ ( W^{ 0 }_{ l } )_{ l \in \{ 0, 1, \ldots, K-1 \} } $ 
	are identically distributed. 
	This and \eqref{elementary_equidistribution_X_processes:functional_dependence} establish item~\eqref{elementary_equidistribution_X_processes:item1}. 
	Next we prove item~\eqref{elementary_equidistribution_X_processes:item2}. 
	Note that for all 
		$ \theta, \vartheta \in \Theta $, 
		$ k,l,m \in \N_0 $ 
	with 
		$ m \leq l \leq k \leq K $ 
	it holds that 
		$ \sigma( ( W^{ \vartheta }_{ s } )_{ s \in [ m, l ) \cap \N_0 }) $
	and 
		$ \sigma( ( W^{ \theta }_{ s } )_{ s \in [ l, k ) \cap \N_0 } ) $ 
	are independent. 
	This, item~\eqref{elementary_equidistribution_X_processes:item1}, item~\eqref{elementary_measurability_X_processes:item1} of \cref{lem:elementary_measurability_X_processes}, and item~\eqref{equidistribution_random_fields:item2} of
	\cref{lem:equidistribution_random_fields} demonstrate that for all 
		$ \theta,\vartheta \in \Theta $, 
		$ k,l,m \in \N_0 $ 
	with 
		$ m \leq l \leq k \leq K $
	it holds that 
		\begin{equation} \label{elementary_equidistribution_X_processes:statement_on_identical_distributions}
		\mc O \times \Omega \ni (x, \omega ) \mapsto X^{ \vartheta, l, X^{ \theta, k, x }_{ l } ( \omega ) }_{ m } ( \omega ) \in \mc O \qandq 
		\mc O \times \Omega \ni ( x, \omega ) \mapsto  X^{ 0, l, X^{ 0, k, x }_{ l } ( \omega ) }_{ m } ( \omega ) \in \mc O 
		\end{equation} 
	are identically distributed random fields. 
	Furthermore, observe that \eqref{setting:dynamics} and induction imply that for all 
		$ k,l,m \in \N_0 $, 
		$ x \in \mc O $, 
		$ \omega \in \Omega $ 
	with 
		$ m \leq l \leq k \leq K $
	it holds that 
		\begin{equation} 
		X^{ 0, l, X^{ 0, k, x }_{ l } ( \omega ) }_{ m }( \omega ) = X^{ 0, k, x }_{ m } ( \omega ).  
		\end{equation} 
	Combining this with \eqref{elementary_equidistribution_X_processes:statement_on_identical_distributions} establishes item~\eqref{elementary_equidistribution_X_processes:item2}. 
	This completes the proof of  \cref{lem:elementary_equidistribution_X_processes}. 
\end{proof}

\begin{prop} \label{lem:equidistribution} 
	Assume \cref{setting}. Then 
	\begin{enumerate} [(i)]
		\item \label{equidistribution:item1} 
		it holds for all 
			$ n \in \N_0 $, 
			$ \theta \in \Theta $  
		that 
			$ \{ 0, 1, \ldots, K \} \times \mc O \times \Omega \ni (k,x,\omega) \mapsto V^{\theta}_{k,n}( x, \omega ) \in \R $
		and 
			$ \{ 0, 1, \ldots, K \} \times \mc O \times \Omega \ni (k,x,\omega) \mapsto V^{0}_{k,n}( x, \omega ) \in \R $
		are identically distributed random fields, 
		\item \label{equidistribution:item2}
		it holds for all 
			$ n \in \N $,  
			$ \theta,\vartheta \in \Theta $
		with 
			$ \{(\theta,\eta)\in\Theta \colon \eta\in\Theta \} \cap \{ (\vartheta,\eta)\in\Theta \colon \eta\in\Theta \} = \emptyset $
		that 
			$ \{ 0, 1, \ldots, K \} \times \mc O \times \Omega \ni ( k, x, \omega ) \mapsto ( V^{\theta}_{k,n}( x, \omega ), V^{\vartheta}_{k,n-1}( x, \omega ) ) \in \R^2 $
		and 
			$ \{ 0, 1, \ldots, K \} \times \mc O \times \Omega \ni ( k, x, \omega ) \mapsto (V^{0}_{k,n}(x,\omega), V^{1}_{k,n-1}( x, \omega ) ) \in \R^2 $
		are identically distributed random fields, and 
		\item \label{equidistribution:item3}
		it holds for all 
			$ n \in \N $, 
			$ \theta,\vartheta \in \Theta $
		with 
			$ \{ (\theta,\eta)\in\Theta \colon \eta\in\Theta \} \cap \{ (\vartheta,\eta)\in\Theta \colon \eta\in\Theta \} = \emptyset $
		that
			\begin{multline}
			\{ 1, 2, \ldots, K \} \times \mc O \times \Omega \ni ( k, x, \omega )    
			\mapsto \\
			\frac{ 1 }{ \mf p_{ k, \mc{R}^{ \theta }_{ k } ( \omega ) } } \Big[ 
			f_{ \mc{R}^{ \theta }_k( \omega ) }( X^{ \theta, k, x }_{ \mc{R}^{ \theta }_k( \omega )}( \omega ), V^{ \theta }_{ \mc{R}^{ \theta }_k( \omega ), n }( X^{ \theta, k, x }_{ \mc{R}^{ \theta }_k( \omega )}( \omega ), \omega ) ) 
			- 
			V^{\theta}_{ \mc{R}^{ \theta }_k( \omega ), n}( X^{ \theta, k, x }_{ \mc{R}^{\theta}_k( \omega ) }( \omega ), \omega ) 
			\\
			- 
			f_{\mc{R}^{\theta}_k(\omega)}(X^{ \theta, k, x }_{\mc{R}^{\theta}_k(\omega)}( \omega ), V^{\vartheta}_{\mc{R}^{\theta}_k(\omega),n-1}( X^{ \theta, k, x }_{ \mc{R}^{\theta}_k(\omega) }(\omega),\omega))
			+ 
			V^{\vartheta}_{ \mc{R}^{\theta}_k( \omega ), n-1 }( X^{ \theta, k, x }_{ \mc{R}^{\theta}_k(\omega) }( \omega ), \omega ) \Big]
			\in \R 
			\end{multline} 
		and 
			\begin{multline}
			\{ 1, 2, \ldots, K \} \times \mc O \times \Omega \ni ( k, x, \omega ) \mapsto \\
			\frac{ 1 }{ \mf p_{ k, \mc{R}^{ 0 }_{ k }( \omega ) } } \Big[ f_{ \mc{R}^{0}_k( \omega ) }( X^{ 0, k, x }_{\mc{R}^{ 0 }_{ k } ( \omega ) }( \omega ), V^{ 0 }_{ \mc{R}^{ 0 }_{ k } ( \omega ), n } ( X^{ 0, k, x }_{ \mc{R}^{ 0 }_{ k } ( \omega ) }( \omega ), \omega ) )
			- 
			V^{ 0 }_{ \mc{R}^{ 0 }_{ k } ( \omega ), n } ( X^{ 0, k, x }_{ \mc{R}^{ 0 }_{ k } ( \omega ) }( \omega ), \omega ) 
			\\
			- 
			f_{ \mc{R}^{ 0 }_{ k } ( \omega ) } ( X^{ 0, k, x }_{ \mc{R}^{ 0 }_{ k } ( \omega ) }( \omega ), V^{ 1 }_{\mc{R}^{ 0 }_{ k } ( \omega ), n-1 }( X^{ 0, k, x }_{ \mc{R}^{ 0 }_{ k } ( \omega ) } ( \omega ), \omega ) )
			+ 
			V^{ 1 }_{ \mc{R}^{ 0 }_{ k } ( \omega ), n-1 } ( X^{ k, 0 }_{ \mc{R}^{ 0 }_{ k } ( \omega ) } ( x, \omega ), \omega ) \Big] \in \R 
			\end{multline} 
		are identically distributed random fields. 
	\end{enumerate}
\end{prop}

\begin{proof}[Proof of \cref{lem:equidistribution}]
	Throughout this proof let 
		$ h_k \colon \mc O \times \R \to \R $, $ k \in \{ 0, 1, \ldots, K - 1 \} $,
	satisfy for all 
		$ k \in \{0,1,\ldots,K-1\} $, 
		$ x \in \mc O $, 
		$ a \in \R $  
	that 
		$ h_k( x, a ) = f_k( x, a ) - a $, 
	let $ \mc A^{I}_{J} \subseteq \mc F $, $ I,J \in \mf 2^{\Theta} $, satisfy for all 
		$ I, J \in \mf 2^{\Theta} $
	that 
		$ \mc A^{I}_{J} = \sigma( ( \mc{R}^{\theta}_{s} )_{ (s, \theta) \in \{1, 2, \ldots, K\} \times I }, ( W^{ \theta }_s )_{ ( s, \theta ) \in \{ 0, 1, \ldots, K-1 \} \times J } ) $, 
	let $ \mc X \subseteq \N_0 $ satisfy that
		\begin{equation}
		\mc X = 
		\left\{ n \in \N_0 
		\colon 
		\left( 
		\begin{array}{c} 
		\text{It holds for all}~\theta \in \Theta~\text{that} \\
		\{ 0, 1, \ldots, K \} \times \mc O \times \Omega \ni ( k, x, \omega ) \mapsto V^{ \theta }_{ k, n }( x, \omega ) \in \R
		~\text{and}~\\
		\{ 0, 1, \ldots, K \} \times \mc O \times \Omega \ni ( k, x, \omega ) \mapsto V^{ 0 }_{ k, n }( x, \omega ) \in \R \\
		\text{are identically distributed random fields.}
		\end{array}
		\right)
		\right\}\!, 
		\end{equation} 
	let $ \mc Y \subseteq \N $ satisfy that 
		\begin{equation} 
		\mc Y = \left\{ n \in \N \colon 
		\left( 
		\begin{array}{c}
		\text{It holds for all}~
			\theta,\vartheta \in \Theta
		~\text{with}~\\
			\{ ( \theta, \eta ) \in \Theta \colon \eta \in \Theta \} \cap \{ (\vartheta,\eta)\in\Theta \colon \eta\in\Theta \} = \emptyset \\
		\text{that}~ 
			\{ 0, 1, \ldots, K \} \times \mc O \times \Omega \ni ( k, x, \omega )	\mapsto \\
		 ( V^{\theta}_{k,n}(x,\omega),V^{\vartheta}_{k,n-1}( x, \omega ) ) \in \R^2~\text{and} \\
			\{ 0, 1, \ldots, K \} \times \mc O \times \Omega \ni ( k, x, \omega ) \mapsto 
		 ( V^{0}_{k,n}( x, \omega ), V^{ 1 }_{ k, n-1 }( x, \omega ) ) \in \R^2 \\
		~\text{are identically distributed random fields.}
		\end{array}
		\right)			
		\right\}\!, 
		\end{equation} 
	and let $ \mc Z \subseteq \N $ satisfy that
		\begin{equation}  
		\mc Z = 
		\left\{ 
		n\in\N\colon
		\left( 
		\begin{array}{c}
		\text{It holds for all}~\theta,\vartheta\in\Theta~\text{with}~\\ 
		\{(\theta,\eta)\in\Theta \colon \eta\in\Theta \} \cap \{ (\vartheta,\eta)\in\Theta \colon \eta\in\Theta \} = \emptyset \\
		~\text{that}~
		\{1,2,\ldots,K\} \times \mc O \times \Omega \ni (k,x,\omega) 
		\mapsto \\
		\frac{ 1 }{ \mf p_{ k, \mc{R}^{ \theta }_{ k }( \omega ) } } \big[ h_{ \mc{R}^{ \theta }_k( \omega ) }( X^{ \theta, k, x }_{ \mc{R}^{ \theta }_k( \omega ) }( \omega ), V^{ \theta }_{ \mc{R}^{ \theta }_k( \omega ),n}( X^{ \theta, k, x }_{ \mc{R}^{ \theta }_k( \omega ) }( \omega ), \omega ) )
		\\
		- 
		h_{ \mc{R}^{ \theta }_k( \omega ) }( X^{ \theta, k, x }_{ \mc{R}^{ \theta }_k( \omega ) }( \omega ), V^{ \vartheta }_{ \mc{R}^{ \theta }_k( \omega ), n-1 }( X^{ \theta, k, x }_{ \mc{R}^{ \theta }_k( \omega ) }( \omega ), \omega ) ) \big]
		\in \R 
		\\~\text{and}~
		\{0,1,\ldots,K\} \times \mc O \times \Omega \ni (k,x,\omega) 
		\mapsto
		\\
		\frac{ 1 }{ \mf p_{ k, \mc{R}^{ 0 }_{ k }( \omega ) } } \big[ h_{ \mc{R}^{ 0 }_{ k }( \omega ) }( X^{ 0, k, x }_{ \mc{R}^{ 0 }_{ k }( \omega ) }( \omega ), V^{ 0 }_{ \mc{R}^{ 0 }_{ k }( \omega ), n }( X^{ 0, k, x }_{ \mc{R}^{ 0 }_{ k }( \omega ) }( \omega ), \omega ) ) 
		\\
		- 
		h_{ \mc{R}^{ 0 }_{ k } ( \omega ) }( X^{ 0, k, x }_{ \mc{R}^{ 0 }_{ k }( \omega ) }( \omega ), V^{ 1 }_{ \mc{R}^{ 0 }_{ k }( \omega ), n-1 }( X^{ 0, k, x }_{ \mc{R}^{ 0 }_{ k }( \omega ) }( \omega ), \omega ) ) \in \R
		\\
		~\text{are identically distributed random fields.}
		\end{array} 
		\right) 
		\right\} \!. 
		\end{equation}
	Observe that the assumption that for all 
		$ \theta \in \Theta $, 
		$ k \in \{0,1,\ldots,K\} $, 
		$ x \in \mc O $
	it holds that 
		$ V^{\theta}_{k,0}(x) = 0 $ 
	ensures that 
		$ 0 \in \mc X $. 
	Next we prove that $ 1 \in \mc X $. 
	For this note that \eqref{setting:mlp_scheme} yields that for all 
		$ \theta \in \Theta $, 
		$ k \in \{0,1,\ldots,K\} $, 
		$ x \in \mc O $
	it holds that 
		\begin{equation} \label{equidistribution:n_is_one}
		V^{\theta}_{k,1}(x) = \frac{1}{M} \sum_{m=1}^M 
		\left[ g(X^{ (\theta,0,-m), k, x }_0 ) + \sum_{l=0}^{K-1} \mathbbm{1}_{ [0,k) }( l ) f_l( X^{ (\theta,0,m), k, x }_l, 0 ) 	\right] \!.
		\end{equation} 
	Moreover, observe that item~\eqref{elementary_equidistribution_X_processes:item1} of \cref{lem:elementary_equidistribution_X_processes} ensures that 
		\begin{enumerate}[(I)]
		\item it holds for all 
			$ \theta \in \Theta $, 
			$ m \in \N $ 
		that 
			$ \{ 0, 1, \ldots, K \} \times \mc O \times \Omega \ni ( k, x, \omega ) \mapsto g( X^{ (\theta,0,-m), k, x }_0( \omega ) ) \in \R $
		and 
			$ \{ 0, 1, \ldots, K \} \times \mc O \times \Omega \ni ( k, x, \omega ) \mapsto g( X^{ (0,0,-m), k, x }_0( \omega ) ) \in \R $			 
		are identically distributed random fields 
		and 
	\item it holds for all 
			$ \theta \in \Theta $, 
			$ m \in \N $ 
		that 
			$ \{ 1, 2, \ldots, K \} \times \mc O \times \Omega \ni ( k, x, \omega ) \mapsto \mathbbm{ 1 }_{ [ 0, k-1 ] } ( l ) f_{ l } ( X^{ ( \theta, 0, m ), k, x }_{ l } ( \omega ), \allowbreak 0 ) \in \R $
		and 
			$ \{ 1, 2, \ldots, K \} \times \mc O \times \Omega \ni ( k, x, \omega ) \mapsto \mathbbm{1}_{[0,k-1]}(l) f_l( X^{ (0,0,m), k, x }_l( \omega ), 0 ) \in \R $
		are identically distributed random fields. 
	\end{enumerate}
	Item~\eqref{elementary_measurability_X_processes:item1} of \cref{lem:elementary_measurability_X_processes},  
	\cref{cor:equidistributed_sums}, and the fact that $ \sigma( ( W^{\vartheta}_s )_{ s \in \{ 0, 1, \ldots, K-1 \} } ) $, $ \vartheta \in \Theta $, are independent sigma-algebras on $ \Omega $ therefore prove that for all 
		$ \theta \in \Theta $, 
		$ m \in \N $ 
	it holds that 
		\begin{multline} 
		\{ 1, 2, \ldots, K \} \times \mc O \times \Omega \ni ( k, x, \omega ) 
		\mapsto  
		g ( X^{ (\theta,0,-m), k, x }_0( \omega ) ) 
		\\
		+ \sum_{l=0}^{K-1} \mathbbm{1}_{[0,k-1]}(l) f_l( X^{ (\theta,0,m), k, x }_l( \omega ), 0 ) 
		\in \R 
		\end{multline}
	and 
		\begin{multline} 
		\{ 1, 2, \ldots, K \} \times \mc O \times \Omega \ni ( k, x, \omega ) 
		\mapsto  
		g ( X^{ (0,0,-m), k, x }_0( \omega ) ) 
		\\
		+
		\sum_{l=0}^{K-1} \mathbbm{1}_{[0,k-1]}(l) f_l( X^{ (0,0,m), k, x }_l( \omega ), 0  ) 
		\in \R 
		\end{multline}
	are identically distributed random fields. 
	Item~\eqref{elementary_measurability_X_processes:item1} of \cref{lem:elementary_measurability_X_processes},   
	\eqref{equidistribution:n_is_one}, 
	\cref{cor:equidistributed_sums}, and the fact that $ \sigma( ( W^{\vartheta}_s )_{ s \in \{ 0, 1, \ldots, K-1 \} } ) $, $ \vartheta \in \Theta $, are independent sigma-algebras on $ \Omega $
	hence ensure that for all 
		$ \theta \in \Theta $ 
	it holds that 
		$ \{ 0, 1, \ldots, K \} \times \mc O \times \Omega \ni ( k, x, \omega ) \mapsto V^{ \theta }_{ k, 1 }( x, \omega ) \in \R $ 
	and 
		$ \{ 0, 1, \ldots, K \} \times \mc O \times \Omega \ni ( k, x, \omega ) \mapsto V^{ 0 }_{ k, 1 }( x, \omega ) \in \R $ 
	are identically distributed random fields. 
	Therefore, we obtain that $ 1 \in \mc X $.   	
	Next we prove that $ \{ n \in \mc X \colon n-1 \in \mc X \} \subseteq \mc Y $. 
	Note that item~\eqref{measurability:item1} of \cref{lem:measurability} ensures that for all 
		$ n \in \N_0 $,
		$ \theta \in \Theta $
	it holds that 
		$ \{ 0, 1, \ldots, K \} \times \mc O \times \Omega \ni ( k, x, \omega ) \mapsto V^{\theta}_{ k, n }( x, \omega ) \in \R $ 
	is 
		$ ( \mf 2^{ \{ 0, 1, \ldots, K \} } \otimes \Borel( \mc O ) \otimes \mc A^{ \{ ( \theta, \eta ) \in \Theta \colon \eta \in \Theta \} }_{ \{ ( \theta, \eta ) \in \Theta \colon \eta \in \Theta \} } ) $/$ \Borel( \R ) $-measurable. 
	The fact that for all 
		$ \theta, \vartheta \in \Theta $ 
	with 
		$ \{ ( \theta, \eta ) \colon \eta \in \Theta \} \cap \{ ( \vartheta, \eta ) \colon \eta \in \Theta \} = \emptyset $ 
	it holds that 
		$ \mc A^{  \{ ( \theta, \eta ) \colon \eta \in \Theta \} }_{  \{ ( \theta, \eta ) \colon \eta \in \Theta \} } $ 
	and 
		$ \mc A^{ \{ ( \vartheta, \eta ) \colon \eta \in \Theta \} }_{ \{ ( \vartheta, \eta ) \colon \eta \in \Theta \} } $ 
	are independent and \cref{lem:equidistributed_vectors} 
	hence demonstrate that for all 
		$ n \in \mc X $, 
		$ \theta, \vartheta \in \Theta $ 
	with 
		$ \{ ( \theta, \eta ) \in \Theta \colon \eta \in \Theta \} \cap \{ ( \vartheta, \eta ) \in \Theta \colon \eta \in \Theta \} = \emptyset $ 
	and 
		$ n-1 \in \mc X $
	it holds that 
		$ \{ 0, 1, \ldots, K \} \times \mc O \times \Omega \ni ( k, x, \omega ) \mapsto ( V^{ \theta }_{ k, n }( x, \omega ), V^{ \vartheta }_{ k, n-1 }( x, \omega ) ) \in \R^2 $ 
	and 
		$ \{ 0, 1, \ldots, K \} \times \mc O \times \Omega \ni ( k, x, \omega ) \mapsto ( V^{ 0 }_{ k, n }( x, \omega ), V^{ 1 }_{ k, n-1 }( x, \omega ) ) \in \R^2 $ 
	are identically distributed random fields. 
	Therefore, we obtain that $ \{ n\in\mc X \colon n-1\in\mc X \} \subseteq \mc Y $. 
	Combining this with the fact that $ \{ 0, 1 \} \subseteq \mc X $ ensures that $ 1 \in \mc Y $. 
	Next we prove that $ \mc Y \subseteq \mc Z $. 
	For this note that item~\eqref{elementary_equidistribution_X_processes:item1} of \cref{lem:elementary_equidistribution_X_processes} and  \cref{lem:equidistribution_random_fields} 
	prove that for all 
		$ \theta \in \Theta $ 
	it holds that 
		$ \{ 1, 2, \ldots, K \} \times \mc O \times \Omega \ni ( k, x, \omega ) \mapsto ( \mc{R}^{ \theta }_k( \omega ), X^{ \theta, k, x }_{ \mc{R}^{\theta}_k( \omega ) }( \omega ) ) \in \{ 0, 1, \ldots, K-1 \} \times \mc O $ 
	and 
		$ \{ 1, 2, \ldots, K \} \times \mc O \times \Omega \ni ( k, x, \omega ) \mapsto ( \mc{R}^{ 0 }_k( \omega ), X^{ 0, k, x }_{\mc{R}^{ 0 }_k( \omega ) }( \omega ) ) \in \{ 0, 1, \ldots, K-1 \} \times \mc O $ 
	are identically distributed random fields. 
	Item~\eqref{elementary_measurability_X_processes:item1} of \cref{lem:elementary_measurability_X_processes}, 
	item~\eqref{measurability:item1} of \cref{lem:measurability}, \cref{lem:equidistribution_random_fields}, and the fact that for all 
		$ \theta, \vartheta \in \Theta $ 
	with 
		$ \theta \notin \{ ( \vartheta, \eta ) \in \Theta \colon \eta \in \Theta \} $ 
	it holds that $ \sigma( (\mc{R}^{ \theta }_{s})_{s\in\{1, 2, \ldots, K\}}, ( W^{ \theta }_s )_{ s \in \{ 0, 1, \ldots, K-1 \} } ) $ and $ \sigma( ( \mc{R}^{ ( \vartheta, \eta ) }_{s} )_{ (s, \eta) \in \{1, 2, \ldots, K\} \times \Theta } , ( W^{ ( \vartheta, \eta ) }_s )_{ (s,\eta) \in \{ 0, 1, \ldots, K-1 \}\times\Theta } ) $ are independent
	therefore demonstrate that for all 
		$ n \in \mc Y $, 
		$ \theta,\vartheta \in \Theta $ 
	with 
		$ \{(\theta,\eta)\in\Theta\colon \eta\in\Theta\}\cap\{(\vartheta,\eta)\in\Theta\colon \eta\in\Theta\} = \emptyset $ 
	it holds that 
		\begin{multline} 
		\{ 1, 2, \ldots, K \} \times \mc O \times \Omega \ni ( k, x, \omega ) \mapsto 
		\frac{ 1 }{ \mf p_{ k, \mc{R}^{ \theta }_{ k }( \omega ) } } 
		\Big[ 
		h_{ \mc{R}^{ \theta }_k( \omega ) }( X^{ \theta, k, x }_{ \mc{R}^{ \theta }_k( \omega ) }( \omega ), V^{ \theta }_{ \mc{R}^{ \theta }_k( \omega ), n }( X^{ \theta, k, x }_{ \mc{R}^{ \theta }_k( \omega ) }( \omega ), \omega ) ) 
		\\
		- h_{ \mc{R}^{ \theta }_k( \omega ) }( X^{ \theta, k, x }_{ \mc{R}^{ \theta }_k( \omega )}( \omega ), V^{ \vartheta }_{ \mc{R}^{ \theta }_k( \omega ), n-1 }( X^{ \theta, k, x }_{ \mc{R}^{ \theta }_k( \omega )}( \omega ), \omega ) ) \Big] \in \R 
		\end{multline}
	and 
		\begin{multline} 
		\{ 1, 2, \ldots, K \} \times \mc O \times \Omega \ni ( k, x, \omega ) \mapsto 
		\frac{ 1 }{ \mf p_{ k, \mc{R}^{ 0 }_{ k }( \omega ) } } \Big[ h_{ \mc{R}^{ 0 }_k( \omega ) }( X^{ 0, k, x }_{ \mc{R}^{ 0 }_k( \omega ) }( \omega ), V^{ 0 }_{ \mc{R}^{ 0 }_k( \omega ), n }( X^{ 0, k, x }_{ \mc{R}^{ 0 }_k( \omega ) }( \omega ), \omega ) ) 
		\\
		- h_{ \mc{R}^{ 0 }_k( \omega ) }( X^{ 0, k, x }_{ \mc{R}^{ 0 }_k( \omega ) }( \omega ), V^{1}_{ \mc{R}^{ \theta }_k( \omega ), n-1 }( X^{ 0, k, x }_{ \mc{R}^{ 0 }_k( \omega ) }( \omega ), \omega ) ) \Big] \in \R 
		\end{multline} 
	are identically distributed random fields. Hence, we obtain that $ \mc Y \subseteq \mc Z $. Combining this with the fact that 
		$ 1 \in \mc Y $ 
	demonstrates that 
		$ 1 \in \mc Z $. 
	Next we prove that  
		$ \{ n \in \N \cap [2,\infty) \colon \{1,2,\ldots,n-1\} \subseteq \mc Z \} \subseteq \mc X $. 
	Note that 
	item~\eqref{measurability:item2} of \cref{lem:measurability} demonstrates that for all 
		$ n \in \N $, 
		$ \theta,\vartheta \in \Theta $ 
	it holds that 
		\begin{equation}  
		\begin{gathered}
		\{ 1, 2, \ldots, K \} \times \mc O \times \Omega \ni ( k, x, \omega ) \mapsto 
		\frac{ 1 }{ \mf p_{ k, \mc{R}^{ \theta }_{ k } ( \omega )  } } \Big[ 
		h_{ \mc{R}^{ \theta }_k( \omega ) }( X^{ \theta, k, x }_{ \mc{R}^{ \theta }_k( \omega ) }( \omega ), V^{ \theta }_{ \mc{R}^{ \theta }_k( \omega ), n }( X^{ \theta, k, x }_{ \mc{R}^{ \theta }_k( \omega ) }( \omega ), \omega ) ) 
		\\
		- 
		h_{ \mc{R}^{ \theta }_k( \omega ) }( X^{ \theta, k, x }_{ \mc{R}^{ \theta }_k( \omega ) }( \omega ), V^{ \vartheta }_{ \mc{R}^{ \theta }_k( \omega ), n-1 }( X^{ \theta, k, x }_{ \mc{R}^{ \theta }_k( \omega ) }( \omega ), \omega ) ) \Big] \in \R
		\end{gathered}
		\end{equation}
	is $ ( ( \mf 2^{ \{ 0, 1, \ldots, K \} } \otimes \Borel( \mc O ) ) \otimes
		\mc A^{ \cup_{ \eta \in \Theta } \{ \theta, ( \theta, \eta ), ( \vartheta, \eta ) \} }_{ \cup_{ \eta \in \Theta } \{ \theta, ( \theta, \eta ), ( \vartheta, \eta ) \} } ) $/$ \Borel( \R ) $-measurable. 
	Combining this and the fact that for all 
		$ \theta \in \Theta $, 
		$ j \in \N $ 
	it holds that 
		$ \mc A^{ \cup_{ \eta \in \Theta } \{ ( \theta, j, m ), ( \theta, j, m, \eta ), ( \theta, j, -m, \eta ) \} }_{ \cup_{ \eta \in \Theta } \{ ( \theta, j, m ), ( \theta, j, m, \eta ), ( \theta, j, -m, \eta ) \} } $, $ m \in \N $, 
	are independent sigma-algebras on $ \Omega $ with \cref{cor:equidistributed_sums} proves that for all 
		$ n \in \{ k\in [2,\infty) \cap \N \colon \{1,2,\ldots,k-1\}\subseteq\mc Z\} $, 
		$ j \in \{1,2,\ldots,n-1\} $,
		$ \theta \in \Theta $ 
	it holds that 
		\begin{multline}
		\{ 1, 2, \ldots, K \} \times \mc O \times \Omega \ni ( k, x, \omega ) \mapsto 
		\\
		\sum_{ m = 1 }^{ M^{ n - j } } \frac{ 1 }{ \mf p_{ k, \mc{R}^{ ( \theta, j, m ) }_{ k }( \omega ) } } \bigg[ h_{ \mc{R}^{ ( \theta, j, m ) }_k( \omega ) }( X^{ ( \theta, j, m ), k, x }_{ \mc{R}^{ ( \theta, j , m ) }_k( \omega ) }( \omega ), V^{ ( \theta, j, m ) }_{ \mc{R}^{ ( \theta, j, m ) }_k( \omega ), j }( X^{ ( \theta, j, m ), k, x }_{ \mc{R}^{ ( \theta, j, m ) }_k( \omega ) }( \omega ), \omega ) ) 
		\\
		- h_{ \mc{R}^{ ( \theta, j, m ) }_k( \omega ) }( X^{ ( \theta, j, m ), k, x }_{ \mc{R}^{ ( \theta, j, m ) }_k( \omega ) }( \omega ), V^{ ( \theta, j, -m ) }_{ \mc{R}^{ ( \theta, j, m ) }_k( \omega ), j-1 }( X^{ (\theta,j,m), k, x }_{ \mc{R}^{ ( \theta, j, m ) }_k( \omega ) }( \omega ), \omega ) ) \bigg] \in \R 
		\end{multline} 
	and 
		\begin{multline} 
		\{ 1, 2, \ldots, K \} \times \mc O \times \Omega \ni ( k, x, \omega ) \mapsto 
		\\
		\sum_{ m = 1 }^{ M^{ n - j } } \frac{ 1 }{ \mf p_{ k, \mc{R}^{ ( 0, j, m ) }_{ k } ( \omega ) } }\bigg[ h_{ \mc{R}^{ (0,j,m) }_k( \omega ) }( X^{ (0,j,m), k, x }_{ \mc{R}^{ (0, j, m ) }_k( \omega ) }( \omega ), V^{ (0,j,m) }_{ \mc{R}^{ (0,j,m) }_k( \omega ), j }( X^{ (0,j,m), k, x }_{ \mc{R}^{ (0,j,m) }_k( \omega ) }( \omega ), \omega ) ) \\
		- h_{ \mc{R}^{ ( 0, j, m ) }_k( \omega ) }( X^{ ( 0, j, m ), k, x }_{ \mc{R}^{ ( 0, j, m ) }_k( \omega ) }( \omega ), V^{ ( 0, j, -m ) }_{ \mc{R}^{ ( 0, j, m ) }_k( \omega ), j - 1 }( X^{ ( 0, j, m ), k, x }_{ \mc{R}^{ ( 0, j, m ) }_k( \omega ) }( \omega ), \omega ) ) \bigg] \in \R 
		\end{multline}
	are identically distributed random fields. 
	This, \eqref{setting:mlp_scheme}, the fact that for all 
		$ n \in \N $, 
		$ \theta \in \Theta $ 
	it holds that 
		$ \{ 0, 1, \ldots, K \} \times \mc O \times \Omega \ni ( k, x, \omega ) \mapsto \frac{ 1 }{ M^n } \sum_{ m = 1 }^{ M^n }[ g( X^{ ( \theta, 0, -m ), k, x }_0( \omega ) ) + \sum_{ l = 0 }^{ k - 1 } f_l( X^{ ( \theta, 0, m ), k, x }_l( \omega ), 0 ) ] \in \R $ 
	and 
		$ \{ 0, 1, \ldots, K \} \times \mc O \times \Omega \ni ( k, x, \omega ) \mapsto \frac{ 1 }{ M^n } \sum_{ m = 1 }^{ M^n }[ g( X^{ ( 0, 0, -m ), k, x }( \omega ) ) + \sum_{ l = 0 }^{ k - 1 } f_l( X^{ ( 0, 0, m ), k, x }( \omega ), 0 ) ] \in \R $ 
	are identically distributed random fields, 
	item~\eqref{elementary_measurability_X_processes:item1} of \cref{lem:elementary_measurability_X_processes}, 
	item~\eqref{measurability:item1} of \cref{lem:measurability}, 
	and \cref{cor:equidistributed_sums} demonstrate that for all 
		$ n \in \{ k\in [2,\infty)\cap\N\colon \{1,\ldots,k-1\}\subseteq \mc Z\} $, 
		$ \theta \in \Theta $ 
	it holds that 
		$ \{ 0, 1, \ldots, K \} \times \mc O \times \Omega \ni ( k, x, \omega ) \mapsto V^{ \theta }_{ k, n }( x, \omega ) \in \R $ 
	and 
		$ \{ 0, 1, \ldots, K \} \times \mc O \times \Omega \ni ( k, x, \omega ) \mapsto V^{ 0 }_{ k, n }( x, \omega ) \in \R $ 
	are identically distributed random fields. 
	Hence, we obtain that 
		$ \{ n \in [2,\infty) \cap \N \colon \{1,2,\ldots,n-1\} \subseteq \mc Z \} \subseteq \mc X $. 
	This, 
	the fact that $ \{ 0,1 \} \subseteq \mc X $, 
	the fact that $ \{ n \in \mc X\colon n-1\in\mc X\} \subseteq \mc Y $, 
	the fact that $ \mc Y \subseteq \mc Z $, 
	and induction prove that 
		$ \mc X = \N_0 $, 
		$ \mc Y = \N $, 
	and 
		$ \mc Z = \N $. 
	Note that the fact that $ \mc X = \N_0 $ establishes item~\eqref{equidistribution:item1}.   
	Moreover, observe that the fact that $ \mc Y = \N $ establishes item~\eqref{equidistribution:item2}. 
	Furthermore, note that the fact that $ \mc Z = \N $ establishes item~\eqref{equidistribution:item3}. 
	This completes the proof of \cref{lem:equidistribution}. 
\end{proof}

\begin{cor} \label{cor:equidistribution_for_mean_lemma}
	Assume \cref{setting}. 
	Then it holds for all 
		$ k,l,n \in \N_0 $, 
		$ \theta,\vartheta \in \Theta $ 
	with 
		$ l \leq k \leq K $ 
	and 
		$ \theta \notin \{ (\vartheta,\eta) \in \Theta \colon \eta \in \Theta \} $ 
	that 
		$ \mc O \times \Omega \ni ( x, \omega ) \mapsto ( X^{ \theta, k, x }_l( \omega ), V^{ \vartheta }_{ l, n }( X^{ \theta, k, x }_l( \omega ), \omega ) ) \in \mc O \times \R $ 
	and 
		$ \mc O \times \Omega \ni ( x, \omega ) \mapsto ( X^{ 0, k, x }_l( \omega ), V^0_{ l, n }( X^{ 0, k, x }_l( \omega ), \omega ) ) \in \mc O \times \R $ 
	are identically distributed random fields. 
\end{cor} 

\begin{proof}[Proof of \cref{cor:equidistribution_for_mean_lemma}]
	First, observe that item~\eqref{equidistribution:item1} of \cref{lem:equidistribution} demonstrates that for all 
		$ n \in \N_0 $, 
		$ l \in \{ 0, 1, \ldots, K \} $, 
		$ \vartheta \in \Theta $ 
	it holds that 
		$ \mc O \times \Omega \ni ( y, \omega ) \mapsto ( y, V^{ \vartheta }_{ l, n }( y, \omega ) ) \in \mc O \times \R $ 
	and 
		$ \mc O \times \Omega \ni ( y, \omega ) \mapsto ( y, V^{0}_{ l, n }( y, \omega ) ) \in \mc O \times \R $
	are identically distributed random fields. 
	Moreover, note that item~\eqref{elementary_equidistribution_X_processes:item1} of \cref{lem:elementary_equidistribution_X_processes} ensures that for all 
		$ k,l \in \N_0 $, 
		$ \theta \in \Theta $ 
	with 
		$ l \leq k \leq K $ 
	it holds that 
		$ \mc O \times \Omega \ni ( x, \omega ) \mapsto X^{ \theta, k, x }_l( \omega ) \in \mc O $ 
	and
		$ \mc O \times \Omega \ni ( x, \omega ) \mapsto X^{ 0, k, x }_l( \omega ) \in \mc O $ 
	are identically distributed random fields. 
	Combining the fact that for all 
		$ n \in \N_0 $, 
		$ l \in \{ 0, 1, \ldots, K \} $,
		$ \vartheta \in \Theta $ 
	it holds that 
		$ \mc O \times \Omega \ni ( y, \omega ) \mapsto ( y, V^{ \vartheta }_{ l, n }( y, \omega ) ) \in \mc O \times \R $ 
	and 
		$ \mc O \times \Omega \ni ( y, \omega ) \mapsto ( y, V^{ 0 }_{ l, n }( y, \omega ) ) \in \mc O \times \R $
	are identically distributed random fields, 
	item~\eqref{elementary_measurability_X_processes:item1} of \cref{lem:elementary_measurability_X_processes}, 
	item~\eqref{measurability:item1} of \cref{lem:measurability}, 
	and \cref{lem:equidistribution_random_fields} 
	with the fact that for all 
		$ \theta, \vartheta \in \Theta $ 
	with 
		$  \theta \notin \{ ( \vartheta, \eta ) \in \Theta \colon \eta \in \Theta \} $ 
	it holds that 
		$ \sigma( ( W^{ \theta }_s )_{ s\in \{ 0, 1, \ldots, K-1 \} } ) $
	and 
		$ \sigma( ( \mc{R}^{ ( \vartheta, \eta ) }_{ s } )_{ (s, \eta ) \in \{1, 2, \ldots, K\} \times \Theta }, ( W^{ ( \vartheta, \eta ) }_s )_{ ( s, \eta ) \in \{ 0, 1, \ldots, K-1 \} \times \Theta } ) $ 
	are independent 
	hence 
	proves that for all 
		$ k,l,n \in \N_0 $, 
		$ \theta,\vartheta \in \Theta $ 
	with 
		$ l \leq k \leq K $ 
	and 
		$ \theta \notin \{ (\vartheta,\eta) \in \Theta\colon \eta \in \Theta \} $ 
	it holds that 
		$ \mc O \times \Omega \ni ( x, \omega ) \mapsto ( X^{ \theta, k, x }_l( \omega ), V^{ \vartheta }_{ l, n }( X^{ \theta, k, x }_l( \omega ), \omega ) ) \in \mc O \times \R $ 
	and 
		$ \mc O \times \Omega \ni ( x, \omega ) \mapsto ( X^{ 0, k, x }_l( \omega ), V^0_{ l, n }( X^{ 0, k, x }_l( \omega ), \omega ) ) \in \mc O \times \R $ 
	are identically distributed random fields. 
	This completes the proof of \cref{cor:equidistribution_for_mean_lemma}. 		
\end{proof}

\subsection{Integrability properties for MLP approximations}
\label{subsec:integrability_properties}

\begin{lemma}[Integrability properties] \label{lem:integrability}
	Assume 
		\cref{setting}, 
	assume for all 
		$ k \in \{ 0, 1, \ldots, K \} $, 
		$ x \in \mc O $
	that 
		$ \EXP{ | g( X^{ 0, k, x }_{ 0 } ) |^2 + \sum_{ l = 0 }^{ k - 1 } | f_l( X^{ 0, k, x }_l ) |^2 } < \infty $, 
	and let
		$ L_0, L_1, \ldots, L_{K} \in \R $
	satisfy for all 
		$ k \in \{ 0, 1, \ldots, K - 1 \} $, 
		$ x \in \mc O $, 
		$ a,b \in \R $	
	that 
		$| ( f_k( x, a ) - a ) - ( f_k( x, b ) - b ) | \leq L_k | a - b | $. 
	Then it holds for all 
		$ k, l, n \in \N_0 $, 
		$ \theta, \vartheta \in \Theta $, 
		$ x \in \mc O $
	with 
		$ l \leq k \leq K $ 
	and 
		$ \vartheta \notin \{ ( \theta, \eta ) \in \Theta \colon \eta \in \Theta \} $
	that 
		$ \EXP{ | V^{ \theta }_{ l, n }( X^{ \vartheta, k, x }_{ l } ) |^2} < \infty $. 
\end{lemma}

\begin{proof}[Proof of \cref{lem:integrability}] 
	Throughout this proof let 
		$ \A \subseteq \N $ 
	satisfy that
		\begin{equation} 
		\A = \left\{ n \in \N \colon \left( 
		\begin{array}{c} \text{It holds for all}~j,k,l\in\N_0,\theta,\vartheta\in\Theta,x\in\mc O~\text{with}~ \\ 
		0\leq l\leq k\leq K, 
		\vartheta \notin \{(\theta,\eta) \in \Theta\colon \eta\in\Theta \},
		~\text{and}~0\leq j\leq n-1 \\ 
		\text{that}~\EXP{ | V^{ \theta }_{ l, j }( X^{ \vartheta, k, x }_l ) |^2 } < \infty. 
		\end{array} 
		\right) 
		\right\}\!.
		\end{equation} 
	Note that the assumption that for all 
		$ k \in \{ 0, 1, \ldots, K \} $, 
		$ \theta \in \Theta $, 
		$ x \in \mc O $
	it holds that 
		$ V^{ \theta }_{ k, 0 }( x ) = 0 $ 
	yields that 
		$ 1 \in \A $. 
	Next observe that \eqref{setting:mlp_scheme} ensures that for all 
		$ k,l,n \in \N $, 
		$ \theta,\vartheta \in \Theta $, 
		$ x \in \mc O $
	with 
		$ 
		l \leq k \leq K $
	it holds that 
		\begin{equation} 
		\begin{split}
		& V^{ \theta }_{ l, n } ( X^{ \vartheta, k, x }_{ l } ) 
		\\
		&= 
		\frac{ 1 }{ M^n } 
		\sum_{ m = 1 }^{ M^n }
		\left[ g\big( X^{ (\theta,0,-m), l, X^{ \vartheta, k, x }_{ l } }_{ 0 } \big) 
		+ \sum_{ s = 0 }^{ l - 1 } f_s\big( X^{ ( \theta, 0, m ), l, X^{ \vartheta, k, x }_{ l } }_s, 0 \big) 
		\right]
		+ \sum_{ j = 1 }^{ n - 1 } \frac{ 1 }{ M^{ n - j } } 
		\sum_{ m = 1 }^{ M^{ n - j } } 
		\frac{ 1 }{ \mf p_{ l, \mc{R}^{ ( \theta, j, m ) }_{ l } } }
		\\
		&  
		\cdot 
		\Bigg[
		\bigg( 
			f_{ \mc{R}^{ ( \theta, j, m ) }_{ l } } 
				\Big( X^{ ( \theta, j, m ), l, X^{ \vartheta, k, x }_{ l } }_{ \mc{R}^{ ( \theta, j, m ) }_{ l } },
					V^{ ( \theta, j, m ) }_{ \mc{R}^{ ( \theta, j, m ) }_{ l }, j }\big( X^{ ( \theta, j, m ), l, X^{ \vartheta, k, x }_{ l } }_{ \mc{R}^{ ( \theta, j, m ) }_{ l } } \big)
				\Big) -
		V^{ ( \theta, j, m ) }_{ \mc{R}^{ ( \theta, j, m ) }_{ l }, j }
			\big( X^{ ( \theta, j, m ), l, X^{ \vartheta, k, x }_{ l } }_{ \mc{R}^{ ( \theta, j, m ) }_{ l } }  \big) 
			\bigg) 
		\\
		&  
		- \bigg( 
		f_{ \mc{R}^{ ( \theta, j, m ) }_l }
		\Big( X^{ ( \theta, j, m ), l, X^{ \vartheta, k, x }_{ l } }_{ \mc{R}^{ ( \theta, j, m ) }_{ l } },  
		V^{ ( \theta, j, -m ) }_{ \mc{R}^{ ( \theta, j, m ) }_{ l }, j - 1 }\big( X^{ ( \theta, j, m ), l, X^{ \vartheta, k, x }_{ l } }_{ \mc{R}^{ ( \theta, j, m ) }_{ l } } \big)
		\Big)  
		- V^{ ( \theta, j, -m ) }_{ \mc{R}^{ ( \theta, j, m ) }_{ l }, j - 1 }
		 	\big( X^{ ( \theta, j, m ), l, X^{ \vartheta, k, x }_{ l } }_{ \mc{R}^{ ( \theta, j, m ) }_{ l } } \big) 
	 	\bigg) 
		\Bigg].   
		\end{split}
		\end{equation} 
	The assumption that for all 
		$ k \in \{0,1,\ldots,K-1\} $, 
		$ x \in \mc O $, 
		$ a,b \in \R $ 
	it holds that 
		$ | ( f_k( x, a ) - a ) - ( f_k( x, b ) - b ) | \leq L_k | a - b | $
	hence implies that for all 
		$ k,l,n \in \N $, 
		$ \theta,\vartheta \in \Theta $,
		$ x \in \mc O $
	with 
		$ 
		l \leq k \leq K $
	it holds that 
		\begin{equation} 
		\begin{split}
		 & 
		 \big| V^{ \theta }_{ l, n }( X^{ \vartheta, k, x }_l ) \big| 
		 \leq  
		 \frac{ 1 }{ M^n } 
		 \sum_{ m = 1 }^{ M^n }
		 \left[ \Big| g\big( X^{ ( \theta, 0, -m ), l, X^{ \vartheta, k, x }_{ l } }_0 \big) \Big| 
			 +  
			 \sum_{ s = 0 }^{ l - 1 } \Big| f_s\big( X^{ ( \theta, 0, m ), l, X^{ \vartheta, k, x }_{ l } }_s, 0 \big) \Big| 
		 \right]
		 \\
		 & + 
		 \sum_{ j = 1 }^{ n - 1 } \frac{ 1 }{ M^{ n - j } } 
		 \sum_{ m = 1 }^{ M^{ n - j } }
		 \frac{ L_{ \mc{R}^{ ( \theta, j, m ) }_l } }{ \mf p_{ l, \mc{R}^{ ( \theta, j, m ) }_{ l } } }
		 	\Big| 
		 	V^{(\theta,j,m)}_{ \mc{R}^{ ( \theta, j, m ) }_{ l }, j }\big( X^{ ( \theta, j, m ), l, X^{ \vartheta, k, x }_l }_{ \mc{R}^{ ( \theta, j, m ) }_l }  \big) 
			- V^{ ( \theta, j, -m ) }_{ \mc{R}^{ ( \theta, j, m ) }_{ l }, j - 1 }\big( X^{ ( \theta, j, m ), l, X^{ \vartheta, k, x }_l }_{ \mc{R}^{ ( \theta, j, m ) }_l } \big)
			\Big|.
		\end{split}
		\end{equation} 
	This shows that for all 
		$ k,l,n \in \N $, 
		$ \theta,\vartheta \in \Theta $,  
		$ x \in \R^d $ 
	with 
		$ 
		l \leq k \leq K $
	it holds that 
		\begin{equation}
		\begin{split}
		&
		\big| V^{ \theta }_{ l, n }( X^{ \vartheta, k, x }_{ l } ) \big|^2
		\\
		& \leq 
		3 \left[ \frac{ 1 }{ M^n } \sum_{ m = 1 }^{ M^n } 
		\Big| g\big( X^{ ( \theta, 0, -m ), l, X^{ \vartheta, k, x }_{ l } }_{ 0 } \big) \Big|\right]^{2}
		+ 
		3 \left[ \frac{ 1 }{ M^n } \sum_{ m = 1 }^{ M^n } \sum_{ s = 0 }^{ l - 1 } \Big| f_s\big( X^{ ( \theta, 0, m ), l, X^{ \vartheta, k, x }_{ l } }_{ s }, 0 \big) \Big| \right]^2 
		\\ 
		& + 
		3 \left[ \sum_{ j = 1 }^{ n - 1 }  \frac{ 1 }{ M^{ n - j } } 
		\sum_{ m = 1 }^{ M^{ n - j } } \frac{ L_{ \mc{R}^{ ( \theta, j, m ) }_{ l } } }{  \mf p_{ l, \mc{R}^{ ( \theta, j, m ) }_{ l } } }
		\Big| V^{ ( \theta, j, m ) }_{ \mc{R}^{ ( \theta, j, m ) }_{ l },j }\big( X^{ ( \theta, j, m ), l, X^{ \vartheta, k, x }_{ l } }_{ \mc{R}^{ ( \theta, j, m ) }_{ l } } \big) - 
		V^{ ( \theta, j, -m ) }_{ \mc{R}^{ ( \theta, j, m ) }_{ l }, j - 1 }\big( X^{ ( \theta, j, m ), l, X^{ \vartheta, k, x }_{ l } }_{ \mc{R}^{ ( \theta, j, m ) }_l } \big) \Big| \right]^2\!. 
		\end{split}
		\end{equation}
	Jensen's inequality therefore ensures that for all 
		$ k,l,n \in \N $, 
		$ \theta,\vartheta \in \Theta $, 
		$ x \in \mc O $
	with 
		$ 
		l \leq k \leq K $ 
	it holds that 
		\begin{equation}
		\begin{split}
		& 
		\Exp{ \big| V^{ \theta }_{ l, n }( X^{ \vartheta, k, x }_{ l } ) \big|^2 } 
		\\
		&  
		\leq 
		\frac{ 3 }{ M^n } \sum_{ m = 1 }^{ M^n } \Exp{ \Big| g\big( X^{ ( \theta, 0, -m ), l, X^{ \vartheta, k, x }_{ l } }_{ 0 } \big) \Big|^2 }
		+ 
		\frac{ 3l }{ M^n } \sum_{ m = 1 }^{ M^n } \sum_{ s = 0 }^{ l - 1 } \Exp{ \Big| f_s\big( X^{ ( \theta, 0, m ), l, X^{ \vartheta, k, x }_{ l } }_{ s }, 0 \big) \Big|^2} 
		\\ 
		& + 
		\sum_{j=1}^{n-1} \frac{ 3n }{ M^{ n - j } } \sum_{ m = 1 }^{ M^{ n - j } } 
		\Exp{  \frac{ | L_{ \mc{R}^{ ( \theta, j, m ) }_{ l } } |^2  }{ | \mf p_{ l, \mc{R}^{ ( \theta, j, m ) }_{ l } } |^2 }
	 	\bigg| V^{ ( \theta, j, m ) }_{ \mc{R}^{ ( \theta, j, m ) }_{ l }, j }\big( X^{ ( \theta, j, m ), l, X^{ \vartheta, k, x }_{ l } }_{ \mc{R}^{ ( \theta, j, m ) }_{ l } } \big) - 
		V^{ ( \theta, j, -m ) }_{ \mc{R}^{ ( \theta, j, m ) }_{ l } , j - 1 } \big( X^{ ( \theta, j, m ), l, X^{ \vartheta, k, x }_{ l } }_{\mc{R}^{(\theta,j,m)}_l} \big) \bigg|^2}\!.
		\end{split}
		\end{equation}
	Item~\eqref{elementary_equidistribution_X_processes:item2} of \cref{lem:elementary_equidistribution_X_processes} hence assures that for all 
		$ k,l,n \in \N $, 
		$ \theta,\vartheta \in \Theta $, 
		$ x \in \mc O $ 
	with 
		$ 
		l \leq k \leq K $ 
	it holds that 
		\begin{equation} \label{integrability:zwischenrechnung}
		\begin{split}
		&
		\Exp{ \big| V^{ \theta }_{ l, n }( X^{ \vartheta, k, x }_{ l } ) \big|^2 } 
		\leq 3 \Exp{ \big| g( X^{ 0, k, x }_{ 0 } ) \big|^2 }
		+ 3 l \sum_{ s = 0 }^{ l - 1 } \Exp{ \big| f_s( X^{ 0, k, x }_{ s }, 0 ) \big|^2 }
		\\
		& + 
		\sum_{ j = 1 }^{ n - 1 } \sum_{ s = 0 }^{ l - 1 } \frac{ 3 n }{ M^{ n - j } } 
		\sum_{ m = 1 }^{ M^{ n - j } } \Exp{ \frac{ |L_{ s }|^{ 2 } }{ | \mf p_{ l, s } |^2 }
		\Big| V^{ ( \theta, j, m ) }_{ s, j }\big( X^{ ( \theta, j, m ), l, X^{ \vartheta, k, x }_{ l } }_{ s } \big) - 
		V^{ ( \theta, j, -m ) }_{ s, j - 1 }\big( X^{ ( \theta, j, m ), l, X^{ \vartheta, k, x }_{ l } }_{ s } \big) \Big|^2 }
		\\
		& \leq 
		3 \Exp{ \big| g( X^{ 0, k, x }_{ 0 } ) \big|^2 } + 3 l \sum_{ s = 0 }^{ l - 1 } \Exp{ \big| f_s( X^{ 0, k, x }_{ s }, 0 ) \big|^2 }
		\\
		& + 
		\sum_{ j = 1 }^{ n - 1 } \sum_{ s = 0 }^{ l - 1 } 
		\frac{ 6 n | L_{ s } |^2 }{ M^{ n - j } | \mf p_{ l, s  } |^2 } \sum_{ m = 1 }^{ M^{ n - j } } 
		\Exp{ \Big| V^{ ( \theta, j, m ) }_{ s, j } \big( X^{ ( \theta, j, m ), l, X^{ \vartheta, k, x }_{ l } }_{ s } \big) \Big|^2
		+  
		\Big| V^{ ( \theta, j, -m ) }_{ s, j - 1 }\big( X^{ ( \theta, j, m ), l, X^{ \vartheta, k, x }_{ l } }_{ s } \big) \Big|^2}.  
		\end{split}
		\end{equation}
	Next note that 
		item~\eqref{elementary_measurability_X_processes:item1} of \cref{lem:elementary_measurability_X_processes}, 
		item~\eqref{measurability:item1} of \cref{lem:measurability}, 
		item~\eqref{elementary_equidistribution_X_processes:item1} of \cref{lem:elementary_equidistribution_X_processes}, 
		item~\eqref{equidistribution:item1} of \cref{lem:equidistribution}, 
 		and item~\eqref{equidistribution_random_fields:item2} of \cref{lem:equidistribution_random_fields} 
	imply that for all 
		$ k,l,m,n \in \N_0 $, 
		$ \theta,\vartheta,\xi \in \Theta $, 
		$ x \in \mc O $ 
	with 
		$ m \leq l \leq k \leq K $
	and 
		$ \xi,\vartheta \notin \{(\theta,\eta) \in \Theta\colon \eta\in\Theta \} $
	it holds that 
		\begin{equation} 
		\mc O \times \Omega \ni ( x, \omega ) \mapsto V^{ \theta }_{ m, n }( X^{ \xi, l, X^{ \vartheta, k, x}_{ l }( \omega ) }_{ m }( \omega ), \omega ) \in \R 
		\end{equation}
	and
		\begin{equation}
		\mc O \times \Omega \ni ( x, \omega ) \mapsto V^{ 0 }_{ m, n }( X^{ 0, k, x }_{ m }( \omega ), \omega ) \in \R 
		\end{equation} 
	are identically distributed random fields. 
	This and \eqref{integrability:zwischenrechnung} imply that 
	for all 
		$ k,l,n \in \N $, 
		$ \theta,\vartheta \in \Theta $,
		$ x \in \mc O $ 
	with 
		$ 
		l \leq k \leq K $ 
	and 
		$ \vartheta \notin \{(\theta,\eta) \in \Theta\colon \eta\in\Theta \} $	
	it holds that 
		\begin{equation} 
		\begin{split}
		\Exp{\big| V^{ \theta }_{ l, n }( X^{ \vartheta, k, x }_{ l } ) \big|^2} 
		& 
		\leq 3 \Exp{\big| g( X^{ 0, k, x }_{ 0 } ) \big|^2 }
		+ 3 l \sum_{ s = 0 }^{ l - 1 } \Exp{ \big| f_s( X^{ 0, k, x }_{ s }, 0 ) \big|^2 }
		\\
		& + \sum_{ j = 1 }^{ n - 1 } \sum_{ s = 0 }^{ l - 1 } \frac{ 12 n | L_{ s } |^2 }{ M^{ n - j } | \mf p_{ l, s } |^2 } 
		\sum_{m=1}^{M^{n-j}} \Exp{ \big| V^{ 0 }_{ s, j }( X^{ 0, k, x }_{ s } ) \big|^2}\!.
		\end{split}
		\end{equation}
	The fact that for all 
		$ k \in \{ 0, 1, \ldots, K \} $, 
		$ x \in \mc O $ 
	it holds that 
		$ \EXP{ | g( X^{ 0, k, x }_{ 0 } ) |^2 } + \sum_{ s = 0 }^{ k - 1 } \EXP{ | f_s( X^{ 0, k, x }_{ s }, 0 ) |^2} < \infty $	
	therefore demonstrates that for all 
		$ n \in \A $ 
	it holds that 
		$ n+1 \in \A $. 
	The fact that $ 1 \in \A $ and induction hence imply that 
		$ \A = \N $. 
	Therefore, we obtain that for all 
		$ k,l,n \in \N_0 $, 
		$ \theta,\vartheta \in \Theta $, 
		$ x \in \mc O $ 
	with 
		$ l \leq k \leq K $ 
	and 
		$ \vartheta \notin \{(\theta,\eta)\in\Theta\colon \eta\in\Theta \} $
	it holds that 
	$ \EXP{ | V^{\theta}_{l,n} ( X^{ \vartheta, k, x }_{ l } )|^2} < \infty $.
	This completes the proof of \cref{lem:integrability}.
\end{proof}

\section[Error analysis for MLP approximations for nested expectations]{Error analysis for MLP approximations for iterated nested expectations}
\label{sec:analysis_of_MLP_approximations}

In this section we provide in \cref{prop:error_estimate} in \cref{subsec:full_error_analysis} below a full error analysis for MLP approximations for iterated nested expectations. 
Our proof of \cref{prop:error_estimate} is inspired by Hutzenthaler et al.~\cite{Overcoming} and is based on the idea to combine the recursive error estimate in \cref{lem:error_recursion} in \cref{subsec:recursive_error_estimate}, the elementary integration by parts type result in \cref{lem:special_alphas} in \cref{subsec:full_error_analysis}, the elementary Gronwall type inequality in \cref{lem:gronwall_type_inequality} in \cref{subsec:full_error_analysis}, and the elementary a priori bounds for iterated nested expectations in \cref{lem:a_priori_estimates} in \cref{subsec:a_priori_bounds}. 
Our proof of the recursive error estimate in \cref{lem:error_recursion} in \cref{subsec:recursive_error_estimate} is, roughly speaking, based on bias-variance decompositions for the approximation errors of the proposed MLP approximation schemes, the bias estimates for MLP approximations in \cref{lem:bias_estimate} in \cref{subsec:bias_and_variance_estimates}, and the variance estimates for MLP approximations in \cref{lem:variance_estimate} in \cref{subsec:bias_and_variance_estimates}.

\subsection{A priori estimates for nested expectations} 
\label{subsec:a_priori_bounds}

\begin{lemma} \label{lem:welldefinedness_induction_step} 
	Let $ d,K \in \N $, 
		$ L \in \R $, 
		$ \mc O \in \Borel( \R^d ) \setminus \{ \emptyset \} $, 
	let $ f \colon \mc O \times \R \to \R $ be $ \Borel( \mc O \times \R ) $/$ \Borel( \R ) $-measurable, 
	assume for all 
		$ x \in \mc O $, 
		$ a,b \in \R $ 
	that 
		$ | f ( x, a ) - f( x, b ) | \leq L | a - b | $,  
	let $ ( S, \mc S ) $ be a measurable space, 
	let $ \phi_k \colon \mc O \times S \to \mc O $, $ k \in \{ 0, 1, \ldots, K \} $, be $ ( \Borel( \mc O ) \otimes \mc S ) $/$ \Borel(\mc O) $-measurable, 
	let $ ( \Omega, \mc F, \P ) $ be a probability space, 
	let $ W_k \colon \Omega \to S $, $ k \in \{ 0, 1, \ldots, K \} $, be  independent random variables, 
	for every 
		$ k \in \{ 0, 1, \ldots, K \} $ 
	let $ X^k = (X^{ k, x }_{ l } )_{ ( l, x ) \in \{ 0, 1, \ldots, k \} \times \mc O } \colon \allowbreak \{ 0, 1, \ldots, k \} \times \mc O \times \Omega \to \mc O $ satisfy for all 
		$ l \in \{ 0, 1, \ldots, k \} $, 
		$ x \in \mc O $ 
	that 
		\begin{equation} \label{welldefinedness_induction_step:X_dynamics}
		X^{ k, x }_{ l } = 
		\begin{cases} 
		x & \colon l=k \\
		\phi_{ l }( X^{ k, x }_{ l + 1 }, W_{ l } ) & \colon l < k, 
	    \end{cases} 
		\end{equation} 
	let $ v \colon \mc O \to \R $ be $ \Borel( \mc O ) $/$ \Borel(\R) $-measurable, 
	and assume for all 
		$ x \in \mc O $ 
	that 
		$ \sum_{ k = 0 }^{ K } \EXP{ | f ( X^{ k, x }_{ 0 }, 0 ) |^2 + | v( X^{ k, x }_{ 0 } ) |^2  } < \infty $. 
	Then there exists a $ \Borel( \mc O ) $/$ \Borel( \R ) $-measurable $ u \colon \mc O \to \R $ such that for all 
		$ x \in \mc O $ 
	it holds that 
		$ \sum_{ k = 1 }^{ K } \EXP{ | u( X^{ k, x }_{ 1 } ) |^2 + | f( X^{ 1, x }_{ 0 }, v( X^{ 1, x }_{ 0 } ) ) |^2 } < \infty $ 
	and 
		$ u(x) = \EXP { f( X^{1, x}_{ 0 }, v( X^{ 1, x }_{ 0 } ) ) } $.
\end{lemma} 

\begin{proof}[Proof of \cref{lem:welldefinedness_induction_step}]
	First, observe that the assumption that for all  
		$ x \in \mc O $, 
		$ a,b \in \R $ 
	it holds that 
		$ | f ( x, a ) - f ( x, b ) | \leq L | a - b | $ 
	and the assumption that for all 
		$ k \in \{ 0, 1, \ldots, K \} $, 
		$ x \in \mc O $ 
	it holds that 
		$ \EXP{ | f ( X^{ k, x }_0, 0 ) |^2 + | v ( X^{ k, x }_0 ) |^2 } < \infty $
	ensure that for all 
		$ k \in \{ 0, 1, \ldots, K \} $, 
		$ x \in \mc O $ 
	it holds that 
		\begin{equation} \label{welldefinedness_induction_step:square_integrability}
		\begin{split}
		\EXPP{ | f ( X^{ k, x }_{ 0 }, v( X^{ k, x }_{ 0 } ) )|^2 } 
		 &\leq 
		\EXPP{ ( | f ( X^{ k, x }_{ 0 }, 0 ) | + L | v( X^{ k, x }_{ 0 } ) | )^{2} } 
		\\
		& 
		\leq 2 \,\EXPP{ | f ( X^{ k, x }_{ 0 }, 0 ) |^2 } + 2 L^2 \,\EXPP{ | v( X^{ k, x }_{ 0 } ) |^2 } < \infty.  
		\end{split}
		\end{equation} 
	Jensen's inequality hence proves that there exists $ u \colon \mc O \to \R $ which satisfies that for all 
		$ x \in \mc O $ 
	it holds that 
		$ \EXP{ | f( X^{ 1, x }_{ 0 }, v( X^{ 1, x }_{ 0 } ) ) | } \leq (\EXP{ | f( X^{ 1, x }_{ 0 }, v( X^{ 1, x }_{ 0 } ) ) |^2 })^{\!\nicefrac12} < \infty $ 
	and 
		\begin{equation} \label{welldefinedness_induction_step:introducing_u}
		u(x) = \EXP{ f( X^{ 1, x }_{ 0 }, v( X^{ 1, x }_{ 0 } ) ) }.
		\end{equation} 
	Next note that Klenke~\cite[Theorem 14.16]{Klenke2014} and \eqref{welldefinedness_induction_step:introducing_u} establish that $ u $ is $ \Borel( \mc O ) $/$ \Borel( \R ) $-measurable. 
	Moreover, observe that item~\eqref{elementary_measurability_X_processes:item1} of \cref{lem:elementary_measurability_X_processes} yields that for all 
		$ k \in \{ 1, 2, \ldots, K \} $, 
		$ x \in \mc O $ 
	it holds that 
		$ X^{ k, x }_{ 1 } \colon \Omega \to \mc O $ 
	is 
		$ \sigma( ( W_{ r } )_{ r \in [ 1, k - 1 ] \cap \N_0 } ) $/$ \Borel( \mc O ) $-measurable. 
	Combining \eqref{welldefinedness_induction_step:X_dynamics}, \eqref{welldefinedness_induction_step:introducing_u}, Jensen's inequality, the assumption that $ W_k $, $ k \in \{ 0, 1, \ldots, K - 1 \} $, are independent random variables, item~\eqref{elementary_measurability_X_processes:item1} of \cref{lem:elementary_measurability_X_processes}, 
	and Hutzenthaler et al.~\cite[Lemma 2.2]{Overcoming} hence guarantees that for all 
		$ k \in \{ 1, 2, \ldots, K \} $, 
		$ x \in \mc O $ 
	it holds that 
		\begin{equation} 
		\begin{split}
		\EXPP{ | u( X^{ k, x }_{ 1 } )|^2} 
		& 
		= 
		\int_{ \mc O } | u( y ) |^2 \, \big( X^{ k, x }_{ 1 }( \P )\big) ( dy ) 
		\leq
		\int_{ \mc O } \EXPP{ | f( X^{ 1, y }_{ 0 }, v( X^{ 1, y }_{ 0 } ) ) |^2 } \, \big( X^{ k, x }_{ 1 }( \P ) \big) ( dy )
		\\
		&  
		= \Exp{ \Big| f \big( X^{ 1, X^{ k, x }_{ 1 } }_{ 0 }, v( X^{ 1, X^{ k, x }_{ 1 } }_{ 0 } ) \big) \Big|^2 } 
		= \EXPP{ | f ( X^{ k, x }_{ 0 }, v( X^{ k, x }_{ 0 } ) |^2 } 
		< \infty. 
		\end{split}
		\end{equation} 
	Combining this, \eqref{welldefinedness_induction_step:square_integrability}, \eqref{welldefinedness_induction_step:introducing_u}, and the fact that $ u $ is $ \Borel( \mc O ) $/$ \Borel( \R ) $-measurable completes the proof of \cref{lem:welldefinedness_induction_step}.
\end{proof}

\begin{lemma} \label{lem:welldefinedness_exact_solution}
	Let $ d,K \in \N $, 
		$ L \in \R $, 
		$ \mc O \in \Borel( \R^d ) \setminus \{ \emptyset \} $, 
	let $ f_k \colon \mc O \times \R \to \R $, $ k \in \{ 0, 1, \ldots, K \} $, be $ \Borel( \mc O \times \R ) $/$ \Borel(\R) $-measurable, 
	let $ g \colon \mc O \to \R $ be $ \Borel( \mc O ) $/$ \Borel(\R) $-measurable, 
	assume for all 
		$ k \in \{0,1,\ldots,K\} $,
		$ x \in \mc O $, 
		$ a,b \in \R $ 
	that 
		$ | f_k ( x, a ) - f_k( x, b ) | \leq L | a - b | $,  
	let $ (S,\mc S) $ be a measurable space, 
	let $ \phi_k \colon \mc O \times S \to \mc O $, $ k \in \{ 0, 1, \ldots, K \} $, be $ ( \Borel( \mc O ) \otimes \mc S ) $/$ \Borel( \mc O ) $-measurable, 
	let $ ( \Omega, \mc F, \P ) $ be a probability space, 
	let $ W_k \colon \Omega \to S $, $ k \in \{ 0, 1, \ldots, K \} $, be independent random variables, 
	for every 
		$ k \in \{0,1,\ldots,K\} $ 
	let 
		$ X^k = ( X^{ k, x }_{ l } )_{ (l,x) \in \{ 0, 1, \ldots, k \} \times \mc O } \colon \{ 0, 1, \ldots, k \} \times \mc O \times \Omega \to \mc O $
	satisfy for all 
		$ l \in \{ 0, 1, \ldots, k \} $, 
		$ x \in \mc O $ 
	that 
		\begin{equation} 
		X^{ k, x }_{ l } = 
		\begin{cases} 
			x & \colon l = k \\
			\phi_{ l }( X^{ k, x }_{ l + 1 }, W_{ l} ) & \colon l<k, 
		\end{cases} 
		\end{equation} 
	and assume for all 
		$ k \in \{ 0, 1, \ldots, K \} $, 
		$ x \in \mc O $ 
	that 
		$ \EXP{ | g( X^{ k, x }_{ 0 } ) |^2 + \sum_{ l = 0 }^{ k - 1 }  | f_l( X^{ k, x }_l, 0 ) |^2 } < \infty $. 
	Then there exist $ \Borel( \mc O ) $/$ \Borel(\R) $-measurable 
		$ v_k \colon \mc O \to \R $, $ k \in \{ 0, 1, \ldots, K \} $, 
	such that for all
		$ k \in \{ 1, 2, \ldots, K \} $, 
		$ x \in \mc O $ 
	it holds that 
		$ \EXP{ | f_{ k - 1 } ( X^{ k, x }_{ k - 1 }, v_{ k - 1 }( X^{ k, x }_{ k - 1 } ) ) | + \sum_{ l = 1 }^{ k } | v_l( X^{ k, x }_{ l } ) |^2 } < \infty $, 
		$ v_{ 0 }( x ) = g( x ) $, 
	and  
		$ v_k(x) = \EXP{ f_{ k - 1 } ( X^{ k, x }_{ k - 1 }, v_{ k - 1 }( X^{ k, x }_{ k - 1 } ) ) } $.
\end{lemma} 

\begin{proof}[Proof of \cref{lem:welldefinedness_exact_solution}] 
	Throughout this proof let 
		$ \A \subseteq \{0,1,\ldots, K\} $ 
	satisfy that 
		\begin{equation} 
		\A = \left\{ k \in \{0,1,\ldots,K\} \colon \left( \begin{array}{c}  \text{There exist}~\Borel( \mc O )\text{/}\Borel( \R )\text{-measurable}~\\
		v_l \colon \mc O \to \R, l \in \{0, 1, \ldots, k\},~\text{such that}~\\
		\text{for all}~l \in \{1, 2, \ldots, k\}, x \in \mc O~\text{it holds that}~\\ 
		\EXP{ | f_{ l - 1 }( X^{ l, x }_{ l - 1 }, v_{ l - 1 }( X^{ l, x }_{ l - 1 } ) ) | + \sum_{m=l}^{K} | v_l( X^{ m, x }_{ l } ) |^2 } < \infty, \\
		v_0(x) = g(x),~\text{and}~v_l(x) = \EXP{ f_{ l - 1 }( X^{ l, x }_{ l - 1 }, v_{ l - 1 }( X^{ l, x }_{ l - 1 } ) ) } .
		\end{array} \right) \right\}.
		\end{equation} 
	Note that the assumption that $ g $ is $ \Borel( \mc O ) $/$ \Borel( \R ) $-measurable ensures that $ 0 \in \A $. 
	Moreover, observe that \cref{lem:welldefinedness_induction_step} (applied with 
		$ d \is d $, 
		$ K \is K - k $, 
		$ L \is L $, 
		$ \mc O \is \mc O $, 
		$ f \is f_k $, 
		$ ( S, \mc S ) \is ( S, \mc S ) $, 
		$ ( \phi_l )_{ l \in \{ 0, 1, \ldots, K - 1 \} } \is ( \phi_{ k + s } )_{ s \in \{ 0, 1, \ldots, K - k - 1 \} } $, 
		$ ( \Omega, \mc F, \P ) \is ( \Omega, \mc F, \P ) $, 
		$ ( W_l )_{ l \in \{ 0, 1, \ldots, K - 1 \} } \is ( W_{ k + s } )_{ s \in \{ 0, 1, \ldots, K - k - 1 \} } $, 
		$ v \is v_k $
	for $ k \in \{ 0, 1, \ldots, K - 1 \} \cap \A $ in the notation of \cref{lem:welldefinedness_induction_step}) and the assumption that for all 
		$ k \in \{ 0, 1, \ldots, K \} $, 
		$ x \in \mc O $ 
	it holds that 
		$ \EXP{ | g( X^{ k, x }_{ 0 } ) |^2 } < \infty$ 
	ensure that for all
		$ k \in \{ 0, 1, \ldots, K - 1 \} \cap \A $ 
	there exists $ v_{ k + 1 } \colon \mc O \to \R $ which satisfies that 
		\begin{enumerate}[(I)]
			\item it holds that $ v_{ k + 1 } $ is $ \Borel( \mc O ) $/$ \Borel( \R ) $-measurable, 
			\item it holds for all 
				$ m \in \{ k + 1, \ldots, K \} $, 
				$ x \in \mc O $ 
			that 
				$ \EXP{ | v_{ k + 1 }( X^{ m, x }_{ k + 1 } ) |^2 } < \infty $, 
			and 
			\item it holds for all 
				$ x \in \mc O $ 
			that 
				$ \EXP{ | f_{ k }( X^{ k + 1, x }_k , v_{ k }( X^{ k + 1, x }_k ) ) | } < \infty $
			and  
				\begin{equation} 
				v_{ k + 1 }( x ) = \EXP{ f_{ k }( X^{ k + 1, x }_k , v_{ k }( X^{ k + 1, x }_k ) ) }.
				\end{equation} 
		\end{enumerate}
	This demonstrates for all
		$ k \in \{ 0, 1, \ldots, K - 1 \} \cap \A $ 
	that 
		$ k + 1 \in \A $. 
	Induction and the fact that $ 0 \in \A $ therefore establish that 
		$ \A = \{ 0, 1, \ldots, K \} $. 
	This completes the proof of \cref{lem:welldefinedness_exact_solution}. 
\end{proof}

\begin{cor} \label{cor:well_definedness_exact_solution_brownian_case}
	Let $d, K \in \N $, 	
		$L, T \in \R $, 
	let $t_k \in \R$, $ k \in \{0, 1, \ldots, K\}$, satisfy 
		$0 = t_0 < t_1 < \ldots < t_K = T$, 
	let $f_k \colon \R^d \times \R \to \R$, $ k \in \{ 0, 1, \ldots, K \} $, be $\Borel(\R^d\times\R)$/$\Borel(\R)$-measurable, 
	let $g\colon\R^d \to \R$ be $\Borel(\R^d)$/$\Borel(\R)$-measurable, 
	assume for all 
		$ k \in \{0, 1, \ldots, K\} $, 
		$ x \in \R^d $, 
		$ a,b \in \R $
	that 
		$ | f_k( x, a) - f_k(x, b) | \leq L | a - b | $, 
	let $( \Omega, \mc F, \P) $ be a probability space, 
	let $W \colon [0, T]\times \Omega \to \R^d $ be a standard Brownian motion, 
	and assume for all 
		$ x \in \mc O $ 
	that 
		$ \EXP{|g(x+W_{T})|^2 + \sum_{l=0}^{K-1} |f_l( x + W_{T - t_l}, 0) |^2 } < \infty $. 
	Then there exist $ \Borel(\R^d) $/$ \Borel(\R) $-measurable
		$ v_k \colon \R^d \to \R $, $ k \in \{ 0, 1, \ldots, K \} $,
	such that for all
		$ k \in \{ 1, 2, \ldots, K \} $, 
		$ x \in \mc O $ 
	it holds that 
		$ \EXP{ | f_{k-1}(x + W_{t_k-t_{k-1}}, v_{k-1}(x + W_{t_k-t_{k-1}})) | + \sum_{ l = 1 }^{ k } |v_l( x + W_{t_k} - W_{t_l})|^2 } < \infty $, 
		$ v_0(x) = g(x) $, 
	and 
		$ v_k(x) = \EXP{ f_{k-1}(x + W_{t_k-t_{k-1}}, v_{k-1}(x + W_{t_k-t_{k-1}})) } $. 
\end{cor}

\begin{proof}[Proof of \cref{cor:well_definedness_exact_solution_brownian_case}]
	First, observe that the assumption that $ W $ is a Brownian motion and the assumption that for all 
		$ x \in \R^d $ 
	it holds that 
		$ \EXP{ | g ( x + W_{ T } ) |^2 } < \infty $  
	show that for all 
		$ k \in \{1, 2, \ldots, K\} $
	it holds that 
		\begin{equation} \label{well_definedness_exact_solution_brownian_case:integrability_of_g}
		\begin{split} 
		& \Exp{ | g( x + W_{ t_k } ) |^2 } 
		= \left[ 2 \pi t_k \right]^{ -\nicefrac{ d }{ 2 } } 
		\int_{ \R^d } | g ( x + z ) |^2 \exp\!\left( -\tfrac{ \norm{ z }^2 }{ 2 t_{ k } } \right)\!\,dz 
		\\
		& \leq \left[ 2 \pi t_{ k } \right]^{ -\nicefrac{ d }{ 2 } } 
		\int_{ \R^d } | g ( x + z ) |^2 \exp\!\left( -\tfrac{ \norm{ z }^2 }{ 2 T } \right)\!\,dz 
		= \left[ \tfrac{ T }{ t_{ k } } \right]^{ \nicefrac{ d }{ 2 } } \Exp{ | g ( x + W_{ T } ) |^2 } 
		< \infty. 
		\end{split} 
		\end{equation}
	Furthermore, note that the assumption that $ W $ is a standard Brownian motion and the assumption that $ \EXP{ \sum_{ k = 0 }^{ K - 1 } | f_l( x + W_{ T - t_l }, 0 ) |^2 } < \infty $ show that for all 
		$ k, l \in \{0, 1, \ldots, K\} $ 
	with 
		$ k \geq l $ 
	it holds that 
		$ \EXP{ |f_l(x+W_{t_k-t_l}, 0)|^2} < \infty $. 
	This, \eqref{well_definedness_exact_solution_brownian_case:integrability_of_g}, and the assumption that $ W $ is a standard Brownian motion prove that for all 
		$ k \in \{ 1, 2, \ldots, K \} $, 
		$ x \in \R^d $  
	it holds that 
		$ \EXP{ | g( x + W_{ t_{ k } } ) |^2 + \sum_{ l = 0 }^{ k - 1 } | f_l ( x + W_{ t_{ k } } - W_{ t_{ l } }, 0 ) |^2 } < \infty $. 
	\cref{lem:welldefinedness_exact_solution} (applied with 
		$ d \is d $, 
		$ K \is K $, 
		$ L \is L $, 
		$ \mc O \is \R^d $, 
		$ (f_k)_{k\in\{0, 1, \ldots, K-1\}} \is (f_k)_{k\in\{0,1,\ldots,K-1\}}$, 
		$ g \is g $, 
		$ (S, \mc S) \is (\R^d,\Borel(\R^d)) $, 
		$ (\phi_k)_{k\in\{0, 1, \ldots, K-1\}} \is ((\R^d \times \R^d \ni (x, w) \mapsto x + w \in \R^d))_{ k\in\{0, 1, \ldots, K-1\}} $, 
		$ (\Omega, \mc F, \P) \is (\Omega, \mc F, \P) $, 
		$ (W_k)_{k\in\{0, 1, \ldots, K-1\}} \is (W_{t_{k+1}} - W_{t_{k}})_{k\in\{0, 1, \ldots, K-1\}} $
	in the notation of \cref{lem:welldefinedness_exact_solution}) and the assumption that $W$ is a standard Brownian motion therefore establish the claim. 
	This completes the proof of \cref{cor:well_definedness_exact_solution_brownian_case}. 
\end{proof}

\begin{lemma} \label{lem:exact_solution_alternative_representation}
	Let $ d,K \in \N $, 
		$ L \in \R $, 
		$ \mc O \in \Borel( \R^d ) \setminus \{ \emptyset \} $, 
	let $ f_k \colon \mc O \times \R \to \R $, $ k \in \{ 0, 1, \ldots, K \} $, be $ \Borel( \mc O \times \R ) $/$ \Borel( \R ) $-measurable, 
	let $ g \colon \mc O \to \R $ be $ \Borel( \mc O ) $/$ \Borel( \R ) $-measurable, 
	assume for all 
		$ k \in \{ 0, 1, \ldots, K \} $,
		$ x \in \mc O $, 
		$ a,b \in \R $ 
	that 
		$ | f_k ( x, a ) - f_k ( x, b ) | \leq L | a - b | $,  
	let $ ( S, \mc S ) $ be a measurable space, 
	let $ \phi_{ k } \colon \mc O \times S \to \mc O $, $ k \in \{ 0, 1, \ldots, K \} $, be $ ( \Borel( \mc O ) \otimes \mc S ) $/$ \Borel( \mc O ) $-measurable, 
	let $ ( \Omega, \mc F, \P ) $ be a probability space, 
	let $ W_{ k } \colon \Omega \to S $, $ k \in \{ 0, 1, \ldots, K \} $, be independent random variables, 
	for every 
		$ k \in \{0,1,\ldots,K\} $ 
	let 
		$ X^k = ( X^{ k, x }_{ l } )_{ ( l, x ) \in \{ 0, 1, \ldots, k \} \times \mc O } \colon \{ 0, 1, \ldots, k \} \times \mc O \times \Omega \to \mc O $
	satisfy for all 
		$ l \in \{ 0, 1, \ldots, k \} $, 
		$ x \in \mc O  $ 
	that 
		\begin{equation} \label{exact_solution_alternative_representation:dynamics}
		X^{ k, x }_{ l } =
		\begin{cases} 
		x & \colon l=k \\
		\phi_{ l }( X^{ k, x }_{ l + 1 }, W_{ l } ) & \colon l < k, 
		\end{cases} 
		\end{equation} 
	assume for all 
		$ k \in \{ 0, 1, \ldots, K \} $, 
		$ x \in \mc O $
	that 
		$ \EXP{ | g( X^{ k, x }_{ 0 } ) |^2 + \sum_{l=0}^{k-1} | f_l ( X^{ k, x }_{ l },0 ) |^2 } < \infty $, 
	and let 
		$ v_k \colon \mc O \to \R $, $ k \in \{ 0, 1, \ldots, K \} $,
	be $ \Borel( \mc O ) $/$ \Borel( \R ) $-measurable functions which satisfy for all 
		$ k \in \{ 1, 2, \ldots, K \} $, 
		$ x \in \mc O $ 
	that 
		$ v_0( x ) = g( x ) $ 
	and 
		\begin{equation} \label{exact_solution_alternative_representation:iteration}
		v_k( x ) = \EXPP{ f_{ k - 1 }( X^{ k, x }_{ k - 1 }, v_{ k - 1 }( X^{ k, x }_{ k - 1 } ) ) }
		\end{equation} 
	(cf.\,\cref{lem:welldefinedness_exact_solution}). 
	Then it holds for all 
		$ k,l \in \{ 0, 1, \ldots, K \} $, 
		$ x \in \mc O $ 
	with 
		$ k \geq l $ 
	that 
		$ \EXP{ | v_l ( X^{ k, x }_{ l } ) | } + \sum_{ s = l }^{ k - 1 } \Exp{ | f_s ( X^{ k, x }_{ s }, v_{ s }( X^{ k, x }_{ s } ) ) - v_{ s }( X^{ k, x }_{ s } ) |} < \infty
		$ 
	and
		\begin{equation} \label{exact_solution_alternative_representation:claim}
		 v_{ k }( x ) = \EXPP{ v_{ l }( X^{ k, x }_{ l } ) } + \sum_{ s = l }^{ k - 1 } \EXPP{ f_{ s }( X^{ k, x }_{ s }, v_{ s }( X^{ k, x }_{ s } ) ) - v_{ s }( X^{ k, x }_{ s } ) }.
		\end{equation} 
\end{lemma} 

\begin{proof}[Proof of \cref{lem:exact_solution_alternative_representation}] 
	Throughout this proof let 
		$ k \in \{ 1, 2, \ldots, K \}$.
	Observe that Jensen's inequality and \cref{lem:welldefinedness_exact_solution} prove that for all 
		$ l \in \{ 0, 1, \ldots, k \}$, 
		$ x \in \mc O $ 
	it holds that 
		\begin{equation} \label{exact_solution_alternative_representation:integrability}
		\EXP{ | v_{ l }( X^{ k, x }_{ l } ) | } \leq \big( \EXPP{ | v_l( X^{ k, x }_{ l } ) |^2 } \big)^{\!\nicefrac{1}{2}} < \infty. 
		\end{equation}
	Combining the assumption that for all 
		$ l \in \{0,\ldots,k-1\} $, 
		$ x \in \mc O $, 
		$ a,b \in \R $ 
	it holds that 
		$| f_{ l }( x, a ) - f_{ l }( x, b ) | \leq L | a - b | $
	and
		$ \EXP{ | f_{ l }( X^{ k, x }_{ l }, 0 ) |^2 } < \infty $ 
	with Jensen's inequality and Minkowski's inequality hence implies that for all 
		$ l \in \{ 1, 2, \ldots, k \} $, 
		$ x \in \mc O $ 
	it holds that 
		\begin{equation} \label{exact_solution_alternative_representation:integrability_of_term_involving_f}
		\begin{split}
		& 
		\EXPP{ | f_{ l - 1 } ( X^{ k, x }_{ l - 1 }, v_{ l - 1 } ( X^{ k, x }_{ l - 1 } ) ) | } 
		\leq 
		\left( \EXPP{ | f_{ l - 1 } ( X^{ k, x }_{ l - 1 }, v_{ l - 1 } ( X^{ k, x }_{ l - 1 } ) ) |^2 }  \right)^{\!\nicefrac12}
		\\
		& \leq 
		\left( \EXPP{ | f_{ l - 1 } ( X^{ k, x }_{ l - 1 }, 0 ) |^2 } \right)^{\!\nicefrac12}
		+ 
		L \left( \EXPP{ | v_{ l - 1 }( X^{ k, x }_{ l - 1 } ) |^2 } 
		\right)^{\!\nicefrac12}
		< \infty. 
		\end{split}
		\end{equation} 
	This and \eqref{exact_solution_alternative_representation:integrability} establish that for all 
		$ l \in \{ 0, 1, \ldots, k \} $, 
		$ x \in \mc O $ 
	it holds that 
		\begin{equation}
		\Exp{ | v_l( X^{ k, x }_{ l } ) | + \sum_{ s = l }^{ k - 1 } | f_{ s }( X^{ k, x }_{ s }, v_{ s }( X^{ k, x }_{ s } ) ) - v_{ s }( X^{ k, x }_{ s } )| } < \infty.  
		\end{equation}
	In the next step we prove that for all 
		$ l \in \{ 0, 1, \ldots, k \} $, 
		$ x \in \mc O $ 
	it holds that 
		\begin{equation} \label{exact_solution_alternative_representation:remains_to_be_proved}
		v_{ k }( x ) = \EXPP{ v_{ l }( X^{ k, x }_{ l } ) } + \sum_{ s = l }^{ k - 1 } \EXPP{ f_{ s }( X^{ k, x }_{ s }, v_{ s }( X^{ k, x }_{ s } ) ) - v_{ s }( X^{ k, x }_{ s } ) }.
		\end{equation}
	For this we observe that  item~\eqref{elementary_measurability_X_processes:item1} of
		\cref{lem:elementary_measurability_X_processes} 
	ensures that 
	\begin{enumerate}[(I)] 
		\item \label{exact_solution_alternative_representation:proof_item1} 
		for all 
			$ l \in \{ 1, 2, \ldots, k \} $
		it holds that 
			$ \mc O \times \Omega \ni ( x, \omega ) \mapsto X^{ l, x }_{ l - 1 }( \omega )  \in \mc O $ 
		is $ ( \Borel( \mc O ) \otimes \sigma( W_{ l - 1 } ) $/$ \Borel( \mc O ) $-measurable 
		and 
		\item \label{exact_solution_alternative_representation:proof_item2} 
		for all 
			$ l \in \{ 0, 1, \ldots, k \} $, 
			$ x \in \mc O $  
		it holds that 			
			$ \Omega \ni \omega \mapsto X^{ k, x }_l( \omega ) \in \mc O $ is $ \sigma( ( W_{ m } )_{ m \in [ l, k ) \cap\N_0 } ) $/$ \Borel(\mc O ) $-measurable. 
	\end{enumerate}
	This, the assumption that $ W_{ m } \colon \Omega \to S $, $ m \in \{ 0, 1, \ldots, K - 1 \} $, are independent random variables, and Hutzenthaler et al.~\cite[Lemma 2.2]{Overcoming} ensure that for all 
		$ l \in \{ 1, 2, \ldots, k \} $, 
		$ x \in \mc O $ 
	it holds that 
		\begin{equation}
		\int_{ \mc O } 
		\Exp{ | f_{ l - 1 } ( X^{ l, y }_{ l - 1 }, v_{ l - 1 } ( X^{ l, y }_{ l - 1 } ) ) |^2 }
		\!\,\big( X^{ k, x }_{ l }(\P)\big)(dy)
		= 
		\Exp{ \Big| f_{ l - 1 } \big( X^{ l, X^{ k, x }_{ l } }_{ l - 1 }, v_{ l - 1 } ( X^{ l, X^{ k, x }_{ l } }_{ l - 1 } ) \big) \Big|^2 } 
		\!.	
		\end{equation}
	Combining this with \eqref{exact_solution_alternative_representation:dynamics} ensures that for all 
		$ l \in \{ 1, 2, \ldots, k \} $, 
		$ x \in \mc O $ 
	it holds that 
		\begin{equation} 
		\int_{ \mc O } \Exp{ | f_{ l - 1 } ( X^{ l, y }_{ l - 1 }, v_{ l - 1 } ( X^{ l, y }_{ l - 1 } ) |^2 } \!\,\big( X^{ k, x }_{ l }(\P)\big) (dy)
		=
		\Exp{ | f_{ l - 1 }( X^{ k, x }_{ l - 1 }, v_{ l - 1 }( X^{ k, x }_{ l - 1 } ) ) |^2 }  
		\!.	
		\end{equation} 
	Jensen's inequality and \eqref{exact_solution_alternative_representation:integrability_of_term_involving_f} hence guarantee that for all 
		$ l \in \{ 1, 2, \ldots, k \} $, 
		$ x \in \mc O $ 
	it holds that 
		\begin{equation} \label{exact_solution_alternative_representation:integrability_of_f_term}
		\begin{split}
		& 
		\int_{ \mc O } \Exp{ \big| f_{ l - 1 }( X^{ l, y }_{ l - 1 }, v_{ l - 1 } ( X^{ l, y }_{ l - 1 } ) ) \big| } \!\,\big( X^{ k, x }_{ l }(\P) \big) (dy)
		\\
		& \leq 
		\left[ \int_{ \mc O } \left( \Exp{ \big| f_{ l - 1 } ( X^{ l, y }_{ l - 1 }, v_{ l - 1 }( X^{ l, y }_{ l - 1 } ) ) \big| } \right)^{\!2} \!\, \big( X^{ k, x }_{ l } (\P) \big) (dy) \right]^{\nicefrac12}	
		\\ 
		& \leq 
		\left[ \int_{ \mc O } \Exp{ \big| f_{ l - 1 } ( X^{ l, y }_{ l - 1 }, v_{ l - 1 } ( X^{ l, y }_{ l - 1 }) ) \big|^2 } 		\!\,\big( X^{ k, x }_{ l }( \P ) \big) (dy)
		\right]^{\nicefrac12}	
		\\
		& =  
		\left( \Exp{ \big| f_{ l - 1 }( X^{ k, x }_{ l - 1 }, v_{ l - 1 }( X^{ k, x }_{ l - 1 } ) ) \big|^2 }
		\right)^{\nicefrac12}
		< \infty. 
		\end{split}
		\end{equation} 
	The assumption that 
		$W_m\colon\Omega\to S$, $m\in\{0,1,\ldots,K-1\}$, 
	are independent random variables, item~\eqref{exact_solution_alternative_representation:proof_item1}, 	item~\eqref{exact_solution_alternative_representation:proof_item2}, and Hutzenthaler et al.~\cite[Lemma 2.2]{Overcoming} therefore ensure that for all 
		$ l \in \{ 1, 2, \ldots, k \} $, 
		$ x \in \mc O $ 
	it holds that
		\begin{equation} 
		\begin{split}
		&
		\Exp{ f_{ l - 1 } \big( X^{ l, X^{ k, x }_{ l } }_{ l - 1 }, v_{ l - 1 } ( X^{ l, X^{ k, x }_{ l } }_{ l - 1 } ) \big) }
		\\
		& = 
		\Exp{ \max \left\{ f_{ l - 1 } \big( X^{ l, X^{ k, x }_{ l } }_{ l - 1 }, v_{ l - 1 } ( X^{ l, X^{ k, x }_{ l } }_{ l - 1 } ) \big), 0 \right\} } 
		\\
		& \quad -
		\Exp{ \max \left\{ -f_{ l - 1 } \big( X^{ l, X^{ k, x }_{ l } }_{ l - 1 }, v_{ l - 1 }( X^{ l, X^{ k, x }_{ l } }_{ l - 1 } ) \big) , 0 \right\} }
		\\
		& = 
		\int_{ \mc O }
		\Exp{ \max \left\{ f_{ l - 1 } ( X^{ l, y }_{ l - 1 }, v_{ l - 1 }( X^{ l, y }_{ l - 1 } ) ), 0 \right\} } \!\,\big( X^{ k, x }_{ l }(\P) \big) ( dy ) 
		\\
		& \quad - \int_{ \mc O } 
		\Exp{ \max \left\{ -f_{ l - 1 }( X^{ l, y }_{ l - 1 }, v_{ l - 1 }( X^{ l, y }_{ l - 1 } ) ), 0 \right\} } 
		\!\,\big( X^{ k, x }_{ l }(\P) \big) ( dy )
		\\
		& = 
		\int_{ \mc O }
		\Exp{ f_{ l - 1 } ( X^{ l, y }_{ l - 1 }, v_{ l - 1 } ( X^{ l, y }_{ l - 1 } ) ) }
		\!\,\big( X^{ k, x }_{ l }( \P ) \big) ( dy ). 
		\end{split}
		\end{equation}
	Combining this with \eqref{exact_solution_alternative_representation:dynamics} yields for all 
		$ l \in \{ 1, 2, \ldots, k \} $, 
		$ x \in \mc O $ 
	that 
		\begin{equation} 
		\Exp{ f_{ l - 1 } ( X^{ k, x }_{ l - 1 }, v_{ l - 1 } ( X^{ k, x }_{ l - 1 } ) ) } 
		= \int_{ \mc O } \Exp{ f_{ l - 1 } ( X^{ l, y }_{ l - 1 }, v_{ l - 1 } ( X^{ l, y }_{ l - 1 } ) ) }
		\!\,\big( X^{ k, x }_{ l }( \P ) \big)( dy ).
		\end{equation} 
	Note that \eqref{exact_solution_alternative_representation:iteration} hence implies for all 
		$ l \in \{ 1, 2, \ldots, k \} $, 
		$ x \in \mc O $ 
	that 
		\begin{equation} 
		\begin{split}
		\Exp{ f_{ l - 1 } ( X^{ k, x }_{ l - 1 }, v_{ l - 1 } ( X^{ k, x }_{ l - 1 } ) ) } 
		= 
		\int_{\R^d}
		v_l(y)\!\,\big(X^{k}_{l}(x)(\P)\big)(dy)
		= 
		\Exp{v_l(X^{k}_l(x))}\!.
		\end{split}
		\end{equation} 
	This, \eqref{exact_solution_alternative_representation:integrability}, and \eqref{exact_solution_alternative_representation:integrability_of_f_term} ensure that for all 
		$ l \in \{ 1, 2, \ldots, k \} $, 
		$ x \in \mc O $ 
	it holds that 
		$\EXP{ | v_{ l - 1 } ( X^{ k, x }_{ l - 1 } ) | + | f_{ l - 1 } ( X^{ k, x }_{ l - 1 }, v_{ l - 1 } ( X^{ k, x }_{ l - 1 } ) ) - v_{ l - 1 } ( X^{ k, x }_{ l - 1 } ) | } < \infty $ 
	and 
		\begin{equation} 	
		\begin{split}
		\EXPP{ v_l ( X^{ k, x }_{ l } ) } 
		= 
		\EXPP{ v_{ l - 1 } ( X^{ k, x }_{ l - 1 } ) } 
		+ \EXPP{ f_{ l - 1 } ( X^{ k, x }_{ l - 1 }, v_{ l - 1 } ( X^{ k, x }_{ l - 1 } ) ) - v_{ l - 1 } ( X^{ k, x }_{ l - 1 } ) }.	
		\end{split}
		\end{equation} 
	The fact that for all 
		$ x \in \mc O $ 
	it holds that
		$v_k(x) = \EXP{ v_k ( X^{ k, x }_{ k } ) } $ 
	and induction hence establish \eqref{exact_solution_alternative_representation:claim}. 
	This completes the proof of \cref{lem:exact_solution_alternative_representation}.
\end{proof} 

\begin{lemma}[A Gronwall lemma] \label{discrete_gronwall}
	Let 
		$ \alpha \in [0,\infty) $, 
		$ K \in \N $, 
	and let 
		$ L_0, L_1, \ldots, L_{K-1} \in [0,\infty)$, 
		$ A_0, A_1, \ldots, A_{K} \in [0,\infty)$
	satisfy for all 
		$ k \in \{0,1,\ldots,K\} $ 
	that 
	\begin{equation} \label{discrete_gronwall:ass}
	 A_k \leq \alpha + \sum_{l=0}^{k-1} L_l A_l. 
	\end{equation}  
	Then it holds for all 
		$ k\in\{0,1,\ldots,K\} $ 
	that 
		\begin{equation} 
		\label{discrete_gronwall:claim}
		A_k 
		\leq 
		\alpha \prod_{l=0}^{k-1} (1 + L_l)
		\leq 
		\alpha \exp\!\left(\sum_{l=0}^{k-1} L_l\right)\!. 
		\end{equation} 
\end{lemma}

\begin{proof}[Proof of \cref{discrete_gronwall}]
	Throughout this proof let 
		$ \A \subseteq \{0,1,\ldots,K\} $ 
	satisfy
		\begin{equation} 
		\A = \left\{ k\in\{0,1,\ldots,K \} \colon \left( \begin{array}{c} \text{For all}~ s \in \{0,1,\ldots,k\}~\text{it holds that}
		\\
		A_s \leq \alpha \prod_{l=0}^{s-1} (1+L_s)
		\end{array} \right)
		\right\}\!.
		\end{equation} 
	Note that the fact that 
		$A_0 \leq \alpha = \alpha \prod_{l=0}^{-1} (1+L_l) $ 
	assures that 
		$ 0 \in \A $. 
	Next observe that \eqref{discrete_gronwall:ass} and the fact that for all 
		$ l \in \{ 0, 1, \ldots, K \} $
	it holds that 
		$ \prod_{s=0}^{l} ( 1 + L_s ) - \prod_{s=0}^{l-1} ( 1 + L_s ) = 
		L_l \prod_{s=0}^{l-1} ( 1 + L_s ) $
	ensure that for all 
		$ k \in \{0,1,\ldots,K-1\} \cap \A $ 
	it holds that 
		\begin{equation} 
		\begin{split}
		A_{k+1} 
		& 
		\leq
		\alpha + 
		\sum_{s=0}^{k} 
		L_s \,\alpha \prod_{l=0}^{s-1} ( 1 + L_s )
		=
		\alpha 
		+ 
		\alpha  
		\sum_{s=0}^{k} \left[ 
	    \prod_{l=0}^{s} ( 1 + L_l ) - \prod_{l=0}^{s-1} (1 + L_l ) \right]
		\\
		& 
		= 
		\alpha 
		+ 
		\alpha \left[ \left( \prod_{l=0}^{k} ( 1 + L_l ) \right) - 1 \right] 
		= 
		\alpha \prod_{l=0}^{k} ( 1 + L_l ).  
		\end{split}  
		\end{equation} 
	This demonstrates that for all 
		$ k \in \{ 0, 1, \ldots, K - 1 \} \cap \A $ 
	it holds that 
		$ k + 1 \in \A $. 
	Induction and the fact that $ 0 \in \A $ hence demonstrate that for all 
		$ k \in \{ 0, 1, \ldots, K \} $ 
	it holds that 
		\begin{equation} 
		A_k \leq \alpha \prod_{ l = 0 }^{ k - 1 } ( 1 + L_l ). 
		\end{equation} 
	Combining this with the fact that for all $ t \in \R $ it holds that $ 1 + t \leq e^t $ establishes \eqref{discrete_gronwall:claim}. 
	This completes the proof of \cref{discrete_gronwall}.
\end{proof}

\begin{lemma} \label{lem:a_priori_estimates}
	Let $ d,K \in \N $, 
		$ L_0, L_1, \ldots, L_{K}  \in \R $, 
		$ \mc O \in \Borel( \R^d ) \setminus \{ \emptyset \} $, 
	let $ f_k \colon \mc O \times \R \to \R $, $ k \in \{ 0, 1, \ldots, K \} $, be $ \Borel( \mc O \times \R ) $/$ \Borel( \R ) $-measurable, 
	let $ g \colon \mc O \to \R $ be $ \Borel( \mc O ) $/$ \Borel(\R) $-measurable, 
	assume for all 
		$ k \in \{ 0, 1, \ldots, K \} $,
		$ x \in \mc O $, 
		$ a,b \in \R $ 
	that 
		$ | ( f_k ( x, a ) - a ) - ( f_k ( x, b ) - b ) | \leq L_k | a - b | $,  
	let $ ( S, \mc S ) $ be a measurable space, 
	let $ \phi_{ k } \colon \mc O \times S \to \mc O $, $ k \in \{ 0, 1, \ldots, K \} $, be $ ( \Borel( \mc O) \otimes \mc S ) $/$ \Borel( \mc O ) $-measurable, 
	let $ ( \Omega, \mc F, \P ) $ be a probability space, 
	let $ W_{ k } \colon \Omega \to S $, $ k \in \{ 0, 1, \ldots, K \} $, be independent random variables, 
	for every 
		$ k \in \{ 0, 1, \ldots, K \} $ 
	let 
		$ X^k = ( X^{ k, x }_{ l } )_{ ( l, x ) \in \{ 0, 1, \ldots, k \} \times \mc O } \colon \{ 0, 1, \ldots, k \} \times \mc O \times \Omega \to \mc O $
	satisfy for all 
		$ l \in \{0,1,\ldots,k\} $, 
		$ x \in \mc O $ 
	that 
		\begin{equation} \label{a_priori_estimates:dynamics}
		X^{ k, x }_{ l } 
		= 
		\begin{cases} 
		x & \colon l = k \\
		\phi_{ l } ( X^{ k, x }_{ l + 1 }, W_{ l } ) & \colon l<k, 
		\end{cases} 
		\end{equation} 
	assume for all 
		$ k \in \{0,1,\ldots,K\} $, 
		$ x \in \mc O $
	that 
		$ \EXP{ | g ( X^{ k, x }_{ 0 } ) |^2 + \sum_{ l = 0 }^{ k - 1 } | f_{ l } ( X^{ k, x }_{ l }, 0 ) |^2 } < \infty $, 
	and	let $ v_{ 0 }, v_{ 1 }, \ldots, v_{ K } \colon \mc O \to \R $ be $ \Borel( \mc O ) $/$ \Borel( \R ) $-measurable functions which satisfy for all 
		$ k \in	\{ 1, 2, \ldots, K \} $, 
		$ x \in \mc O $ 
	that 
		$ v_{ 0 } ( x ) = g ( x ) $ 
	and 
		\begin{equation} \label{a_priori_estimates:iteration}
		v_{ k } ( x ) = \EXP{ f_{ k - 1 }( X^{ k, x }_{ k - 1 }, v_{ k - 1 }( X^{ k, x }_{ k - 1 } ) ) }
		\end{equation} 
	(cf.\,\cref{lem:welldefinedness_exact_solution}). 
	Then it holds for all 
		$ k \in \{ 0, 1, \ldots, K \} $, 
		$ x \in \mc O $ 
	that 
		\begin{equation} \label{a_priori_estimates:claim}
		\begin{split}
		& 
		\left( \EXPP{ | v_k( X^{ K, x }_{ k } ) |^2 } \right)^{\!\nicefrac12}
		\leq \exp\!\left( \sum_{ l = 0 }^{ K - 1 } L_{ l } \right) \!\left[ \left( \EXPP{ | g ( X^{ K, x }_{ 0 } ) |^2 } \right)^{\!\nicefrac12} + \sum_{ l = 0 }^{ K - 1 } \left( \EXPP{ | f_{ l } ( X^{ K, x }_{ l }, 0 ) |^2 } \right)^{\!\nicefrac12} \right]\!.
		\end{split}
		\end{equation} 
\end{lemma}

\begin{proof}[Proof of \cref{lem:a_priori_estimates}]
	First, observe that \cref{lem:welldefinedness_exact_solution} ensures that for all 
	 	$ k \in \{ 0, 1, \ldots, K \} $, 
	 	$ x \in \mc O $  
	it holds that 
		\begin{equation} \label{a_priori_estimates:L2_bound}
		\EXPP{ | v_{ k } ( X^{ K, x }_{ k } ) |^2 } < \infty.
		\end{equation} 
	Moreover, note that \cref{lem:exact_solution_alternative_representation} yields that for all 
 		$ k \in \{ 0, 1, \ldots, K \} $, 
 		$ x \in \mc O $ 
	it holds that 
		\begin{equation} \label{a_priori_estimates:fixed_point_equation}
		v_{ k } ( x ) 
		= \EXPP{ g ( X^{ k, x }_{ 0 } ) } + \sum_{ l = 0 }^{ k - 1 } \EXPP{ f_{ l }( X^{ k, x }_{ l }, v_{ l }( X^{ k, x }_{ l } ) ) - v_l( X^{ k, x }_{ l } ) }. 
		\end{equation} 
	Minkowski's inequality therefore guarantees for all 
		$ k \in \{ 0, 1, \ldots, K \} $, 
		$ x \in \mc O $ 
	that 
		\begin{equation} 
		\begin{split}
		& 
		\left( \EXPP{ | v_{ k } ( X^{ K, x }_{ k } ) |^2} \right)^{\!\nicefrac12} 
		= 
		\left[ \int_{ \mc O } | v_{ k } ( y ) |^2 \, \big( X^{ K, x }_{ k } (\P) \big)( dy ) \right]^{\nicefrac12}
		\\
		&
		\leq \left[ \int_{ \mc O } \big| \EXP{ g ( X^{ k, y }_{ 0 } ) } \big|^2 \, \big( X^{ K, x }_{ k } ( \P ) \big)( dy ) 	\right]^{\nicefrac12}
		\\
		& 
		+ \sum_{ l = 0 }^{ k - 1 } \left[ \int_{ \mc O } 	\big| \EXPP{ f_{ l } ( X^{ k, y }_{ l } ( y ), v_{ l } ( X^{ k, y }_{ l } ) ) - v_{ l } ( X^{ k, y }_{ l } ) } \big|^2\,\big( X^{ K, x }_{ k } (\P) \big) ( dy )
		\right]^{\nicefrac12}. 
		\end{split}
		\end{equation} 
	Jensen's inequality hence yields for all 
		$ k \in \{ 0, 1, \ldots, K \} $, 
		$ x \in \mc O $ 
	that 
		\begin{equation} \label{a_priori_estimates:eq04}
		\begin{split}
		& 
		\left(\EXPP{ | v_{ k }( X^{ K, x }_{ k } ) |^2 } \right)^{\!\nicefrac12} 
		\leq 
		\left[ \int_{ \mc O } \EXPP{ | g( X^{ k, y }_{ 0 } ) |^2 }\!\,\big( X^{ K, x }_{ k }(\P) \big) ( dy ) 	\right]^{\nicefrac{1}{2}}\\ 
		& + 
		\sum_{ l = 0 }^{ k - 1 } 
		\left[ \int_{ \mc O } \EXPPP{ \big| f_{ l } ( X^{ k, y }_{ l },  v_{ l } ( X^{ k, y }_{ l } ) ) 
		 	- v_{ l } ( X^{ k, y }_{ l } ) \big|^2 }\!\,\big( X^{ K, x }_{ k } ( \P ) \big) ( dy )
		\right]^{\nicefrac12}.
		\end{split}
		\end{equation}  
	Next observe that item~\eqref{elementary_measurability_X_processes:item1} of \cref{lem:elementary_measurability_X_processes} yields that for all 
		$ k \in \{ 0, 1, \ldots, K \} $, 
		$ x \in \mc O $ 
	it holds that 
		$ X^{ K, x }_{ k } \colon \Omega \to \mc O $ 
	is 
		$ \sigma( ( W_m )_{ m \in [ k, K ) \cap \N_0 } ) $/$ \Borel( \mc O ) $-measurable. 
	This, the fact that for all 
		$ k \in \{ 0, 1, \ldots, K \} $ 
	it holds that 
		$ \mc O \times \Omega \ni ( x, \omega ) \mapsto  g( X^{ k, x }_{ 0 }( \omega ) ) \in \R $ 
	is 
		$ ( \Borel( \mc O ) \otimes \sigma( ( W_m )_{ m \in [ 0, k ) \cap \N_0 } ) )$/$ \Borel( \R ) $-measurable, 
	the assumption that 
		$ W_{ l } \colon \Omega \to S $, $ l \in \{ 0, 1, \ldots, K-1 \} $, 
	are independent random variables, and Hutzenthaler et al.~\cite[Lemma 2.2]{Overcoming} therefore demonstrate that for all 
		$ k \in \{ 0, 1, \ldots, K \} $, 
		$ x \in \mc O $ 
	it holds that 
		\begin{equation} \label{a_priori_estimates:eq05}
		\int_{ \mc O } \EXPP{ | g ( X^{ k, y }_{ 0 } ) |^2 } \big( X^{ K, x }_{ k }( \P ) \big)( dy )
		= \Exp{ \Big| g \big( X^{ k, X^{ K, x }_{ k } }_{ 0 } \big) \Big|^2}\!.
		\end{equation} 
	Moreover, note that the fact that for all 
		$ k,l \in \N_0 $, 
		$ x \in \mc O  $ 
	with 
		$ l \leq k \leq K $ 
	it holds that $ X^{ k, x }_{ l } \colon \Omega \to \mc O  $ is $ \sigma( ( W_{ m } )_{ m \in [ l, k ) \cap \N_0 } ) $/$ \Borel(\mc O ) $-measurable (cf.~item~\eqref{elementary_measurability_X_processes:item1} of \cref{lem:elementary_measurability_X_processes}), 
	the assumption that 
		$ W_l \colon \Omega\to S $, $ l \in \{ 0, 1, \ldots, K - 1 \} $, 
	are independent random variables, and Hutzenthaler et al.~\cite[Lemma 2.2]{Overcoming} demonstrate that for all 
	 	$ k \in \{ 0, 1, \ldots, K \} $, 
	 	$ l \in \{ 0, 1, \ldots, k - 1 \} $,
	 	$ x \in \mc O $ 
	it holds that 
		\begin{equation}
		\begin{split}
		& 
		\sum_{ l = 0 }^{ k - 1 } \int_{ \mc O } \EXPP{ | f_{ l } ( X^{ k, y }_{ l }, v_{ l } ( X^{ k, y }_{ l } ) ) - 		v_{ l } ( X^{ k, y }_{ l } ) |^2 }\!\,\big( X^{ K, x }_{ k } ( \P ) \big)( dy )
	 	\\& = 
	 	\sum_{ l = 0 }^{ k - 1 } \Exp{ \Big| f_{ l } ( X^{ k, X^{ K, x }_{ k } }_{ l }, v_{ l } ( X^{ k, X^{ K, x }_{ k } }_{ l })) - v_{ l } ( X^{ k, X^{ K, x }_{ k } }_{ l } ) \Big|^2}.
	 	\end{split}
	 	\end{equation}
	This, \eqref{a_priori_estimates:dynamics}, \eqref{a_priori_estimates:eq04}, and \eqref{a_priori_estimates:eq05}  ensure for all 
	 	$ k \in \{ 0, 1, \ldots, K \} $, 
	 	$ x \in \mc O $ 
	that 
		\begin{equation} 
		\begin{split}
		& \left( \EXPP{ | v_{ k } ( X^{ K, x }_{ k } ) |^2 } \right)^{\!\nicefrac12} 
		\\
		& \leq \left( \Exp{ \Big| g \big( X^{ k, X^{ K, x }_{ k } }_{ 0 } \big) \Big|^2 } \right)^{\!\!\nicefrac12}
		+
		\sum_{ l = 0 }^{ k - 1 } 
		\left( \Exp{ \Big| f_l \big( X^{ k, X^{ K, x }_{ k } }_{ l } ), v_{ l } ( X^{ k, X^{ K, x }_{ k } }_{ l } ) \big) - v_{ l } ( X^{ k, X^{ K, x }_{ k } }_{ l } ) \Big|^2}
		\right)^{\!\nicefrac12} 
		\\
		& = 
		\left( \EXPP{ | g ( X^{ K, x }_{ 0 } ) |^2 } \right)^{\!\nicefrac12}
		+ \sum_{ l = 0 }^{ k - 1 } \left( \EXPP{ | f_{ l } ( X^{ K, x }_{ l }, v_{ l } ( X^{ K, x }_{ l } ) ) - v_{ l } ( X^{ K, x }_{ l } ) |^2 } \right)^{\!\nicefrac12}.
		\end{split}
		\end{equation} 
	Minkowski's inequality and the assumption that for all 
	 	$ k \in \{ 0, 1, \ldots, K - 1 \} $,
		$ a,b \in \R $, 
		$ x \in \mc O $ 
	it holds that 
	 	$ | ( f_k ( x, a ) - a ) - ( f_k ( x, b ) - b ) | \leq L_k | a - b | $ 
	hence guarantee that for all 
	 	$ k \in \{ 0, 1, \ldots, K \} $, 
		$ x \in \mc O $
	it holds that 
		\begin{equation}
		\begin{split}
		&
		\left( \EXPP{ | v_{ k } ( X^{ K, x }_{ k } ) |^2 } \right)^{\!\nicefrac12}
		\\[1ex] 
		& \leq 
		\left( \EXPP{ | g ( X^{ K, x }_{ 0 } ) |^2 } \right)^{\!\nicefrac12}
		+ \sum_{ l = 0 }^{ k - 1 } \left[ \left( \EXPP{ | f_{ l } ( X^{ K, x }_{ l }, 0 ) |^2 } \right)^{\!\nicefrac12} 
		+ L_l \left( \EXPP{ | v_{ l } ( X^{ K, x }_{ l } ) |^2 } \right)^{\!\nicefrac12} \right]
		\\ 
		&
		=
		\left[ \left( \EXPP{ | g( X^{ K, x }_{ 0 } ) |^2 } \right)^{\!\nicefrac12}
			+ \sum_{ l = 0 }^{ K - 1 } \left( \EXPP{ | f_{ l } ( X^{ K, x }_{ l }, 0 ) |^2} \right)^{\!\nicefrac12} \right]
		+ \sum_{ l = 0 }^{ k - 1 } L_{ l } \!\left( \EXPP{ | v_{ l } ( X^{ K, x }_{ l } ) |^2}\right)^{\!\nicefrac12} .
		\end{split}  
		\end{equation}
	\cref{discrete_gronwall} 
	and \eqref{a_priori_estimates:L2_bound} hence ensure for all 
	 	$ k \in \{ 0, 1, \ldots, K \} $, 
	 	$ x \in \mc O $ 
	that 
		\begin{equation} 
		\begin{split}
		\left( \EXPP{ | v_{ k } ( X^{ K, x }_{ k } ) |^2 } \right)^{\!\nicefrac12} 
		\leq 
		\exp\!\left(\sum_{ l = 0 }^{ k - 1 } L_{ l } \right)\!
		\left[ \left( \EXPP{ | g( X^{ K, x }_{ 0 } ) |^2 } \right)^{\!\nicefrac12} 
		+ 
		\sum_{ l = 0 }^{ K - 1 } \left( \EXPP{ | f_{ l } ( X^{ K, x }_{ l }, 0 ) |^2 } \right)^{\!\nicefrac12}
		\right]\!.
		\end{split}
		\end{equation} 
	This establishes \eqref{a_priori_estimates:claim}. 
	This completes the proof of \cref{lem:a_priori_estimates}.
\end{proof}

\subsection{Bias and variance estimates for MLP approximations} 
\label{subsec:bias_and_variance_estimates}

\begin{lemma} \label{lem:mean}
	Assume \cref{setting}, 
	let $ c, L_0, L_1, \ldots, L_{ K } \in [ 0, \infty ) $	satisfy for all 
		$ k \in \{ 0, 1, \ldots, K \} $, 
		$ x \in \mc O $, 
		$ a,b \in \R $ 
	that 	 
		$ | ( f_k ( x, a ) - a ) - ( f_k ( x, b ) - b ) | \leq L_k | a - b | $, 
	assume for all
		$ k,l \in \N_0 $ 
	with 
		$ l < k \leq K $
	that 
		$ L_l \leq c \P( \mc R^{0}_k = l ) $,
	and assume for all 
		$ k \in \{ 0, 1, \ldots, K \} $, 
		$ x \in \mc O $
	that 
		$ \EXP{ | g ( X^{ 0, k, x }_{ 0 } ) |^2 + \sum_{ l = 0 }^{ k - 1 } | f_{ l } ( X^{ 0, k, x }_{ l }, 0 ) |^2 } < \infty $. 
	Then it holds for all 
		$ k \in \{ 1, 2, \ldots, K \} $, 
		$ n \in \N $, 
		$ x \in \mc O $ 
	that 
		$ \EXP{ | g ( X^{ 0, k, x }_{ 0 } ) | } + \sum_{ l = 0 }^{ k - 1 } \EXP{ |f_{ l } ( X^{ 0, k, x }_{ l }, V^{ 0 }_{ l, n - 1 } ( X^{ 0, k, x }_{ l } ) ) - V^{ 0 }_{ l, n - 1 } ( X^{ 0, k, x }_{ l } ) | } < \infty $ 
	and 
		\begin{equation} \label{mean:claim}
		\EXPP{ V^{ 0 }_{ k, n } ( x ) } 
		= \EXPP{ g ( X^{ 0, k, x }_{ 0 } ) } + \sum_{ l = 0 }^{ k - 1 } \EXPP{ f_{ l } ( X^{ 0, k, x }_{ l },  V^{ 0 }_{ l, n - 1 }( X^{ 0, k, x }_{ l } ) ) - V^{ 0 }_{ l, n - 1 } ( X^{ 0, k, x }_{ l } ) }.
		\end{equation} 
\end{lemma}

\begin{proof}[Proof of \cref{lem:mean}] 
	First, observe that \cref{lem:integrability} guarantees that for all 
		$ k,l,n \in \N_0 $, 
		$ x \in \mc O $ 
	with 
		$ l \leq k \leq K $ 
	it holds that 
		$ \EXP{ | V^{ 0 }_{ l, n } ( X^{ 0, k, x }_{ l } ) |^2 } < \infty $. 
	This, the assumption that for all 
	 	$ k \in \{ 0, 1, \ldots, K - 1 \} $, 
	 	$ x \in \mc O $, 
	 	$ a,b \in \R $
	it holds that 
	 	$ | ( f_{ k } ( x, a ) - a ) - ( f_{ k } ( x, b ) - b ) | \leq L_{ k } | a - b | $, 
	the assumption that for all 
		$ k \in \{ 0, 1, \ldots, K \} $, 
		$ x \in \mc O $ 
	it holds that 
		$ \EXP{ | g ( X^{ 0, k, x }_{ 0 } ) |^2 + \sum_{ l = 0 }^{ k - 1 } | f_{ l } ( X^{ 0, k, x }_{ l }, 0 ) |^2 } < \infty $, 
	and \eqref{setting:mlp_scheme} ensure that for all 
	 	$ k \in \{ 1, 2, \ldots, K \} $, 
	 	$ n \in \N $, 
	 	$ x \in \mc O $ 
	it holds that 
	 	$ \EXP{ | g ( X^{ 0, k, x }_{ 0 } ) | + \sum_{ l = 0 }^{ k - 1 } | f_{ l } ( X^{ 0, k, x }_{ l }, V^{ 0 }_{ l, n - 1 } ( X^{ 0, k, x }_{ l } ) ) - V^{ 0 }_{ l, n - 1 } ( X^{ 0, k, x }_{ l } ) | } < \infty $ 
	and 
	 	\begin{equation} \label{mean:eq01}
	 	\begin{split}
		& 
	  	\Exp{ V^{ 0 }_{ k, n } ( x ) } 
	  	= \frac{ 1 }{ M^n } \sum_{ m = 1 }^{ M^n } \Exp{ g ( X^{ ( 0, 0, -m ), k, x }_{ 0 } ) 
	  	+ \sum_{ l = 0 }^{ k - 1 } f_{ l } ( X^{ ( 0, 0, m ), k, x }_{ l }, 0 ) }
	  	+ 
	  	\sum_{ j = 1 }^{ n - 1 } 
	  	\sum_{ m = 1 }^{ M^{ n - j } } 
	  	\frac{ 1 }{ M^{ n - j } }
		\\  
	  	& 
	    \cdot \Bigg[ 
	    \EXPPPP{ \frac{ 1 }{ \mf p_{ k, \mc{R}^{ ( 0, j, m ) }_{ k } } } \Big( f_{ \mc{R}^{ ( 0, j, m ) }_{ k } } \Big( X^{ ( 0, j, m ), k, x }_{ \mc{R}^{ ( 0, j, m ) }_{ k } },  V^{ ( 0, j, m ) }_{ \mc{R}^{ ( 0, j, m ) }_{ k }, j } \big( X^{ ( 0, j, m ), k, x }_{ \mc{R}^{ ( 0, j, m ) }_{ k } } \big) \Big) 
	  	-
	  	V^{ ( 0, j, m ) }_{ \mc{R}^{ ( 0, j, m ) }_{ k }, j } \big( X^{ ( 0, j, m ), k, x }_{ \mc{R}^{ ( 0, j, m ) }_{ k } } \big) \Big) 
	  	\\ 
	  	& - 
		\frac{ 1 }{ \mf p_{ k,  \mc{R}^{ ( 0, j, m ) }_{ k } } } \Big( f_{ \mc{R}^{ ( 0, j, m ) }_{ k } } \Big( X^{ ( 0, j, m ), k, x }_{ \mc{R}^{ ( 0, j, m ) }_{ k } },   V^{ ( 0, j, -m ) }_{ \mc{R}^{ ( 0, j, m) }_{ k }, j - 1 } \big( X^{ ( 0, j, m ), k, x }_{ \mc{R}^{ ( 0, j, m ) }_{ k } } \big) \Big) - V^{ ( 0, j, -m ) }_{ \mc{R}^{ ( 0, j, m ) }_{ k }, j - 1 } \big( X^{ ( 0, j, m ), k, x }_{ \mc{R}^{ ( 0, j, m ) }_{ k } } \big) \Big) }
	  	\Bigg]. 
	 	\end{split}
	 	\end{equation}
	Next note that the fact that for all 
		$ \theta \in \Theta $, 
		$ k \in \{ 1, 2, \ldots, K \} $ 
	it holds that 
		$ \P( \mc{R}^{ \theta }_{ k } \in \{ 0, 1, \ldots, k - 1\} ) = 1 $
	ensures that for all 
		$ n,m \in \N $, 
		$ k \in \{1,2,\ldots,K\} $, 
		$ j \in \{1,2,\ldots,n-1\} $, 
		$ x \in \mc O $ 
	it holds that  
		\begin{equation} 
		\begin{split}
		&
		\Exp{ \frac{ 1 }{ \mf p_{ k, \mc{R}^{ ( 0, j, m ) }_{ k } } } \Big( f_{ \mc{R}^{ ( 0, j, m ) }_{ k } } \big( X^{ ( 0, j, m ), k, x }_{ \mc{R}^{ ( 0, j, m ) }_{ k } }, V^{ ( 0, j, m ) }_{ \mc{R}^{ ( 0, j, m ) }_{ k }, j } ( X^{ ( 0, j, m ), k, x }_{ \mc{R}^{ ( 0, j, m ) }_{ k } } ) \big) 
		- V^{ ( 0, j, m ) }_{ \mc{R}^{ ( 0, j, m ) }_{ k }, j } ( X^{ ( 0, j, m ), k, x }_{ \mc{R}^{ ( 0, j, m ) }_{ k } } ) \Big) }
		\\
		& - 
		\Exp{ \frac{ 1 }{ \mf p_{ k, \mc{R}^{ ( 0, j, m ) }_{ k } } } \Big( f_{ \mc{R}^{ ( 0, j, m ) }_{ k } } \big( X^{ ( 0, j, m ), k, x }_{ \mc{R}^{ ( 0, j, m ) }_{ k } }, V^{ ( 0, j, -m ) }_{ \mc{R}^{ ( 0, j, m ) }_{ k }, j - 1 } ( X^{ ( 0, j, m ), k, x }_{ \mc{R}^{ ( 0, j, m ) }_{ k }} ) \big) - V^{ ( 0, j, -m ) }_{ \mc{R}^{ ( 0, j, m ) }_{ k }, j - 1 } ( X^{ ( 0, j, m ), k, x }_{ \mc{R}^{ ( 0, j, m ) }_{ k } } ) \Big) }
		\\
		& = 
		\sum_{ l = 0 }^{ k - 1 } 
		\Exp{ \frac{ \mathbbm{ 1 }_{ \{ \mc{R}^{ ( 0, j, m ) }_{ k } = l \} } }{ \mf p_{ k, l } }  \left( 
		f_{ l } ( X^{ ( 0, j, m ), k, x }_{ l }, V^{ ( 0, j, m ) }_{ l, j } ( X^{ ( 0, j, m ), k, x }_{ l } ) ) - 
		V^{ ( 0, j, m ) }_{ l, j } ( X^{ ( 0, j, m ), k, x }_{ l } ) \right) }
		\\
		& - 
		\sum_{ l = 0 }^{ k - 1 } 
		\Exp{ \frac{ \mathbbm{ 1 }_{ \{ \mc{R}^{ ( 0, j, m ) }_{ k } = l \} } }{ \mf p_{ k, l } } \left( 			
		f_{ l } ( X^{ ( 0, j, m ), k, x }_{ l }, V^{ ( 0, j, -m ) }_{ l, j - 1 }( X^{ ( 0, j, m ), k, x }_{ l } ) ) - 
		V^{ ( 0, j, -m ) }_{ l, j - 1 } ( X^{ ( 0, j, m ), k, x }_{ l } ) \right) }\!.
		\end{split}
		\end{equation}
	Item~\eqref{elementary_measurability_X_processes:item1} of \cref{lem:elementary_measurability_X_processes}, 	item~\eqref{measurability:item1} of \cref{lem:measurability}, and the fact that for all 
		$ n, m \in \N $, 
		$ j \in \{ 1, 2, \ldots, n - 1 \} $
	it holds that 
		$ \sigma ( ( \mc{R}^{ ( 0, j, m ) }_{ s } )_{ s \in \{1, 2, \ldots, K\} } ) $ 
	and 
		$ \sigma( 
			( \mc{R}^{ ( 0, j, m, \vartheta ) }_{ s } )_{ ( s, \vartheta ) \in \{ 1, 2, \ldots, K \} \times \Theta },
			( \mc{R}^{ ( 0, -j, m, \vartheta ) }_{ s } )_{ ( s, \vartheta ) \in \{1, 2, \ldots, K \} \times \Theta },  \allowbreak
			( W^{ ( 0, j, m, \vartheta ) }_{ s } )_{ ( s, \vartheta ) \in \{ 0, 1, \ldots, K - 1 \} \times \Theta },
			( W^{ ( 0, j, -m, \vartheta ) }_{ s } )_{ ( s, \vartheta ) \in \{ 0, 1, \ldots, K - 1 \} \times \Theta }, \allowbreak 
			\allowbreak			 
			( W^{ ( 0, j, m ) }_{ s })_{ s \in \{ 0, 1, \ldots, K - 1 \}}
			)
		$
	are independent therefore ensure that for all 
		$ n,m \in \N $, 
		$ k \in \{ 1, 2, \ldots, K \} $, 
		$ j \in \{ 1, 2, \ldots, n - 1 \} $, 
		$ x \in \mc O $ 
	it holds that 
		\begin{equation}
		\begin{split}
		&
		\Exp{ \frac{ 1 }{ \mf p_{ k, \mc{R}^{ ( 0, j, m ) }_{ k  } } } \Big( f_{ \mc{R}^{ ( 0, j, m ) }_{ k } } \big( X^{ ( 0, j, m ), k, x }_{ \mc{R}^{ ( 0, j, m ) }_{ k } }, V^{ ( 0, j, m ) }_{ \mc{R}^{ ( 0, j, m ) }_{ k } , j } ( X^{ ( 0, j, m ), k, x }_{ \mc{R}^{ ( 0, j, m ) }_{ k } } ) \big) 
		- V^{ ( 0, j, m ) }_{ \mc{R}^{ ( 0, j, m ) }_{ k }, j } ( X^{ ( 0, j, m ), k, x }_{ \mc{R}^{ ( 0, j, m ) }_{ k } } ) \Big) }
		\\
		& - 
		\Exp{ \frac{ 1 }{ \mf p_{ k, \mc{R}^{ ( 0, j, m ) }_{ k } } } \Big( f_{ \mc{R}^{ ( 0, j, m ) }_{ k } } \big( X^{ ( 0, j, m ), k, x }_{ \mc{R}^{ ( 0, j, m ) }_{ k } },  V^{ ( 0, j, -m ) }_{ \mc{R}^{ ( 0, j, m ) }_{ k } , j - 1 } ( X^{ ( 0, j, m ), k, x }_{ \mc{R}^{ ( 0, j, m ) }_{ k } } ) \big) 
		- V^{ ( 0, j, -m ) }_{ \mc{R}^{ ( 0, j, m ) }_{ k }, j - 1 } ( X^{ ( 0, j, m ), k, x }_{ \mc{R}^{ ( 0, j, m ) }_{ k } } ) \Big) }
		\\
		& = 
		\sum_{ l = 0 }^{ k - 1 } 
		\frac{ \P\big( \mc{R}^{ ( 0, j, m ) }_{ k } = l \big) }{ \mf p_{ k, l } } 
		\Bigg[ \Exp{ f_{ l } ( X^{ ( 0, j, m ), k, x }_{ l }, V^{ ( 0, j, m ) }_{ l, j } ( X^{ ( 0, j, m ), k, x }_{ l } ) ) - 
			V^{ ( 0, j, m ) }_{ l, j } ( X^{ ( 0, j, m ), k, x }_{ l } ) }
		\\
		& - \Exp{ f_{ l } ( X^{ ( 0, j, m ), k, x }_{ l }, V^{ ( 0, j, -m ) }_{ l, j - 1 } ( X^{ ( 0, j, m ), k, x }_{ l } ) ) - 
			V^{ ( 0, j, -m ) }_{ l, j - 1 } ( X^{ ( 0, j, m ), k, x }_{ l } )  } \Bigg].
		\end{split}
		\end{equation}
	The fact that for all 
		$ \theta \in \Theta $,
		$ k \in \{ 1, 2, \ldots, K \} $, 
		$ l \in \{ 0, 1, \ldots, k - 1 \} $ 
	with $ \P( \mc{R}^{ \theta }_{ k } = l ) > 0 $ it holds that $ \P( \mc{R}^{ \theta }_{ k } = l ) = \mf p_{ k, l } $ 
	and the fact that for all
		$ \theta \in \Theta $, 
		$ k \in \{ 1, 2, \ldots, K \} $,
		$ l \in \{ 0, 1, \ldots, k - 1 \} $ 
	with $ \P( \mc R^{\theta}_k = l ) = 0 $ it holds that 
		$ L_l = 0 $ 
	hence yield that for all 
		$ n,m \in \N $, 
		$ k \in \{ 1, 2, \ldots, K \} $, 
		$ j \in \{ 1, 2, \ldots, n - 1 \} $, 
		$ x \in \mc O $ 
	it holds that 
		\begin{equation} \label{mean:eq04}
		\begin{split}
		&
		\Exp{ \frac{ 1 }{ \mf p_{ k, \mc{R}^{ ( 0, j, m ) }_{ k } } } \Big( f_{ \mc{R}^{ ( 0, j, m ) }_{ k } } \big( X^{ ( 0, j, m ), k, x }_{ \mc{R}^{ ( 0, j, m ) }_{ k } },  V^{ ( 0, j, m ) }_{ \mc{R}^{ ( 0, j, m ) }_{ k }, j } ( X^{ ( 0, j, m ), k, x }_{ \mc{R}^{ ( 0, j, m ) }_{ k } } ) \big) 
		- V^{ ( 0, j, m ) }_{ \mc{R}^{ ( 0, j, m ) }_{ k }, j } ( X^{ ( 0, j, m ), k, x }_{ \mc{R}^{ ( 0, j, m ) }_{ k } } ) \Big) }
		\\
		& - 
		\Exp{ \frac{ 1 }{ \mf p_{ k, \mc{R}^{ ( 0, j, m ) }_{ k } } } \Big( f_{ \mc{R}^{ ( 0, j, m ) }_{ k } }\big( X^{ ( 0, j, m ), k, x }_{ \mc{R}^{ ( 0, j, m ) }_{ k } },  V^{ ( 0, j, -m ) }_{ \mc{R}^{ ( 0, j, m ) }_{ k }, j - 1 } ( X^{ ( 0, j, m ), k, x }_{ \mc{R}^{ ( 0, j, m ) }_{ k } } ) \big) 
		- V^{ ( 0, j, -m ) }_{ \mc{R}^{ ( 0, j, m ) }_{ k }, j - 1 } ( X^{ ( 0, j, m ), k, x }_{ \mc{R}^{ ( 0, j, m ) }_{ k } } ) \Big) } 
		\\
		& = 
		\sum_{ l = 0 }^{ k - 1 } 
		\Bigg[ 
			\Exp{ f_{ l } ( X^{ ( 0, j, m ), k, x }_{ l }, V^{ ( 0, j, m ) }_{ l, j } ( X^{ ( 0, j, m ), k, x }_{ l } ) ) - 
			V^{ ( 0, j, m ) }_{ l, j } ( X^{ ( 0, j, m ), k, x }_{ l } ) }
		\\
		& - \Exp{ f_{ l } ( X^{ ( 0, j, m ), k, x }_{ l }, V^{ ( 0, j, -m ) }_{ l, j - 1 } ( X^{ ( 0, j, m ), k, x }_{ l } ) ) - 
			V^{ ( 0, j, -m ) }_{ l, j - 1 } ( X^{ ( 0, j, m ), k, x }_{ l } ) }
		\Bigg].
		\end{split}
		\end{equation}
	This and \cref{cor:equidistribution_for_mean_lemma} assure that for all 
		$ n,m \in \N $, 
		$ k \in \{ 1, 2, \ldots, K \} $, 
		$ j \in \{ 0, 1, \ldots, n - 1 \} $, 
		$ x \in \mc O $ 
	it holds that 
		\begin{equation}
		\begin{split} 
		& 
		\Exp{ \frac{ 1 }{ \mf p_{ k, \mc{R}^{ ( 0, j, m ) }_{ k } } } \Big( f_{ \mc{R}^{ ( 0, j, m ) }_{ k } } ( X^{ ( 0, j, m ), k, x }_{ \mc{R}^{ ( 0, j, m ) }_{ k } }, V^{ ( 0, j, m ) }_{ \mc{R}^{ ( 0, j, m ) }_{ k }, j } ( X^{ ( 0, j, m ), k, x }_{ \mc{R}^{ ( 0, j, m ) }_{ k } } ) ) 
		- V^{ ( 0, j, m ) }_{ \mc{R}^{ ( 0, j, m ) }_{ k }, j }( X^{ ( 0, j, m ), k, x }_{ \mc{R}^{ ( 0, j, m ) }_{ k } } ) \Big) } 
		\\[1ex]
		& - 
		\Exp{ \frac{ 1 }{ \mf p_{ k, \mc{R}^{ ( 0, j, m ) }_{ k } } } \Big( f_{ \mc{R}^{ ( 0, j, m ) }_{ k } } ( X^{ ( 0, j, m ), k, x }_{ \mc{R}^{ ( 0, j, m ) }_{ k } },  V^{ ( 0, j, -m ) }_{ \mc{R}^{ ( 0, j, m ) }_{ k }, j - 1 } ( X^{ ( 0, j, m ), k, x }_{ \mc{R}^{ ( 0, j, m ) }_{ k } } ) ) - V^{ ( 0, j, -m ) }_{ \mc{R}^{ ( 0, j, m ) }_{ k }, j - 1 } ( X^{ ( 0, j, m ), k, x }_{ \mc{R}^{ ( 0, j, m ) }_{ k } } ) \Big) }
		\\ 
		& = 
		\sum_{ l = 0 }^{ k - 1 } 
		\EXPP{ f_{ l } ( X^{ 0, k, x }_{ l }, V^{ 0 }_{ l, j } ( X^{ 0, k, x }_{ l } ) ) - V^{ 0 }_{ l, j } ( X^{ 0, k, x }_{ l } ) } 
		\\
		& - 
		\sum_{ l = 0 }^{ k - 1 } 
		\EXPP{ f_{ l } ( X^{ 0, k, x }_{ l }, V^{ 0 }_{ l, j - 1 } ( X^{ 0, k, x }_{ l } ) ) - V^{ 0 }_{ l, j - 1 } ( X^{ 0, k, x }_{ l } ) }. 
		\end{split} 
		\end{equation}
	This, \eqref{mean:eq01}, and item~\eqref{elementary_equidistribution_X_processes:item1} of \cref{lem:elementary_equidistribution_X_processes} imply that for all 
		$ k \in \{ 1, 2, \ldots, K \} $, 
		$ n \in \N $, 
		$ x \in \mc O $
	it holds that 
		\begin{equation}
		\begin{split}
		& \Exp{ V^{ 0 }_{ k, n } ( x ) }  
		= \EXPP{ g ( X^{ 0, k, x }_{ 0 } ) } 
		+ \sum_{ l = 0 }^{ k - 1 } \EXPP{ f_{ l } ( X^{ 0, k, x }_{ l }, 0 ) } 
		\\ 
		& \quad + 
		\sum_{ j = 1 }^{ n - 1 } 
		\sum_{ l = 0 }^{ k - 1 } 
		 \EXPP{ 
			f_{ l } ( X^{ 0, k, x }_{ l }, V^{ 0 }_{ l, j } ( X^{ 0, k, x }_{ l } ) ) 
		- V^{ 0 }_{ l, j } ( X^{ 0, k, x }_{ l } ) }
		\\
		& 
		\quad -
		\sum_{ j = 1 }^{ n - 1 } 
		\sum_{ l = 0 }^{ k - 1 } 
		\EXPP{ f_{ l } ( X^{ 0, k, x }_{ l }, V^{ 0 }_{ l, j - 1 } ( X^{ 0, k, x }_{ l } ) ) - V^{ 0 }_{ l, j - 1 } ( X^{ 0, k, x }_{ l } ) }
		\\
		& = 
		\EXPP{ g ( X^{ 0, k, x }_{ 0 } ) } 
		+ \sum_{ l = 0 }^{ k - 1 } \EXPP{ f_{ l } ( X^{ 0, k, x }_{ l }, 0 ) } 
		\\ 
		& \quad + 
		\sum_{ l = 0 }^{ k - 1 } 
		\Exp{ 
			f_{ l } ( X^{ 0, k, x }_{ l }, V^{ 0 }_{ l, n - 1 } ( X^{ 0, k, x }_{ l } ) ) 
			- V^{ 0 }_{ l, n - 1 } ( X^{ 0, k, x }_{ l } ) }
		\\
		& \quad -
		\sum_{ l = 0 }^{ k - 1 } 
		\Exp{ f_{ l } ( X^{ 0, k, x }_{ l }, V^{ 0 }_{ l, 0 } ( X^{ 0, k, x }_{ l } ) ) - V^{ 0 }_{ l, 0 } ( X^{ 0, k, x }_{ l } ) }\!. 
		\end{split}
		\end{equation} 
	The fact that for all 
		$ k \in \{ 0, 1, \ldots, K \} $,
		$ x \in \mc O $ 
	it holds that 
		$ V^{ 0 }_{ k, 0 } ( x ) = 0 $ 
	therefore implies that for all 
		$ k \in \{ 1, 2, \ldots, K \} $, 
		$ n \in \N $, 
		$ x \in \mc O $ 
	it holds that 
		\begin{equation} 
		 \Exp{ V^{ 0 }_{ k, n } ( x ) } 
		 = \Exp{ g ( X^{ 0, k, x }_{ 0 } ) } 
		 + \sum_{ l = 0 }^{ k - 1 } \Exp{ f_{ l } \big( X^{ 0, k, x }_{ l }, V^{ 0 }_{ l, n - 1 }( X^{ 0, k, x }_{ l } ) \big) - V^{ 0 }_{ l, n - 1 } ( X^{ 0, k, x }_{ l } ) }\!.
		\end{equation} 
	This establishes \eqref{mean:claim}. 
	This completes the proof of \cref{lem:mean}.
\end{proof} 

\begin{lemma}[Bias estimate] \label{lem:bias_estimate}
	Assume 
		\cref{setting}, 
	let 
		$ c, L_{ 0 }, L_{ 1 }, \ldots, L_{ K } \in [ 0, \infty ) $ 
	satisfy for all
		$ k \in \{ 0, 1, \ldots, K \} $, 
		$ x \in \mc O $,  
		$ a,b \in \R $
	that 
		$ | ( f_{ k } ( x, a ) - a  ) - ( f_{ k } ( x, b ) - b ) | \leq L_{ k } | a - b | $, 
	assume for all 
		$ k, l \in \N_{ 0 } $ 
	with $ l < k \leq K $ that $ L_l \leq c \P( \mc R^{ 0 }_{ k } = l ) $, 
	assume for all
		$ k \in \{ 0, 1, \ldots, K \} $, 
		$ x \in \mc O $
	that 
		$ \EXP{ | g ( X^{ 0, k, x }_{ 0 } ) |^2 + \sum_{ l = 0 }^{ k - 1 } | f_{ l } ( X^{ 0, k, x }_{ l }, 0 ) |^2 } < \infty $, 
	and let $ v_{ k } \colon \mc O \to \R $, $ k \in \{ 0, 1, \ldots, K \} $, be $ \Borel( \mc O ) $/$ \Borel( \R ) $-measurable functions which satisfy for all 
		$ k \in \{ 1, 2, \ldots, K \} $,  
		$ x \in \mc O $ 
	that 
		$ v_{ 0 } ( x ) = g ( x ) $ 
	and 
		\begin{equation} \label{bias_estimate:exact_solution}
		v_{ k } ( x ) 
		= \EXPP{ f_{ k - 1 } ( X^{ 0, k, x }_{ k - 1 }, v_{ k - 1 } ( X^{ 0, k, x }_{ k - 1 } ) ) }
		\end{equation} 
	(cf.~\cref{lem:welldefinedness_exact_solution}). 
	Then it tholds for all 
		$ k \in \{ 1, 2, \ldots, K \} $, 
		$ n \in \N $, 
		$ x \in \mc O $ 
	that 
		\begin{equation} \label{bias_estimate:claim}
		 \left| \Exp{ V^{ 0 }_{ k, n } ( x ) } - v_{ k } ( x ) \right|^2
		 \leq  
		 \left( \sum_{ l = 0 }^{ k - 1 } L_{ l } \right) \!
		 \left( \sum_{ l = 0 }^{ k - 1 } L_{ l } \,\Exp{ \big| V^{ 0 }_{ l, n - 1 } ( X^{ 0, k, x }_{ l } ) - v_{ l } ( X^{ 0, k, x }_{ l } ) \big|^2 } \right)\!.
		\end{equation} 
\end{lemma}

\begin{proof}[Proof of \cref{lem:bias_estimate}]
	First, observe that 
		\cref{lem:exact_solution_alternative_representation} 
	ensures for all 
		$ k \in \{ 1, 2, \ldots, K \} $, 
		$ x \in \mc O $ 
	that 
		\begin{equation} \label{bias:exact_solution_formula}
		v_{ k } ( x ) = \EXPP{ g ( X^{ 0, k, x }_{ 0 } ) } + \sum_{ l = 0 }^{ k - 1 } \EXPP{ f_{ l } ( X^{ 0, k, x }_{ l }, v_{ l } ( X^{ 0, k, x }_{ l } ) ) - v_{ l } ( X^{ 0, k, x }_{ l } ) }. 
		\end{equation} 
	Moreover, note that \cref{lem:mean} proves that for all 
		$ k \in \{ 1, 2, \ldots, K \} $, 
		$ n \in \N $, 
		$ x \in \mc O $  
	it holds that 
		\begin{equation} 
		\EXPP{ V^{ 0 }_{ k, n } ( x ) } 
		= 
		\EXPP{ g ( X^{ 0, k, x }_{ 0 } ) } 
		+ \sum_{ l = 0 }^{ k - 1 } \EXPP{ f_{ l } ( X^{ 0, k, x }_{ l }, V^{ 0 }_{ l, n - 1 } ( X^{ 0, k, x }_{ l } ) ) - V^{ 0 }_{ l, n - 1 } ( X^{ 0, k, x }_{ l } ) }.
		\end{equation} 
	This and \eqref{bias:exact_solution_formula} ensure that for all
		$ k \in \{ 1, 2, \ldots, K \} $, 
	 	$ n \in \N $, 
	 	$ x \in \mc O $ 
	it holds that 
		\begin{equation} 
		\begin{split}
		& \EXPP{ V^{ 0 }_{ k, n } ( x ) } - v_{ k } ( x ) 
		\\
		& = 
		\sum_{ l = 0 }^{ k - 1 } 
		\EXPP{ 
			\big(
			f_{ l } ( X^{ 0, k, x }_{ l },
		  	V^{ 0 }_{ l, n - 1 } ( X^{ 0, k, x }_{ l } ) ) - V^{ 0 }_{ l, n - 1 } ( X^{ 0, k, x }_{ l } )
		  	\big) 
			-
			\big( 
			f_{ l } ( X^{ 0, k, x }_{ l }, v_{ l } ( X^{ 0, k, x }_{ l } ) ) - v_{ l } ( X^{ 0, k, x }_{ l } ) 
			\big) 	
			}.
		\end{split}
		\end{equation} 
	The assumption that for all 
	 	$ k \in \{ 0, 1, \ldots, K - 1 \} $, 
	 	$ a,b \in \R $, 
	 	$ x	\in \mc O $
	it holds that 
	 	$ | ( f_{ k } ( x, a ) - a ) - ( f_{ k } ( x, b ) - b ) | \leq L_{ k } | a - b | $	
	hence yields that for all 
	 	$ k \in \{ 1, 2, \ldots, K \} $, 
	 	$ n \in \N $, 
	 	$ x \in \mc O $ 
	it holds that 
		\begin{equation} 
		\begin{split} 
		\left| \Exp{ V^{ 0 }_{ k, n } ( x ) } - v_{ k } ( x ) \right| 
		& \leq  
		\sum_{ l = 0 }^{ k - 1 } 
		L_{ l } \, \EXPP{ | V^{ 0 }_{ l, n - 1 } ( X^{ 0, k, x }_{ l } ) - v_{ l } ( X^{ 0, k, x }_{ l } ) | } .
		\end{split}
		\end{equation} 
	The Cauchy-Schwarz inequality therefore shows that for all 
	 	$ k \in \{ 0, 1, \ldots, K \} $, 
	 	$ n \in \N $, 
	 	$ x \in \mc O $ 
	it holds that 
		\begin{equation} 
		\begin{split} 
		\left| \Exp{ V^{ 0 }_{ k, n } ( x ) } - v_{ k } ( x ) \right|^2 
		& \leq 
		\left[ \sum_{ l = 0 }^{ k - 1 } L_{ l }^{ \nicefrac12 } 
		\left( L_{ l }^{ \nicefrac12 } \,
		\Exp{ \big| V^{ 0 }_{ l, n - 1 } ( X^{ 0, k, x }_{ l } ) - v_{ l } ( X^{ 0, k, x }_{ l } ) \big| } \right) \right]^2
		\\
		& \leq 
		\left( \sum_{ l = 0 }^{ k - 1 } L_{ l } \right) \!
		\left( \sum_{ l = 0 }^{ k - 1 } L_{ l } \left| \EXPP{ | V^{ 0 }_{ l, n - 1 } ( X^{ 0, k, x }_{ l } ) - v_{ l } ( X^{ 0, k, x }_{ l } ) | } \right|^2 \right)
		\\
		& \leq 
		\left( \sum_{ l = 0 }^{ k - 1 } L_{ l } \right) \!
		\left( \sum_{ l = 0 }^{ k - 1 } L_{ l } \, \EXPP{ |  V^{ 0 }_{ l, n - 1 } ( X^{ 0, k, x }_{ l } ) - v_{ l } ( X^{ 0, k, x }_{ l } ) |^2} \right)\!. 
		\end{split}
		\end{equation} 
	This establishes \eqref{bias_estimate:claim}. 
	This completes the proof of \cref{lem:bias_estimate}. 
\end{proof}

\begin{lemma}[Variance estimate] \label{lem:variance_estimate}
	Assume 
		\cref{setting}, 
	let 
		$ L_{ 0 }, L_{ 1 }, \ldots, L_{ K } \in [ 0, \infty ) $
	satisfy for all 
		$ k	\in \{ 0, 1, \ldots, K \} $,
		$ x	\in \mc O $,
		$ a,b \in \R $	
	that 
		$ | ( f_{ k } ( x, a ) - a ) - ( f_{ k } ( x, b ) - b ) | \leq L_{ k } | a - b | $, 
	and assume for all 
		$ k \in \{ 0, 1, \ldots, K \} $, 
		$ x \in \mc O $
	that 
		$ \EXP{ | g ( X^{ 0, k, x }_{ 0 } ) |^2 + \sum_{ l = 0 }^{ k - 1 } | f_{ l } ( X^{ 0, k, x }_{ l }, 0 ) |^2 } < \infty $.  
	Then it holds for all 
		$ k \in \{ 1, 2, \ldots, K \} $, 
		$ n \in \N $, 
		$ x \in \mc O $ 
	that 
		\begin{equation} \label{variance_estimate:claim}
		\begin{split}
		\var{ V^{ 0 }_{ k, n } ( x ) } 
		& \leq 
		\frac{ 1 }{ M^n } \Exp{ | g ( X^{ 0, k, x }_{ 0 } ) |^2
		+ 
		\big| \smallsum\nolimits_{ l = 0 }^{ k - 1 } f_{  l } ( X^{ 0, k, x }_{ l }, 0 ) \big|^2 } 
		\\
	 	& + 
	 	\sum_{ j = 1 }^{ n - 1 } \frac{ 1 }{ M^{ n - j } } \left[
	 	\max_{ l \in \{ 0, 1, \ldots, k - 1 \} } \frac{ L_{ l } }{ \mf p_{ k, l } } \right]\!
	 	\left( \sum_{ l = 0 }^{ k - 1 } L_{ l } \,\Exp{ \big|
	 		V^{ 0 }_{ l, j } ( X^{ 0, k, x }_{ l } ) 
	 		- V^{ 1 }_{ l, j - 1 } ( X^{ 0, k, x} _{ l } ) 		\big|^2}
	 	\right)\!.
		\end{split}
	\end{equation}
\end{lemma} 

\begin{proof}[Proof of \cref{lem:variance_estimate}] 
	First, observe that 
		item~\eqref{elementary_measurability_X_processes:item1} of \cref{lem:elementary_measurability_X_processes},  
 		item~\eqref{measurability:item2} of \cref{lem:measurability}, 
 	and the fact that for all 
 		$ k \in \{ 1, 2, \ldots, K \} $ 
 	it holds that 
 		$ \{ W^{\theta}_s, \mc{R}^{\theta}_{k} \colon \theta \in \Theta, s \in \{ 0, 1, \ldots, K - 1 \} \} $ 
	is a family of independent random variables  	
	ensure that for all 
		$ k \in \{ 1, 2, \ldots, K \} $, 
		$ n \in \N $,  
		$ x \in \mc O $
	it holds that 
		\begin{equation}
		  \begin{split}
		  & \var{ V^{ 0 }_{ k, n } ( x ) } 
		  = \sum_{ m = 1 }^{ M^n } \var{ \frac{ g ( X^{ ( 0, 0, -m ), k, x }_{ 0 } ) }{ M^n } } 
		  + \sum_{ m = 1 }^{ M^n } 
		  \var{ \frac{ \sum_{ l = 0 }^{ k - 1 } f_{ l } ( X^{ ( 0, 0, m ), k, x }_{ l }, 0 ) }{ M^n } }
		  \\ 
		  & \quad  
		  + \sum_{ j = 1 }^{ n - 1 } \sum_{ m = 1 }^{ M^{ n - j } } 
		  \VARRRRR{
		  \bigg(
		  \frac{ f_{ \mc{R}^{ ( 0, j, m ) }_{ k } } ( X^{ ( 0, j, m ), k, x }_{ \mc{R}^{ ( 0, j, m ) }_{ k } },  V^{ ( 0, j, m ) }_{ \mc{R}^{ ( 0, j, m ) }_{k}, j } ( X^{ ( 0, j, m ), k, x }_{ \mc{R}^{ ( 0, j, m ) }_{ k } } ) ) - 
		  V^{ ( 0, j, m ) }_{ \mc{R}^{ ( 0, j, m ) }_{ k }, j } ( X^{ ( 0, j, m ), k, x }_{ \mc{R}^{ ( 0, j, m ) }_{ k } } ) }{ M^{ n - j } }
		  \\
		  & \quad    
		  - \frac{ 
		  		f_{ \mc{R}^{ ( 0, j, m ) }_{ k } } (  X^{ ( 0, j, m ), k, x }_{ \mc{R}^{ ( 0, j, m ) }_{ k } },  V^{ ( 0, j, -m ) }_{ \mc{R}^{ ( 0, j, m ) }_{ k }, j - 1 } ( X^{ ( 0, j, m ), k, x }_{ \mc{R}^{ ( 0, j, m ) }_{ k } } ) ) - V^{ ( 0, j, -m ) }_{ \mc{R}^{ ( 0, j, m ) }_{ k }, j - 1 } ( X^{ ( 0, j, m ), k, x }_{ \mc{R}^{ ( 0, j, m ) }_{ k } } ) }{ M^{ n - j } }
		  	\bigg) 
		  \cdot 
		  \frac{ 1 }{ \mf p_{ k, \mc{R}^{ ( 0, j, m ) }_{ k } } }
		   } .
		\end{split}
		\end{equation}  
	Item~\eqref{elementary_equidistribution_X_processes:item1} of \cref{lem:elementary_equidistribution_X_processes} and 
	item~\eqref{equidistribution:item3} of \cref{lem:equidistribution} therefore 
	ensure that for all 
		$ k \in \{ 1, 2, \ldots, K \} $, 
		$ n \in \N $, 
		$ x \in \mc O $ 
	it holds that 
		\begin{equation} 
		\begin{split}
		 \var{ V^{ 0 }_{ k, n } ( x ) } 
		 & = \frac{ 1 }{ M^n } \left( \VAR{ g ( X^{ 0, k, x }_{ 0 } ) } 
		 + \VARRR{ \smallsum\nolimits_{ l = 0 }^{ k - 1 } f_{ l } ( X^{ 0, k, x }_{ l }, 0 ) }
		 \right)
		 \\
		 &  + \sum_{ j = 1 }^{ n - 1 } 
		 \frac{ 1 }{ M^{ n - j } } \VARRRR{ \frac{ 1 }{ \mf p_{ k, \mc{R}^{ 0 }_{ k } } } \big( f_{ \mc{R}^{ 0 }_{ k } }( X^{ 0, k, x }_{ \mc{R}^{ 0 }_{ k } }, V^{ 0 }_{ \mc{R}^{ 0 }_{ k }, j } ( X^{ 0, k, x }_{ \mc{R}^{ 0 }_{ k } } ) ) - 		V^{ 0 }_{ \mc{R}^{ 0 }_{ k }, j } ( X^{ 0, k, x }_{ \mc{R}^{ 0 }_{ k } } ) \big)
	 	\\
	 	& 
	 	- \big( f_{ \mc{R}^{ 0 }_{ k } } ( X^{ 0, k, x }_{ \mc{R}^{ 0 }_{ k } }, V^{ 1 }_{ \mc{R}^{ 0 }_{ k }, j - 1 } ( X^{ 0, k, x }_{ \mc{R}^{ 0 }_{ k } } ) ) - V^{ 1 }_{ \mc{R}^{ 0 }_{ k }, j - 1 } ( X^{ 0, k, x }_{ \mc{R}^{ 0 }_{ k } } ) \big) } .
		\end{split}
		\end{equation} 
	Combining this with the fact that for every random variable 
		$ \mc X \colon \Omega \to \R $ 
	with $\EXP{ | \mc X |} < \infty $ it holds that 
		$ \var{ \mc X } \leq \Exp{ | \mc X |^2 } \in [ 0, \infty ] $ 
	and \cref{lem:integrability} yields that for all 
		$ k \in \{ 0, 1, \ldots, K \} $, 
		$ n \in \N $, 
		$ x \in \mc O $ 
	it holds that 
		\begin{equation} \label{variance:eq03}
		\begin{split}
		\var{ V^{ 0 }_{ k, n } ( x ) } 
		& \leq 
		\frac{ 1 }{ M^n } \Exp{ | g ( X^{ 0, k, x }_{ 0 } ) |^2
		+ \left| \smallsum\nolimits_{ l = 0 }^{ k - 1 }  f_{ l } ( X^{ 0, k, x }_{ l }, 0 ) \right|^2 } 
		\\
		& + \sum_{ j = 1 }^{ n - 1 } 
		\frac{ 1 }{ M^{ n - j } }  
		\EXPPPP{ \Big| \frac{ 1 }{ \mf p_{ k, \mc{R}^{ 0 }_{ k } } }  \big( f_{ \mc{R}^{ 0 }_{ k } } ( X^{ 0, k, x }_{ \mc{R}^{ 0 }_{ k } }, V^{ 0 }_{ \mc{R}^{ 0 }_{ k }, j } ( X^{ 0, k, x }_{ \mc{R}^{ 0 }_{ k } } ) ) 
		- V^{ 0 }_{ \mc{R}^{ 0 }_{ k }, j } ( X^{ 0, k, x }_{ \mc{R}^{ 0 }_{ k } } ) \big)
		\\
		& - \big( f_{ \mc{R}^{ 0 }_{ k } } ( X^{ 0, k, x }_{ \mc{R}^{ 0 }_{ k } }, V^{ 1 }_{ \mc{R}^{ 0 }_{ k }, j - 1 } ( X^{ 0, k, x }_{ \mc{R}^{ 0 }_{ k } } ) ) - 	V^{ 1 }_{ \mc{R}^{ 0 }_{ k }, j - 1 } ( X^{ 0, k, x }_{ \mc{R}^{ 0 }_{ k } } ) \big) \Big|^2 }.
		\end{split} 
		\end{equation} 
	Moreover, note that item~\eqref{elementary_measurability_X_processes:item2} of \cref{lem:elementary_measurability_X_processes}, 	item~\eqref{measurability:item1} of \cref{lem:measurability}, the fact that 
		$ \sigma( ( \mc{R}^0_{s} )_{ s \in \{1, 2, \ldots, K\} } ) $ 
	and 
		$ \sigma( ( \mc{R}^{ ( 0, \eta ) }_{ s } )_{ ( s, \eta ) \in \{1, 2, \ldots, K\} \times \Theta }, ( W^{ ( 0, \eta ) }_{ s } )_{ ( s, \eta ) \in \{ 0, 1, \ldots, K - 1  \} \times \Theta }, ( W^{ 0 }_{ s } )_{ s \in \{ 0, 1, \ldots, K - 1 \} } ) $ 
	are independent sigma-algebras, 
	and the fact that 
		$ \sigma ( ( \mc{R}^0_{s} )_{ s\in \{1, 2, \ldots, K\} } ) $ 
	and 
		$ \sigma ( ( \mc{R}^{ ( 1, \eta ) }_{ s } )_{ (s, \eta ) \in \{1, 2, \ldots, K \} \times \Theta }, ( W^{ ( 1, \eta ) }_{ s } )_{ ( s, \eta ) \in\{ 0, 1, \ldots, K - 1 \} \times \Theta }, \allowbreak ( W^{ 0 }_{ s } )_{ s \in \{ 0, 1, \ldots, K - 1 \} }) $ are independent sigma-algebras on $ \Omega $ ensure that for all 
		$ k \in \{ 1, 2, \ldots, K \} $, 
		$ n \in \N $, 
		$ j \in \{ 1, 2, \ldots, n - 1 \} $, 
		$ x \in \mc O $ 
	it holds that 
		\begin{equation} 
		\begin{split}
		&
		\EXPPPP{ \Big| \frac{ 1 }{ \mf p_{ k, \mc{R}^{ 0 }_{ k } } } \big( f_{ \mc{R}^{ 0 }_{ k } } ( X^{ 0, k, x }_{ \mc{R}^{ 0 }_{ k } }, V^{ 0 }_{ \mc{R}^{ 0 }_{ k }, j } ( X^{ 0, k, x }_{ \mc{R}^{ 0 }_{ k } } ) ) - 
		V^{ 0 }_{ \mc{R}^{ 0 }_{ k }, j } ( X^{ 0, k, x }_{ \mc{R}^{ 0 }_{ k } } ) \big)
		\\
		& \quad 
		- \big( f_{ \mc{R}^{ 0 }_{ k } } ( X^{ 0, k, x }_{ \mc{R}^{ 0 }_{ k } }, V^{ 1 }_{ \mc{R}^{ 0 }_{ k }, j - 1 } ( X^{ 0, k, x }_{ \mc{R}^{ 0 }_{ k } } ) ) - V^{ 1 }_{ \mc{R}^{ 0 }_{ k }, j - 1 } ( X^{ 0, k, x }_{ \mc{R}^{ 0 }_{ k } } )  \big)\Big|^2 }
		\\
		& = 
		\sum_{ l = 0 }^{ k -  1 } 
		\EXPPPP{ \frac{ \mathbbm{1}_{ \{ \mc{R}^{ 0 }_{ k } = l \} } }{ | \mf p_{ k, l } |^2 } \big| \big( f_{ l } \big( X^{ 0, k, x }_{ l }, V^{ 0 }_{ l, j }( X^{ 0, k, x }_{ l } ) \big) - V^{ 0 }_{ l, j }(  X^{ 0, k, x }_{ l } ) \big) 
		\\
		& \quad 
		- \big( f_{ l }( X^{ 0, k, x }_{ l }, V^{ 1 }_{ l, j - 1 } ( X^{ 0, k, x }_{ l } ) ) - V^{ 1 }_{ l, j - 1 } ( X^{ 0, k, x }_{ l } ) \big) \big|^2 }
		\\
		& = 
		\sum_{ l = 0 }^{ k - 1 } 
		\frac{ 1 }{ \mf p_{ k, l } }
		\EXPPP{ \big| \big( f_{ l }( X^{ 0, k, x }_{ l }, V^{ 0 }_{ l, j } ( X^{ 0, k, x }_{ l } ) ) - 	V^{ 0 }_{ l, j } ( X^{ 0, k, x }_{ l } ) \big) 
		\\
		& \quad 
		- \big( f_{ l } ( X^{ 0, k, x }_{ l }, V^{ 1 }_{ l, j - 1 } (X^{ 0, k, x }_{ l } ) ) - V^{ 1 }_{ l, j - 1 }  ( X^{ 0, k, x }_{ l } ) \big) \big|^2 }.
		\end{split}
		\end{equation} 
	This and \eqref{variance:eq03} yield that for all 
		$ k \in \{ 1, 2, \ldots, K \} $, 
		$ n \in \N $, 
		$ x \in \mc O $ 
	it holds that 
		\begin{equation} 
		\begin{split}
		\var{ V^{ 0 }_{ k, n } ( x ) } 
		& \leq 
		\frac{ 1 }{ M^n } 
		\Exp{ | g ( X^{ 0, k, x }_{ 0 } ) |^2
		+ \big| \smallsum_{ l = 0 }^{ k - 1 } f_{ l } ( X^{ 0, k, x }_{ l }, 0 ) \big|^2
		}
		\\
		&  
		+ \sum_{ j = 1 }^{ n - 1 } 
		\frac{ 1 }{ M^{ n - j } } \sum_{ l = 0 }^{ k - 1 } 
		\frac{ 1 }{ \mf p_{ k, l } }
		\EXPPP{\big| \big( f_{ l } ( X^{ 0, k, x }_{ l }, V^{ 0 }_{ l, j } ( X^{ 0, k, x }_{ l } ) ) 
		- V^{ 0 }_{ l, j } ( X^{ 0, k, x }_{ l } ) \big) 
		\\
		& - 
		\big( f_{ l } ( X^{ 0, k, x }_{ l }, V^{ 1 }_{ l, j - 1 } ( X^{ 0, k, x }_{ l } ) ) - V^{ 1 }_{ l, j - 1 } ( X^{ 0, k, x }_{ l } ) \big) 
		\big|^2}.
		\end{split} 
		\end{equation} 
	The assumption that for all 
		$ k \in \{ 0, 1, \ldots, K - 1 \} $, 
		$ a,b \in \R $, 
		$ x \in \mc O  $ 
	it holds that 
		$ | ( f_{ k } ( x, a ) - a ) - ( f_{ k } ( x, b ) - b ) | \leq L_{ k } | a - b | $
	therefore implies that for all 
		$ k \in \{ 1, 2, \ldots, K \} $, 
		$ n \in \N $, 
		$ x \in \mc O $ 
	it holds that 
		\begin{equation} 
		\begin{split}
		 \var{ V^{ 0 }_{ k, n } ( x ) } 
		 & \leq \frac{ 1 }{ M^n } \Exp{ | g ( X^{ 0, k, x }_{ 0 } ) |^2 + \big| \smallsum\nolimits_{ l = 0 }^{ k - 1 } f_{ l } ( X^{ 0, k, x }_{ l }, 0 ) \big|^2 } 
		 \\
		 & 
		 + \sum_{ j = 1 }^{ n - 1 } 
		 \frac{ 1 }{ M^{ n - j } } \sum_{ l = 0 }^{ k - 1 } 
		 \frac{ L_{ l }^2 }{ \mf p_{ k, l } } \Exp{ \big| V^{ 0 }_{ l, j } ( X^{ 0, k, x }_{ l } ) - V^{ 1 }_{ l, j - 1 } ( X^{ 0, k, x }_{ l } ) 	 \big|^2 } \!.
		\end{split}
		\end{equation}
	H\"older's inequality hence ensures for all 
		$ k \in \{ 1, 2, \ldots, K \} $, 
		$ n \in \N $, 
		$ x \in \mc O $ 
	that 
		\begin{equation} 
		\begin{split}
		\var{ V^{ 0 }_{ k, n } ( x ) }  
		& \leq 
		\frac{ 1 }{ M^n } \Exp{ | g ( X^{ 0, k, x }_{ 0 } ) |^2	+ \big| \smallsum\nolimits_{ l = 0 }^{ k - 1 } f_{ l } ( X^{ 0, k, x }_{ l }, 0 ) \big|^2 } 
		\\
		& + 
		\sum_{ j = 1 }^{ n - 1 } \frac{ 1 }{ M^{ n - j } } \left[
		\max_{ l \in [ 0, k ) \cap \N_0 } \frac{ L_{ l } }{ \mf p_{ k, l } }  \right] \!
		\left( \sum_{ l = 0 }^{ k - 1 }  L_{ l } \, \Exp{ \big|
			V^{ 0 }_{ l, j }( X^{ 0, k, x }_{ l } ) -  	V^{ 1 }_{ l, j - 1 } ( X^{ 0, k, x }_{ l } ) \big|^2}
	 	\right)\!.
		\end{split}
		\end{equation}  
	This establishes \eqref{variance_estimate:claim}. 
	This completes the proof of \cref{lem:variance_estimate}. 
\end{proof}

\subsection{Recursive error estimates for MLP approximations} 
\label{subsec:recursive_error_estimate}

\begin{lemma} \label{lem:error_recursion}
	Assume 
		\cref{setting}, 
	let $ c \in [0, \infty) $, 
	let $ \alpha_{ k }, \beta_{ k }, L_{ k } \in [ 0, \infty ) $, $ k \in \{ 0, 1, \ldots, K \} $, 
	assume for all
		$ k \in \{ 0, 1, \ldots, K \} $, 
		$ x \in \mc O $
	that 
		$ \EXP{ | g ( X^{ 0, k, x }_{ 0 } ) |^2 + \sum_{ l = 0 }^{ k - 1 } | f_{ l } ( X^{ 0, k, x }_{ l }, 0 ) |^2 } < \infty $, 
	assume for all 
		$ k \in \{ 0, 1, \ldots, K - 1 \} $, 
		$ x \in \mc O $, 
		$ a,b \in \R $ 
	that  
		$ | ( f_{ k } ( x, a ) - a ) - ( f_{ k } ( x, b ) - b ) | \leq L_{ k } | a - b | $, 
	assume for all 
		$ k, l \in \N_{ 0 } $ 
	with $ l < k \leq K $ that $ L_l \leq c \P(\mc R^{0}_k = l ) $, 
	assume for all 
		$ l \in \{ 0, 1, \ldots, K - 1 \} $ 
	that 
		$ L_{ l } \sum_{ k = l + 1 }^{ K } \alpha_{ k } \leq \beta_{ l } \sum_{ j = 0 }^{ K - 1 } L_{ j } $, 
	and let $ v_{ k } \colon \mc O \to \R $, $ k \in \{ 0, 1, \ldots, K \} $, be $ \Borel( \mc O ) $/$ \Borel( \R ) $-measurable functions which  satisfy for all 
		$ k \in \{ 1, 2, \ldots, K \} $, 
		$ x \in \mc O $ 
	that 
		$ v_{ 0 } ( x ) = g ( x ) $ 
	and 
		\begin{equation} 
		v_{ k } ( x ) = \EXPP{ f_{ k - 1 } ( X^{ 0, k, x }_{ k - 1 }, v_{ k - 1 } ( X^{ 0, k, x }_{ k - 1 } ) ) }
		\end{equation} 
	(cf.~\cref{lem:welldefinedness_exact_solution}). 
	Then it holds for all 
		$ n \in \N $, 
		$ x \in \mc O $ 
	that 
		\begin{equation} \label{error_recursion:claim}
		\begin{split} 
		& 
		\left(
		\sum_{ k = 0 }^{ K } 
		\alpha_{ k } \, \Exp{ \big| V^{ 0 }_{ k, n } ( X^{ 0, K, x }_{ k } ) - v_{ k } ( X^{ 0, K, x }_{ k } ) \big|^2 } 
		\right)^{\!\nicefrac12}
		\\
		& \leq 
		\frac{ 1 }{ \sqrt{ M^n } } 
		\left[ \sum_{ k = 0 }^{ K } \alpha_{ k } \right]^{\nicefrac12}
		\left[ \left( \Exp{ \big| g ( X^{ 0, K, x }_{ 0 } ) \big|^2 } \right)^{\!\nicefrac12}
		+ \sum_{ l = 0 }^{ K - 1 } \left( \Exp{ \big| f_{ l } ( X^{ 0, K, x }_{ l }, 0 ) \big|^2 } \right)^{\!\nicefrac12}
		\right] 
		\\
		& + 2\left[ \max_{ k,l \in [ 0, K ] \cap \N_0, k > l }  \left( \frac{ L_{ l } }{ \mf p_{ k, l } } \right) \right]  
		\left[ \sum_{ j = 0 }^{ n - 1 } \frac{ 1 }{ \sqrt{ M^{ n - j - 1 } } } 
		\left( \sum_{ k = 0 }^{ K - 1 } \beta_{ k } \, \Exp{ \big| V^{ 0 }_{ k, j }( X^{ 0, K, x }_{ k } ) - v_{ k } ( X^{ 0, K, x }_{ k } ) \big|^2 } \right)^{\!\nicefrac12} \right]\!.
		\end{split}
		\end{equation} 
\end{lemma}

\begin{proof}[Proof of \cref{lem:error_recursion}]
	First, note that Minkowski's inequality guarantees that for all 
		$ n \in \N_0 $, 
	 	$ x \in \mc O $  
	it holds that 
		\begin{equation} \label{error_recursion:triangle_inequality}
		\begin{split}
			&
			\left( \sum_{ k = 0 }^{ K } \alpha_{ k } \, \Exp{ \big| V^{ 0 }_{ k, n } - v_{ k } \big|^2  ( X^{ 0, K, x }_{ k } ) } \right)^{\!\nicefrac12} 
			\\
			& \leq 
			\left( \sum_{ k = 0 }^{ K } \alpha_{ k } \, \Exp{ \big| V^{ 0 }_{ k, n } - \Exp{ V^{ 0 }_{ k, n } } \! \big|^2 ( X^{ 0, K, x }_{ k } ) } 
			\right)^{\!\nicefrac12} 
			+ 
			\left( \sum_{ k = 0 }^{ K } \alpha_{ k } \, \Exp{ \big| \Exp{ V^{ 0 }_{ k, n } } - v_{ k } \big|^2 ( X^{ 0, K, x }_{ k } ) } 
			\right)^{\!\nicefrac12} .
		\end{split}
		\end{equation} 	
	Moreover, note that \cref{lem:bias_estimate} ensures that for all 
		$ k \in \{ 1, 2, \ldots, K \} $, 
		$ n \in \N $, 
		$ x \in \mc O $ 
	it holds that 
		\begin{equation} 
		\left| \Exp{ V^{ 0 }_{ k, n } ( x ) } - v_{ k } ( x ) \right|^2
		\leq  
		\left( \sum_{ l = 0 }^{ k - 1 } L_{ l } \right) \!
		\left( \sum_{ l = 0 }^{ k - 1 } L_{ l } \, \Exp{ \big| V^{ 0 }_{ l, n - 1 } ( X^{ 0, k, x }_{ l } ) 
		- v_{ l } ( X^{ 0, k, x }_{ l } ) \big|^2 } \right)\!.
		\end{equation} 
	In addition, observe that 
		item~\eqref{elementary_measurability_X_processes:item2} of \cref{lem:elementary_measurability_X_processes} 
	and 	
		item~\eqref{measurability:item1} of \cref{lem:measurability} 		
	ensure that for all 
		$ k, l, n \in \N_0 $ 
	with 
		$ l \leq k \leq K $ 
	it holds that 
		$ \mc O \times \Omega \ni ( x, \omega ) \mapsto V^{ 0 }_{ l, n } ( X^{ 0, k, x }_{ l } ( \omega ), \omega ) \in \R $ 
	is
		$ ( \Borel( \mc O ) \otimes \sigma( ( \mc{R}^{ ( 0, \vartheta ) }_{ s } )_{ ( s, \vartheta ) \in \{1, 2, \ldots, K\} \times \Theta }, ( W^{ ( 0, \vartheta ) }_{ s } )_{ ( s, \vartheta ) \in \{ 0, 1, \ldots, K - 1 \} \times \Theta }, ( W^{ 0 }_{ s } )_{ s \in [ l, k ) \cap \N_0 } ) ) $/$ \Borel( \R ) $-measurable. 
	Combining this and the fact that for all 
		$ k,l \in \N_0 $ 
	with 
		$ l \leq k \leq K $
	it holds that 
		$ \sigma( ( W^{ 0 }_{ s } )_{ s \in [ k, K - 1 ] \cap \N_0 } ) $ 
	and 
		$ \sigma( ( \mc{R}^{ ( 0, \vartheta ) }_{ s } )_{ (s, \vartheta ) \in \{1, 2, \ldots, K\} \times \Theta }, ( W^{ 0 }_{ s } )_{ s \in [ l, k ) \cap \N_0 }, \allowbreak ( W^{ ( 0, \vartheta ) }_{ s } )_{ ( s, \vartheta ) \in \{ 0, 1, \ldots, K - 1 \} \times \Theta }  ) $ 
	are independent sigma-algebras on $ \Omega $ with \cref{lem:bias_estimate} and Hutzenthaler et al.~\cite[Lemma 2.2]{Overcoming} ensures that for all 
		$ k \in \{ 1, 2, \ldots, K \} $, 
		$ n \in \N $, 
		$ x \in \mc O $ 
	it holds that 
		\begin{equation} 
		\begin{split}
		 & 
		 \Exp{ \big| \Exp{ V^{ 0 }_{ k, n } } - v_{ k } \big|^2 ( X^{ 0, K, x }_{ k } \big) }
		 = \int_{ \mc O } \big| \Exp{ V^{ 0 }_{ k, n } ( y ) } - v_{ k } ( y ) \big|^2 \left[ \big( X^{ 0, K, x }_{ k } ( \P ) \big) ( dy) \right]
		 \\ 
		 & \leq  
		 \int_{ \mc O } \left( \sum_{ l = 0 }^{ k - 1 } L_{ l } \right)\!
		 \left( \sum_{ l = 0 }^{ k - 1 } L_{ l } \, \Exp{ \big| V^{ 0 }_{ l, n - 1 } ( X^{ 0, k, y }_{ l } ) 
		 - v_{ l } ( X^{ 0, k, y }_{ l } ) \big|^2 } \right)\!
		 \left[ \big( X^{ 0, K, x }_{ k }( \P ) \big)( dy ) \right] 
		 \\
		 & = 
		 \left( \sum_{ l = 0 }^{ k - 1 } L_{ l } \right)\!\left( \sum_{ l = 0 }^{ k - 1 } L_{ l }
		 \int_{ \mc O } \Exp{ \big| V^{ 0 }_{ l, n - 1 } ( X^{ 0, k, y }_{ l } ) - v_{ l } ( X^{ 0, k, y }_{ l } ) \big|^2 } 
		 \!\left[ \big( X^{ 0, K, x }_{ k }( \P ) \big)( dy ) \right] \right)
		 \\
		 & = 
		 \left( \sum_{ l = 0 }^{ k - 1 } L_{ l } \right)\!\left( \sum_{ l = 0 }^{ k - 1 } L_{ l } \, \Exp{ \Big| V^{ 0 }_{ l, n - 1 } \big( X^{ 0, k, X^{ 0, K, x }_{ k } }_{ l } \big) 
		 - v_{ l } \big( X^{ 0, k, X^{ 0, K, x }_{ k } }_{ l } \big) \Big|^2 } \right)
		 \\
		 & = \left( \sum_{ l = 0 }^{ k - 1 } L_{ l } \right)\!\left( \sum_{ l = 0 }^{ k - 1 } L_{ l } \, \Exp{ |  V^{ 0 }_{ l, n - 1 } ( X^{ 0, K, x }_{ l } ) - v_{ l } ( X^{ 0, K, x }_{ l } ) \big|^2 } 
		 \right)\!.
		\end{split}
		\end{equation}  
	This and the fact that for all 
		$ n \in \N $, 
		$ x \in \mc O $ 
	it holds that 
		$ V^{ 0 }_{ 0, n } ( x ) = g ( x ) $
	demonstrate that for all 
		$ n \in \N $, 
		$ x \in \mc O $
	it holds that 
		\begin{equation} \label{error_recursion:bias_part}
		\begin{split} 
		& 
		\sum_{ k = 0 }^{ K } 
		\alpha_k \, \Exp{ \big| \Exp{ V^{ 0 }_{ k, n } } - v_{ k } \big|^2 ( X^{ 0, K, x }_{ k } ) } 
		\\
		& 
		\leq \sum_{ k = 1 }^{ K } \alpha_{ k } \left( \sum_{ l = 0 }^{ k - 1 } L_{ l } \right) 
		\left( \sum_{ l = 0 }^{ k - 1 } L_{ l } \, \Exp{ \big| V^{ 0 }_{ l, n - 1 } ( X^{ 0, K, x }_{ l } ) - v_{l} ( X^{ 0, K, x  }_{ l } ) \big|^2 } 
		\right) 
		\\
		& = 
		\sum_{ l = 0 }^{ K - 1 } \sum_{ k = l + 1 }^{ K } 
		\alpha_{ k } \left( \sum_{ j = 0 }^{ k - 1 } L_{ j } \right)
		L_{ l } \, \Exp{ \big| V^{ 0 }_{ l, n - 1 } ( X^{ 0, K, x }_{ l } ) - v_{ l } ( X^{ 0, K, x }_{ l } ) \big|^2 } 
		\\
		& \leq 
		\left( \sum_{ j = 0 }^{ K - 1 } L_{ j } \right) \sum_{ l = 0 }^{ K - 1 }
		L_{ l } \left( \sum_{ k = l + 1 }^{ K } \alpha_{ k } \right)  
		\Exp{ \big| V^{ 0 }_{ l, n - 1 } ( X^{ 0, K, x }_{ l } ) - v_{ l } ( X^{ 0, K, x }_{ l } ) \big|^2 }\!.
		\end{split}
		\end{equation} 
	Next note that the fact that for all 
		$ l \in \{ 0, 1, \ldots, K - 1 \} $
	with 
		$ \P( \mc R^{ 0 }_{ K } = l ) > 0 $ 
	it holds that 
		$ \P( \mc R^{ 0 }_{ K } = l ) = \mf p_{ K, l } $,  
	the fact that for all 
		$ l \in \{ 0, 1, \ldots, K - 1 \} $
	with 
		$ \P( \mc R^{ 0 }_{ K } = l ) = 0 $
	it holds that  
		$ L_l = 0 $, 
	and the fact that 
		$ \sum_{ l = 0 }^{ K - 1 } \P ( \mc R^{ 0 }_{ K } = l ) = 1 $
	ensure that 
		\begin{equation} \label{error_recursion:auxiliary_inequality}
		\begin{split} 
		\sum_{ l = 0 }^{ K - 1 } L_l 
		& = 
		\sum_{ l = 0 }^{ K - 1 } L_l \, \mathbbm{1}_{ (0, \infty) }( \P( \mc R^{ 0 }_K = l ) ) 
		\\
		& = 
		\sum_{ l = 0 }^{ K - 1 } \frac{ L_l }{ \mf p_{ K, l } } \, \mf p_{ K, l} \, \mathbbm{1}_{ (0, \infty) }( \P( \mc R^{ 0 }_K = l ) ) 
		= 
		\sum_{ l = 0 }^{ K - 1 } \frac{ L_l }{ \mf p_{ K, l } } \, \P( \mc R^{ 0 }_K = l ) \, \mathbbm{1}_{ (0, \infty) }( \P( \mc R^{ 0 }_K = l ) ) 
		\\
		& \leq 
		\left[ \max_{ l \in \{ 0, 1, \ldots, K - 1 \} } \frac{ L_l }{ \mf p_{ K, l } } \right] 
		\left[ \sum_{ l = 0 }^{ K - 1 } \P( \mc R^{ 0 }_K = l ) \, \mathbbm{1}_{ ( 0, \infty ) } ( \P( \mc R^{ 0 }_K = l ) ) \right]
		\\
		& = 
		\left[ \max_{ l \in \{ 0, 1, \ldots, K - 1 \} } \frac{ L_l }{ \mf p_{ K, l } } \right] 
		\left[ \sum_{ l = 0 }^{ K - 1 } \P( \mc R^{ 0 }_K = l ) \right]
		= 
		\left[ \max_{ l \in \{ 0, 1, \ldots, K - 1 \} } \frac{ L_l }{ \mf p_{ K, l } } \right]\!.
		\end{split}
		\end{equation} 
	The assumption that for all 
		$ l \in \{ 0, 1, \ldots, K - 1 \} $ 
	it holds that  
		$ L_{ l } \sum_{ k = l + 1 }^{ K } \alpha_{ k } \leq \beta_{ l } \sum_{ j = 0 }^{ K - 1 } L_{ j } $ 
	and \eqref{error_recursion:bias_part} therefore ensure that for all 
		$ n \in \N $, 
		$ x \in \mc O $ 
	it holds that 
		\begin{equation} \label{error_recursion:estimation_of_the_bias_error}
		\begin{split} 
		&  
		\sum_{ k = 0 }^{ K } \alpha_k \, \Exp{ \big| \Exp{ V^{ 0 }_{ k, n } } - v_{ k } \big|^2 ( X^{ 0, K, x }_{ k } ) } 
		\\
		& 
		\leq \left[ \sum_{ j = 0 }^{ K - 1 } L_{ j } \right]^{ 2 } \!
		\left( \sum_{ l = 0 }^{ K - 1 } \beta_{ l } \, \Exp{ \big| V^{ 0 }_{ l, n - 1 } ( X^{ 0, K, x }_{ l } ) 
		- v_{ l } ( X^{ 0, K, x }_{ l } ) \big|^2 } \right)
		\\
		& 
		\leq \left[ \max_{ j \in [0,K)\cap\N_0 } \frac{ L_j }{ \mf p_{ K, j } } \right]^2\!		
		\left( \sum_{ l = 0 }^{ K - 1 } \beta_{ l } \, \Exp{ \big| V^{ 0 }_{ l, n - 1 } ( X^{ 0, K, x }_{ l } ) 
		- v_{ l } ( X^{ 0, K, x }_{ l } ) \big|^2 } 
		\right)\!.
		\end{split}
		\end{equation} 
	In the next step we note that item~\eqref{measurability:item1} of \cref{lem:measurability} ensures that for all 
		$ n \in \N_0 $, 
		$ k \in \{ 0, 1, \ldots, K \} $ 
	it holds that 
		\begin{equation} 
		\mc O \times \Omega \ni ( x, \omega ) \mapsto \left| V^{ 0 }_{ k, n } ( x, \omega ) - \int_{ \Omega } V^{ 0 }_{ k, n } ( x, \xi ) \, \P( d\xi ) \right|^2 \in \R 
		\end{equation} 
	is
		$ ( \Borel( \mc O ) \otimes \sigma ( ( \mc{R}^{ ( 0, \vartheta ) }_{ s } )_{ (s, \vartheta) \in \{1, 2, \ldots, K \} \times \Theta }, ( W^{ ( 0, \vartheta ) }_{ s } )_{ ( s, \vartheta ) \in \{ 0, 1, \ldots, K - 1 \} \times \Theta } ) ) $/$ \Borel( \R ) $-measurable. 
	The fact that 	
		$ \sigma( ( \mc{R}^{ ( 0, \vartheta ) }_{ s } )_{ (s, \vartheta) \in \{1, 2, \ldots, K\} \times \Theta }, ( W^{ ( 0, \vartheta ) }_{ s } )_{ ( s, \vartheta ) \in \{ 0, 1, \ldots, K - 1 \} \times \Theta } ) $ 
	and 
		$ \sigma( (W^{ 0 }_{ s } )_{ s \in \{ 0, 1, \ldots, K - 1 \} } ) $ 
	are independent, item\ \eqref{elementary_measurability_X_processes:item1} of \cref{lem:elementary_measurability_X_processes}, 
	and Hutzenthaler et al.~\cite[Lemma 2.2]{Overcoming} hence show that for all 
		$ k \in \{ 0, 1, \ldots, K \} $, 
		$ n \in \N $, 
		$ x \in \mc O $ 
	it holds that  
		\begin{equation} 
		\begin{split} 
		& 
		\Exp{ \big| V^{ 0 }_{ k, n } - \Exp{ V^{ 0 }_{ k, n } } \! \big|^2 ( X^{ 0, K, x }_{ k } ) } 
		= \int_{ \mc O } \Exp{ \big | V^{ 0 }_{ k, n } ( y ) - \Exp{ V^{ 0 }_{ k, n } ( y ) }\! \big|^2}\!\,\big( X^{ 0, K, x }_{ k } ( \P ) \big) ( dy ) .
		\end{split}
		\end{equation} 
	\cref{lem:variance_estimate} hence guarantees for all 
		$ k \in \{ 1, 2, \ldots, K \} $, 
		$ n \in \N $, 
		$ x \in \mc O $ 
	that 
		\begin{equation} 
		\begin{split}
		& 
		\Exp{ \big| V^{ 0 }_{ k, n } - \Exp{ V^{ 0 }_{ k, n } }\! \big|^2 ( X^{ 0, K, x }_{ k } ) } 
		\\
		& 
		\leq \int_{ \mc O } \frac{ 1 }{ M^n } \Exp{ \big| g ( X^{ 0, k, y }_{ 0 } ) \big|^2 + 
		 \Big| \sum\nolimits_{ l = 0 }^{ k - 1 } f_{ l } ( X^{ 0, k, y }_{ l }, 0 ) \Big|^2
		 }\! ( X^{ 0, K, x }_{ k } ( \P ) \big) ( dy ) 
		\\
		&
		+ \int_{ \mc O } \sum_{ j = 1 }^{ n - 1 } \frac{  
		\max_{ l \in \{ 0, 1, \ldots, k - 1 \} } ( \frac{ L_{ l } }{ \mf p_{ k, l } } ) }{ M^{ n - j } }
		\! \left( \sum_{ l = 0 }^{ k - 1 } L_{ l } \, \Exp{ \big| V^{ 0 }_{ l, j } ( X^{ 0, k, y }_{ l } ) - V^{ 1 }_{ l, j - 1 } ( X^{ 0, k, y }_{ l } ) \big|^2 } \right)\!\,\big( X^{ 0, K, x }_{ k } ( \P ) \big) ( dy ) .  
		\end{split} 
		\end{equation} 
	Combining this with item~\eqref{elementary_measurability_X_processes:item1} of \cref{lem:elementary_measurability_X_processes} 
	and Hutzenthaler et al.~\cite[Lemma 2.2]{Overcoming} yields that for all 
		$ k \in \{ 1, 2, \ldots, K \} $, 
		$ n \in \N $, 
		$ x \in \mc O $ 
	it holds that 
		\begin{equation} 
		\begin{split}
		& 
		\Exp{ \big| V^{ 0 }_{ k, n } - \Exp{ V^{ 0 }_{ k, n } }\! \big|^2 ( X^{ 0, K, x }_{ k } ) } 
		\\
		& 
		\leq \frac{ 1 }{ M^n } \Exp{ \Big| g \big( X^{ 0, k, X^{ 0, K, x }_{ k } }_{ 0 } \big ) \Big|^2
			+ \Big| \sum\nolimits_{ l = 0 }^{ k - 1 } f_{ l } \big( X^{ 0, k, X^{ 0, K, x }_{ k } }_{ l }, 0 \big) \Big|^2
		}  
		\\
		&
		+ \int_{ \mc O }
		\sum_{ j = 1 }^{ n - 1 } \frac{ \max_{ l \in \{ 0, 1, \ldots, k - 1 \} } ( \frac{ L_{ l } }{ \mf p_{ k, l } } ) }{ M^{ n - j } } 
		\! \left( \sum_{ l = 0 }^{ k - 1 } L_{ l } \, \Exp{ \big| V^{ 0 }_{ l, j } ( X^{ 0, k, y }_{ l } ) - V^{ 1 }_{ l, j - 1 } ( X^{ 0, k, y }_{ l } ) \big|^2 } \right)\!\,\big( X^{ 0, K, x }_{ k }( \P ) \big)(dy).  
		\end{split} 
		\end{equation} 
	Next observe that item~\eqref{elementary_measurability_X_processes:item2} of \cref{lem:elementary_measurability_X_processes} and 
	item~\eqref{measurability:item1} of \cref{lem:measurability} show that for all
		$ k,l \in \N_0 $, 
		$ n \in \N $ 
	with 
		$ l \leq k \leq K $ 
	it holds that 
		$ \mc O \times \Omega \ni ( x, \omega ) \mapsto | V^{ 0 }_{ l, n } ( X^{ 0, k, x }_{ l } ( \omega ), \omega ) - V^{ 1 }_{ l, n - 1 } ( X^{ 0, k, x }_{ l } ( \omega ), \omega ) |^2 \in [0,\infty) $ 
	is 
		$ ( \Borel( \mc O ) \otimes \sigma ( 
		 ( \mc{R}^{ ( 0, \eta ) }_{ s } )_{ (s, \eta) \in \{1, 2, \ldots, K\} \times \Theta }, 
		 ( \mc{R}^{ ( 1, \eta ) }_{ s } )_{ (s, \eta) \in \{1, 2, \ldots, K\} \times \Theta }, 
		 ( W^{ ( 0, \eta ) }_{ s } )_{ ( s, \eta ) \in \{ 0, 1, \ldots, K - 1 \} \times \Theta }, 
		 ( W^{ 0 }_{ s } )_{ s \in [ l, k ) \cap \N_0 }, \allowbreak 		 
		 ( W^{ ( 1, \eta ) }_{ s } )_{ ( s, \eta ) \in \{ 0, 1, \ldots, K - 1 \} \times \Theta }
		 )
		 ) $/$ \Borel([0,\infty)) $-measurable. 
	This, the fact that for all 
		$ k,l \in \N_0 $ 
	with 
		$ l \leq k \leq K $
	it holds that 
		$ \sigma( ( W^{ 0 }_{ s } )_{ s \in [ k, K ) \cap \N_0 } ) $ 
	and
		$ \sigma( 
		( \mc{R}^{ ( 0, \eta ) }_{ s } )_{ (s, \eta) \in \{1, 2, \ldots, K\} \times \Theta },  
		( \mc{R}^{ ( 1, \eta ) }_{ s } )_{ (s, \eta) \in \{1, 2, \ldots, K\} \times \Theta }, 
		( W^{ 0 }_{ s } )_{ s \in [ l, k ) \cap \N_0 }, \allowbreak
		( W^{ ( 0, \eta ) }_{ s } )_{ ( s, \eta ) \in \{ 0, 1, \ldots, K - 1 \} \times \Theta }, \allowbreak		 
		( W^{ ( 1, \eta ) }_{ s } )_{ ( s, \eta ) \in \{ 0, 1, \ldots, K - 1 \} \times \Theta }) $ 		
	are independent, and Hutzenthaler et al.~\cite[Lemma 2.2]{Overcoming} demonstrate that for all 
		$ k \in \{ 1, 2, \ldots, K \} $, 
		$ n \in \N $, 
		$ x \in \mc O $ 
	it holds that
		\begin{equation} 
		\begin{split} 
		& 
		\Exp{ \big| V^{ 0 }_{ k, n } - \Exp{ V^{ 0 }_{ k, n } }\! \big|^2 ( X^{ 0, K, x }_{ k } ) } 
		\\
		& \leq
	 	\frac{ 1 }{ M^n } \Exp{ \Big| g \big( X^{ 0, k, X^{ 0, K, x }_{ k } }_{ 0 } \big) \Big|^2
		+ \Big| \sum\nolimits_{ l = 0 }^{ k - 1 } f_{ l } \big( X^{ 0, k, X^{ 0, K, x }_{ k } }_{ l }, 0 \big) \Big|^2
	 	} 
		\\
		&
		+ \sum_{ j = 1 }^{ n -  1 } \frac{ 1 }{ M^{ n - j } } 
		\left[ \max_{ l \in \{ 0, 1, \ldots, k - 1 \} } \frac{ L_{ l } }{ \mf p_{ k, l } } \right]  
		\!\left( \sum_{ l = 0 }^{ k - 1 } L_{ l } \, \Exp{ \Big| V^{ 0 }_{ l, j } \big( X^{ 0, k, X^{ 0, K, x }_{ k } }_{ l } ) - V^{ 1 }_{ l, j - 1 } ( X^{ 0, k, X^{ 0, K, x }_{ k } }_{ l } \big) \Big|^2 }
	 \right)\!.
	\end{split} 
	\end{equation} 
	This and \eqref{setting:dynamics} ensure for all 
		$ k \in \{ 1, 2, \ldots, K \} $, 
		$ n \in \N $, 
		$ x \in \mc O $ 
	that 
		\begin{equation} 
		\begin{split} 
		& 
		\Exp{\big| V^{ 0 }_{ k, n } - \Exp{ V^{ 0 }_{ k, n } } \! \big|^2 ( X^{ 0, K, x }_{ k } ) } 
		\leq \frac{ 1 }{ M^n } \Exp{ \big| g ( X^{ 0, K, x }_{ 0 } ) \big|^2 +
		\Big| \sum\nolimits_{ l = 0 }^{ k - 1 } f_{ l } ( X^{ 0, K, x }_{ l }, 0 ) \Big|^2 } 
		\\
		& + 
		\sum_{ j = 1 }^{ n - 1 } \frac{ 1 }{ M^{ n - j } } 
		\left[ \max_{ l \in \{ 0, 1, \ldots, k - 1 \} } \frac{ L_{ l } }{ \mf p_{ k, l } } \right] \!
		\left( \sum_{ l = 0 }^{ k - 1 } L_{ l } \, \Exp{ \big| V^{ 0 }_{ l, j } ( X^{ 0, K, x }_{ l } ) 
		- V^{ 1 }_{ l, j - 1 } ( X^{ 0, K, x }_{ l } ) \big|^2 } \right)\!. 
		\end{split}
		\end{equation} 
	The fact that for all 
		$ n \in \N $, 
		$ x \in \mc O $ 
	it holds that 
		$ V^{ 0 }_{ 0, n } ( x ) = g ( x ) $
	hence guarantees that for all 
		$ n \in \N $, 
		$ x \in \mc O $
	it holds that 
		\begin{equation} 
		\begin{split}
		& 
		\sum_{ k = 0 }^{ K } \alpha_{ k } \, \Exp{ \big| V^{ 0 }_{ k, n } - \Exp{ V^{ 0 }_{ k, n } } \!\big|^2 ( X^{ 0, K, x }_{ k } ) }
		\\&
		\leq \frac{ 1 }{ M^n } \left[ \sum_{ k = 0 }^{ K } \alpha_{ k } \right] 
		\Exp{ \big| g ( X^{ 0, K, x }_{ 0 } ) \big|^2 + \left( \sum\nolimits_{ l = 0 }^{ K - 1 } 
			| f_{ l } ( X^{ 0, K, x }_{ l }, 0 ) | \right)^{\!2}
		} 
		\\
		& 
		+ \sum_{ k = 1 }^{ K } \sum_{ j = 1 }^{ n - 1 } \frac{ \alpha_{ k } }{ M^{ n - j } } 
		\left[ \max_{ l \in \{ 0, 1, \ldots, k - 1 \} } \frac{ L_{ l } }{ \mf p_{ k, l } }  \right] \! \left( \sum_{ l = 0 }^{ k - 1 } 
		L_{ l } \, \Exp{ \big| V^{ 0 }_{ l, j } ( X^{ 0, K, x }_{ l } ) - V^{ 1 }_{ l, j - 1 } ( X^{ 0, K, x }_{ l } ) \big|^2 } \right)\!.
		\end{split}
		\end{equation} 
	This yields for all 
		$ n \in \N $, 
		$ x \in \mc O $ 
	that 
		\begin{equation} 
		\begin{split}
		& 
		\sum_{ k = 0 }^{ K } \alpha_{ k } \, \Exp{ \big| V^{ 0 }_{ k, n } - \Exp{ V^{ 0 }_{ k, n } }\! \big|^2 ( X^{ 0, K, x }_{ k } ) }
		\\
		& \leq \frac{ 1 }{ M^n } \left[ \sum_{ k = 0 }^{ K } \alpha_{ k } \right] \Exp{ \big| g ( X^{ 0, K, x }_{ 0 } ) \big|^2 +
		\left( \sum\nolimits_{ l = 0 }^{ K - 1 } | f_{ l } ( X^{ 0, K, x }_{ l }, 0 ) | \right)^{\!2} } 
		\\
		& + 
		\left[ \max_{ k,l \in [0,K]\cap\N_0, k > l } \frac{ L_{ l } }{ \mf p_{ k, l } } \right] 
		\sum_{ j = 1 }^{ n - 1 }
		\frac{ 1 }{ M^{ n - j } }
		\sum_{ l = 0 }^{ K - 1 } 
		\left[ L_{ l } \left( \sum_{ k = l + 1 }^{ K } \alpha_{ k } \right) \right]
		\Exp{\big| V^{ 0 }_{ l, j } ( X^{ 0, K, x }_{ l } ) - V^{ 1 }_{ l, j - 1 } ( X^{ 0, K, x }_{ l } ) \big|^2 }\!.
		\end{split}
		\end{equation} 
	The assumption that for all 
		$ l \in \{ 0, 1, \ldots, K - 1 \} $ 
	it holds that 
		$ L_l \sum_{ k = l + 1 }^{ K } \alpha_{ k } \leq \beta_{ l } \sum_{ j = 0 }^{ K - 1 } L_{ j } $ 
	therefore demonstrates that for all 
		$ n \in \N $, 
		$ x \in \mc O $ 
	it holds that 
		\begin{equation}
		\begin{split}
		& \sum_{ k = 0 }^{ K } 
		\alpha_{ k } \, \Exp{ \big| V^{ 0 }_{ k, n } - \Exp{ V^{ 0 }_{ k, n } } \! \big|^2 ( X^{ 0, K, x }_{ k } ) }
		\\
		& 
		\leq \frac{ 1 }{ M^n } \left[ \sum_{ k = 0 }^{ K } \alpha_{ k } \right] 
		\Exp{ \big| g ( X^{ 0, K, x }_{ 0 } ) \big|^2 + \left( \sum\nolimits_{ l = 0 }^{ K - 1 } | f_{ l } ( X^{ 0, K, x }_{ l }, 0 ) | \right)^{\!2}
		} 
		\\
		& + 
		\left[ \max_{ k,l \in [ 0, K ] \cap \N_0, k>l } \frac{ L_{ l } }{ \mf p_{ k, l } } \right] 
		\left[ \sum_{ l = 0 }^{ K - 1 } L_{ l } \right]
		\left[ \sum_{ j = 1 }^{ n - 1 } \frac{ 1 }{ M^{ n - j } } 
		\left( \sum_{ l = 0 }^{ K - 1 } \beta_{ l } \, \Exp{ \big| V^{ 0 }_{ l, j } ( X^{ 0, K, x }_{ l } ) - 	V^{ 1 }_{ l, j - 1 } ( X^{ 0, K, x }_{ l } ) \big|^2 } \right) \right]\!.
		\end{split}
		\end{equation} 
	This and \eqref{error_recursion:auxiliary_inequality} ensure that for all 
		$ n \in \N $, 
		$ x \in \mc O $ 
	it holds that 
		\begin{equation} 
		\begin{split}
		& \sum_{ k = 0 }^{ K } 
		\alpha_k \, \Exp{ \big| V^{ 0 }_{ k, n } - \Exp{ V^{ 0 }_{ k, n } }\! \big|^2 ( X^{ 0, K, x }_{ k } ) }
		\\
		& 
		\leq 
		\frac{1}{M^n} \left[ \sum_{k=0}^{K} \alpha_k \right] 
		\Exp{ | g ( X^{ 0, K, x }_{ 0 } ) |^2 + \left( \sum\nolimits_{ l = 0 }^{ K - 1 } | f_{ l } ( X^{ 0, K, x }_{ l }, 0 ) | \right)^{\!2} } 
		\\
		& + 
		\left[ \max_{ k,l \in [ 0, K ] \cap \N_0, k > l }  \left( \frac{ L_{ l } }{ \mf p_{ k, l } } \right) \right]^2
		\left[ \sum_{ j = 1 }^{ n - 1 } \frac{ 1 }{ M^{ n - j } }
		\left( \sum_{ l = 0 }^{ K - 1 } \beta_{ l } \, \Exp{ \big| V^{ 0 }_{ l, j }( X^{ 0, K, x }_{ l } ) 
		- V^{ 1 }_{ l, j - 1 } ( X^{ 0, K, x }_{ l } ) \big|^2 } \right) \right]\!.
		\end{split}
		\end{equation} 
	The fact that for all 
		$ m \in \N $, 
		$ a_{ 1 }, a_{ 2 }, \ldots, a_{ m } \in [ 0, \infty ) $ 
	it holds that 
		$ \sqrt{ \sum_{ i = 1 }^{ m } a_{ i } } \leq \sum_{ i = 1 }^{ m } \sqrt{ a_{ i } } $, 
	and Minkowski's inequality therefore prove that for all 
		$ n \in \N $, 
		$ x \in \mc O $ 
	it holds that 
		\begin{equation} 
		\begin{split} 
		& 
		\left( \sum_{ k = 0 }^{ K } 
		\alpha_{ k } \, \Exp{ \big| V^{ 0 }_{ k, n } - \Exp{ V^{ 0 }_{ k, n } } \! \big|^2 ( X^{ 0, K, x }_{ k } ) } \right)^{ \!\nicefrac12 }
		\\
		& \leq 
		\frac{ 1 }{ \sqrt{ M^n } } 
		\left[ \sum_{ k = 0 }^{ K } \alpha_{ k } \right]^{ \nicefrac12 }
		\left[ \left( \Exp{ \big| g ( X^{ 0, K, x }_{ 0 } ) \big|^2 } \right)^{ \!\nicefrac12 } 
		+ \sum_{ l = 0 }^{ K - 1 } \left( 
			\Exp{ \big| f_{ l } ( X^{ 0, K, x }_{ l }, 0 ) \big|^2 } \right)^{ \!\nicefrac12 } \right] 
		\\
		& 
		+ \left[ \max_{ k,l \in [ 0, K ] \cap \N_0, k > l }  \left( \frac{ L_{ l } }{ \mf p_{ k, l } } \right) \right]
		\left[ \sum_{ j = 1 }^{ n - 1 } 
		\frac{ 1 }{ \sqrt{ M^{ n - j } } } 
		\left( \sum_{ l = 0 }^{ K - 1 } \beta_{ l } \, \Exp{ \big| V^{ 0 }_{ l, j } ( X^{ 0, K, x }_{ l } ) - v_{ l } ( X^{ 0, K, x }_{ l } ) \big|^2} \right)^{ \!\nicefrac12 } \right]
		\\
		& + 
		\left[ \max_{ k,l \in [ 0, K ] \cap \N_0, k > l }  \left( \frac{ L_{ l } }{ \mf p_{ k, l } } \right) \right] 
		\left[ 
		\sum_{ j = 1 }^{ n - 1 } 
		\frac{ 1 }{ \sqrt{ M^{ n - j } } } 
		\left( \sum_{ l = 0 }^{ K - 1 } \beta_{ l } \, \Exp{ \big| V^{ 1 }_{ l, j - 1 } ( X^{ 0, K, x }_{ l } ) - v_{ l } ( X^{ 0, K, x }_{ l } ) \big|^2 } \right)^{ \!\nicefrac12 }
		\right]\!.
	\end{split}
	\end{equation} 
	Item~\eqref{elementary_measurability_X_processes:item1} of \cref{lem:elementary_measurability_X_processes}, 
	item~\eqref{measurability:item1} of \cref{lem:measurability}, 
	the assumption that $ \sigma ( ( W^{ ( 1, \eta ) }_{ s } )_{ ( s, \eta ) \in \{ 0, 1, \ldots, K - 1 \} \times \Theta }, \allowbreak ( \mc{R}^{ ( 1, \eta ) }_{ s } )_{ (s, \eta) \in \{1, 2, \ldots, K\} \times \Theta } ) $ and $ \sigma( ( W^{ 0 }_{ s } )_{ s \in \{ 0, 1, \ldots, K - 1 \} } ) $ are independent,
	item~\eqref{elementary_equidistribution_X_processes:item1} of \cref{lem:elementary_equidistribution_X_processes}, 
	item~\eqref{equidistribution:item1} of \cref{lem:equidistribution}, 
	the assumption that $ \sigma ( ( W^{ ( 0, \eta ) }_{ s } )_{ ( s, \eta ) \in \{ 0, 1, \ldots, K - 1 \} \times \Theta }, ( \mc{R}^{ ( 0, \eta ) }_{ s } )_{ (s, \eta) \in \{1, 2, \ldots, K\} \times \Theta } ) $ and $ \sigma ( ( W^{ 0 }_{ s } )_{ s \in \{ 0, 1, \ldots, K - 1 \} } ) $ are independent, 
	and 
	Hutzenthaler et al.~\cite[Lemma 2.2]{Overcoming} hence yield that for all
		$ n \in \N $, 
		$ x \in \mc O $ 
	it holds that
		\begin{equation} 
		\begin{split}
		& 
		\left( \sum_{ k = 0 }^{ K } 
		\alpha_{ k } \, \Exp{ \big| V^{ 0 }_{ k, n } - \Exp{ V^{ 0 }_{ k, n } } \! \big|^2 ( X^{ 0, K, x }_{ k } ) } \right)^{ \!\nicefrac12 }
		\\
		& \leq 
		\frac{ 1 }{ \sqrt{ M^n } } 
		\left[ \sum_{ k = 0 }^{ K } \alpha_{ k } \right]^{ \nicefrac12 }
		\left[ \left( \Exp{ \big| g ( X^{ 0, K, x }_{ 0 } ) \big|^2 } \right)^{ \!\nicefrac12 } 
		+ \sum_{ l = 0 }^{ K - 1 } \left( 
		\Exp{ \big| f_{ l } ( X^{ 0, K, x }_{ l }, 0 ) \big|^2 } \right)^{ \!\nicefrac12 } \right] 
		\\
		& + 
		\left[ \max_{ k,l \in [ 0, K ] \cap \N_0, k > l }  \left( \frac{ L_{ l } }{ \mf p_{ k, l } } \right) \right] 
		\left[ \sum_{ j = 1 }^{ n - 1 } 	\frac{ 1 }{ \sqrt{ M^{ n - j } } } 
		\left( \sum_{ l = 0 }^{ K - 1 } \beta_{ l } \, \Exp{ \big| 			
			V^{ 0 }_{ l, j } ( X^{ 0, K, x }_{ l } ) 
		 	- v_{ l } ( X^{ 0, K, x }_{ l } ) \big|^2}
		\right)^{\!\nicefrac12}
		\right]
		\\
		&  + 
		\left[ \max_{ k,l \in [ 0, K ] \cap \N_0, k > l }  \left( \frac{ L_{ l } }{ \mf p_{ k, l } } \right) \right] 
		\left[ \sum_{j=1}^{n-1} 
		\frac{ 1 }{ \sqrt{ M^{ n - j } } } \left( \sum_{ l = 0 }^{ K - 1 } \beta_{ l } \, \Exp{ \big| 	V^{ 0 }_{ l, j - 1 } ( X^{ 0, K, x }_{ l } ) - 		 	v_{ l } ( X^{ 0, K, x }_{ l } ) \big|^2 } \right)^{ \!\nicefrac12 } \right]\!.
		\end{split}
		\end{equation}
	This, \eqref{error_recursion:triangle_inequality}, and \eqref{error_recursion:estimation_of_the_bias_error} 
	therefore demonstrate that for all 
		$ n \in \N $, 
		$ x \in \mc O $ 
	it holds that 
		\begin{equation} 
		\begin{split}
		& 
		\left( \sum_{ k = 0 }^{ K } 
		\alpha_{ k } \, \Exp{ \big| V^{ 0 }_{ k, n } - v_{ k } \big|^2 ( X^{ 0, K, x }_{ k } ) } \right)^{\!\nicefrac12} 
		\\
		& \leq 
		\frac{ 1 }{ \sqrt{ M^n } } 
		\left[ \sum_{ k = 0 }^{ K } \alpha_{ k } \right]^{ \nicefrac12 }
		\left[ \left( \Exp{ \big| g ( X^{ 0, K, x }_{ 0 }  ) \big|^2 } \right)^{ \!\nicefrac12 }
		+ 
		\sum_{ l = 0 }^{ K - 1 } \left( 
		\Exp{ \big| f_{ l } ( X^{ 0, K, x }_{ l }, 0 ) \big|^2 }
		\right)^{ \!\nicefrac12 }
		\right]
		\\
		& + 
		2 \left[ \max_{ k,l \in [ 0, K ] \cap \N_0, k > l }  \left( \frac{ L_{ l } }{ \mf p_{ k, l } } \right) \right] 
		\left[ \sum_{ j = 0 }^{ n - 1 } 
		\frac{ 1 }{ \sqrt{ M^{ n - j - 1 } } } 
		\left( \sum_{ l = 0 }^{ K -1 } \beta_{ l } \, \Exp{ \big| V^{ 0 }_{ l, j } ( X^{ 0, K, x }_{ l } ) - 	 	v_{ l } ( X^{ 0, K, x }_{ l } ) \big|^2 } \right)^{ \!\nicefrac12 }
		\right]\!. 
		\end{split}
		\end{equation} 
	This establishes \eqref{error_recursion:claim}. 
	This completes the proof of \cref{lem:error_recursion}. 
\end{proof}

\subsection{Full error analysis for MLP approximations} 
\label{subsec:full_error_analysis}

\begin{lemma} \label{lem:special_alphas} 
	Let 
	 	$K\in\N$, 
	 	$L_0,L_1,\ldots,L_{K-1}\in [0,\infty)$
	satisfy
	 	$\sum_{j=0}^{K-1} L_j > 0$, 
	let 
	 	$\alpha_k^{(q)}\in [0,\infty)$, 
	 	$q \in \N_0$, $k \in \{0,1,\ldots,K\}$, 
	and assume for all 
	 	$k\in \{0,1,\ldots,K-1\}$, 
	 	$q\in \N$ 
	that 
		\begin{equation} \label{special_alphas:definition} 
		\begin{split}
		& 
		\alpha_K^{(0)} = 1, 
		\qquad 
		\alpha_k^{(0)} = 0, 
		\qquad 
		\text{and}\qquad
		\alpha_k^{(q)} = \frac{\left[\smallsum_{j=k}^{K-1} L_j\right]^{q}-\left[\smallsum_{j=k+1}^{K-1} L_j\right]^{q}}{q! \left[\smallsum_{j=0}^{K-1} L_j\right]^q } .
		\end{split} 
		\end{equation} 
	Then it holds for all 
	 	$k\in \{0,1,\ldots,K-1\}$, 
	 	$q\in\N_0$
	that 
	 	$
		\alpha_k^{(q+1)} \sum_{j=0}^{K-1} L_j 
		\geq 
		L_k \sum_{j=k+1}^{K} \alpha_j^{(q)} . 
		$
\end{lemma}

\begin{proof}[Proof of \cref{lem:special_alphas}] 
	First, note that for all 
 		$ k \in \{ 0, 1, \ldots, K - 1 \} $ 
	it holds that 
 		$ \sum_{ i = k + 1 }^{ K } \alpha^{ ( 0 ) }_{ i } = 1 $. 
	This and \eqref{special_alphas:definition} ensure for all 
 		$ k \in \{ 0, 1, \ldots, K - 1 \} $  
	that 
	 	\begin{equation} \label{special_alphas:anchor}
	    L_{ k } \left[ \smallsum_{ i = k + 1 }^{ K } \alpha^{ ( 0 ) }_{ i } \right]
	 	= L_{ k }
	 	= \left[ \smallsum_{ j = k }^{ K - 1 } L_{ j } \right]
	 	- \left[ \smallsum_{ j = k + 1 }^{ K - 1 } L_{ j } \right]
	 	= \alpha_{ k }^{ ( 1 ) } \left[ \smallsum_{ j = 0 }^{ K - 1 } L_{ j } \right]\!.
	    \end{equation} 
	In the next step we observe that the fact that for all 
		$ q \in \N_0 $, 
		$ a,b \in \R $ 
	it holds that 
 		$ a^{ q + 1 } - b^{ q + 1 } = ( a - b ) \sum_{ i = 0 }^{ q } a^{ i } b^{ q - i } $
	implies that for all 
 		$ q \in \N_{ 0 } $, 
 		$ a,b \in \R $ 
	with 
		$ a \geq b $ 
	it holds that 
		\begin{equation} 
		a^{ q + 1 } - b^{ q + 1 } = ( a - b )  \sum_{ i = 0 }^{ q } a^{ i } b^{ q - i } 
		\geq ( a - b ) \sum_{ i = 0 }^{ q } b^{ q } 
		= ( q + 1 ) ( a - b ) b^{ q }. 
		\end{equation} 
	This and \eqref{special_alphas:definition} ensure that for all 
	 	$ k \in \{ 0, 1, \ldots, K - 1 \} $, 
	 	$ q \in \N $ 
	it holds that 
		\begin{equation} 
		\begin{split}
		&
		( q + 1 )! \, L_{ k } \left[ \smallsum_{ j = k + 1 }^{ K } \alpha_{ j }^{ ( q ) } \right] \!
		\left[ \smallsum_{ j = 0 }^{ K - 1 } L_{ j } \right]^{ q + 1 }
		\\
		& = 
		( q + 1 ) L_{ k } \left[ \smallsum_{ j = k + 1 }^{ K - 1 } \left( \left[ \smallsum_{ l = j }^{ K - 1 } L_{ l } \right]^{ q } - \left[ \smallsum_{ l = j + 1 }^{ K - 1 } L_{ l } \right]^{ q } \right) \right]
		\! \left[ \smallsum_{ j = 0 }^{ K - 1 } L_{ j } \right]
		\\
		& = 
		( q + 1 ) L_{ k } \left[ \smallsum_{ j = k + 1 }^{ K - 1 } L_{ j } \right]^{ q } \left[\smallsum_{j=0}^{K-1} L_j\right]
		\leq 
		\left( \left[ \smallsum_{ j = k }^{ K - 1 } L_{ j } \right]^{ q + 1 }
		- \left[ \smallsum_{ j = k + 1 }^{ K - 1 } L_{ j } \right]^{ q + 1 } \right)
		\! \left[ \smallsum_{ j = 0 }^{ K - 1 } L_{ j }\right]
		\\
		& = 
		( q +  1)! \, \alpha_{ k }^{ ( q + 1 ) } \left[ \smallsum_{ j = 0 }^{ K - 1 } L_{ j } \right]^{ q + 2 }\!. 
		\end{split}
		\end{equation} 
	This and \eqref{special_alphas:anchor} establish that for all 
 		$ k \in \{ 0, 1, \ldots, K - 1 \} $, 
 		$ q \in \N_0 $ 
	it holds that 
	 	$ \alpha_{ k }^{ ( q + 1 ) } \sum_{ j = 0 }^{ K - 1 } L_{ j } \geq L_{ k } \sum_{ j = k + 1 }^{ K } \alpha_{ j }^{ ( q ) } $. 
	This completes the proof of \cref{lem:special_alphas}. 
\end{proof}

\begin{lemma} \label{lem:gronwall_type_inequality}
	Let $ M \in \N $, 
		$ N \in \N_0 $,
		$ \alpha, \beta, \kappa \in [ 0, \infty ) $, 
	let $ e_{ n, q } \in [ 0, \infty ) $, $ n, q \in \N_{ 0 } $, 
	satisfy for all 
		$ n \in \{ 0, 1, \ldots, N \} $, 
		$ q \in \N_{ 0 } $
	that 
		\begin{equation} 
		\label{gronwall_type_inequality:ass}
		e_{ n, q } \leq \frac{ \alpha \sqrt{ \kappa^{ q } } }{ \sqrt{ q! M^n } } 
		+ \beta \sum_{ l = 0 }^{ n - 1 } \frac{ e_{ l, q + 1 } }{ \sqrt{ M^{ n - 1 - l } } } . 
		\end{equation}
	Then it holds that
		\begin{equation} 
		\label{gronwall_type_inequality:claim}
		e_{ N, 0 } \leq \frac{ \alpha \exp( \tfrac{ \kappa M }{ 2 } ) }{ M^{ \nicefrac{ N }{ 2 } } } ( 1 + \beta )^N . 
 		\end{equation} 
\end{lemma} 

\begin{proof}[Proof of \cref{lem:gronwall_type_inequality}] 
	Throughout this proof assume w.l.o.g.~that 
		$ N \in \N $ 
	and let 
		$ \varepsilon_{ n } \in [ 0, \infty ) $, $ n \in \{ 0, 1, \ldots, N \} $, 
	satisfy for all 
		$ n \in \{ 0, 1, \ldots, N \} $ 
	that 
		\begin{equation} 
		\label{gronwall_type_inequality:epsilon_n}
		\varepsilon_{ n } = \sup \left\{ \frac{ e_{ n, q } }{ \sqrt{ M^k } } \colon k,q \in \N_{ 0 }, n + k + q = N \right\}\!.
		\end{equation} 
	Observe that \eqref{gronwall_type_inequality:ass} ensures that for all 
		$ n, k, q \in \N_0 $ 
	with 
		$ n + k + q = N $ 
	it holds that
		\begin{equation} 
		\begin{split}
		\frac{ e_{ n, q } }{ \sqrt{ M^k } } 
		&
		\leq \frac{ \alpha \sqrt{ \kappa^{ q } } }{ \sqrt{ q! M^{ n + k } } } 
		+ \beta \sum_{ l = 0 }^{ n - 1 } \frac{ e_{ l, q + 1 } }{ \sqrt{ M^{ n- 1 - l + k } } }
		= \frac{ \alpha \sqrt{ \kappa^{ q } } }{ \sqrt{ q! M^{ N - q } } }
		+ \beta \sum_{ l = 0 }^{ n - 1 } \frac{ e_{ l, q + 1 } }{ \sqrt{ M^{ N - ( q + 1 ) - l } } }
		\\
		& \leq \frac{ \alpha \sqrt{ \kappa^{ q } } }{ \sqrt{ q! M^{ N - q } } } 
		+ \beta \sum_{ l = 0 }^{ n - 1 } \varepsilon_{ l }.   
		\end{split}
		\end{equation} 
	The fact that for all 
		$ q \in \N_0 $ 
	it holds that 
		$ \frac{ \kappa^q M^q }{ q! } \leq e^{ \kappa M } $
	hence implies that for all 
		$ n, k, q \in \N_0 $ 
	with 
		$ n + k + q  = N $ 
	it holds that 
		\begin{equation} 
		\frac{ e_{ n, q } }{ \sqrt{ M^k } } 
		\leq \frac{ \alpha \exp( \tfrac{ \kappa M }{ 2 } ) }{ M^{ \nicefrac{ N }{ 2 } } }
		+ \beta \sum_{ l = 0 }^{ n - 1 } \varepsilon_{ l }.
		\end{equation} 
	This and \eqref{gronwall_type_inequality:epsilon_n} yield for all 
		$ n \in \{ 0, 1, \ldots, N \} $ 
	that 
		\begin{equation} 
		\varepsilon_n 
		= \sup\left\{ \frac{ e_{ n, q } }{ \sqrt{ M^k } } \colon k,q \in \N_0, n + k + q = N \right\}
		\leq \frac{ \alpha \exp( \tfrac{ \kappa M }{ 2 } ) }{ M^{ \nicefrac{ N }{ 2 } } }
		+ \beta \sum_{ l = 0 }^{ n - 1 } \varepsilon_{ l }. 
		\end{equation}
	\cref{discrete_gronwall} hence ensures for all 
		$ n \in \{ 0, 1, \ldots, N \} $ 
	that 
		\begin{equation} 
		\varepsilon_{ n } 
		\leq \frac{ \alpha \exp( \tfrac{ \kappa M }{ 2 } ) }{ M^{ \nicefrac{ N }{ 2 } } } ( 1 + \beta )^{ n }. 
		\end{equation}
	This establishes \eqref{gronwall_type_inequality:claim}. 
	This completes the proof of \cref{lem:gronwall_type_inequality}.
\end{proof}

\begin{prop} \label{prop:error_estimate}
	Let $ d,K,M \in \N $, 
		$ \Theta = \bigcup_{ n \in \N } \! \Z^n $, 
		$ \mc O \in \Borel( \R^d ) \setminus \{ \emptyset \} $, 
		$ c, L_{ 0 }, L_{ 1 }, \ldots, L_{ K } \in \R $,
	let $ f_{ k } \colon \mc O \times \R \to \R $, $ k \in \{ 0, 1, \ldots, K \} $, be $ \Borel( \mc O \times \R ) $/$ \Borel( \R ) $-measurable, 	
	let $ g \colon \mc O \to \R $ be $ \Borel( \mc O ) $/$ \Borel( \R ) $-measurable,  
	assume for all 
		$ k \in \{ 0, 1, \ldots, K \} $, 
		$ x \in \mc O $, 
		$ a,b \in \R $ 
	that  
		$ | ( f_{ k } ( x, a ) - a ) - ( f_{ k } ( x, b ) - b ) | \leq L_{ k } | a - b | $, 	
	let $ ( S, \mc S ) $ be a measurable space, 
	let $ \phi_{ k } \colon \mc O \times S \to \mc O $, $ k \in \{ 0, 1, \ldots, K \} $, be $ ( \Borel( \mc O ) \otimes \mc  S ) $/$  \Borel( \mc O )$-measurable, 
	let $ ( \Omega, \mc F, \P ) $ be a probability space, 
	let $ W^{ \theta } = ( W^{ \theta }_{ k } )_{ k \in \{0, 1, \ldots, K \} } \colon \{ 0, 1, \ldots, K \} \times \Omega \to S $, $ \theta \in \Theta $, be i.i.d.\ stochastic processes, 
	assume for all 
		$ \theta \in \Theta $ 
	that 
		$ W^{ \theta }_{ 0 }, W^{ \theta }_{ 1 }, \ldots, W^{ \theta }_{ K - 1 } $ are independent,
	let $ \mc{R}^{ \theta } = (\mc{R}^{\theta}_{k})_{k \in \{ 0, 1, \ldots, K\}} \colon \{ 0, 1, \ldots, K\} \times \Omega \to \N_0 $, $ \theta \in \Theta $, be i.i.d.\ stochastic processes, 
	assume for every 
		$ k \in \{1, 2, \ldots, K\} $ 
	that $ \mc{R}^{ \theta }_{ k } \leq k-1 $,  
	let $ \mf p_{ k, l } \in ( 0, \infty ) $, $ k \in \{ 1, 2, \ldots, K \} $, $ l \in \N_0 $, satisfy for all 
		$ k \in \{ 1, 2, \ldots, K \} $, 
		$ l \in \N_{ 0 } $ 
	that 
		$ \mf p_{ k, l } \P( \mc{R}^{ 0 }_{ k } = l ) = | \P ( \mc{R}^{ 0 }_{ k } = l ) |^2 $, 
	assume for all 
		$ k, l \in \N_0 $ 
	with $ l < k \leq K $ that $ L_l \leq c \P( \mc R^{ 0 }_{ k } = l ) $, 
	assume that $ ( \mc{R}^{ \theta })_{ \theta \in \Theta } $ and $ ( W^{ \theta }_{ k } )_{ ( \theta, k ) \in \Theta \times \{ 0, 1, \ldots, K \}} $ are independent, 
	let $ X^{ \theta, k } = ( X^{ \theta, k, x }_{ l } )_{ ( l, x ) \in \{ 0, 1, \ldots, k \} \times \mc O } \colon \{ 0, 1, \ldots, k \} \times \mc O \times \Omega \to \mc O $, $ k \in \{ 0, 1, \ldots, K \} $, $ \theta \in \Theta $, satisfy for all 
		$ \theta \in \Theta $, 
		$ k \in \{ 0, 1, \ldots, K \} $, 
		$ l \in \{ 0, 1, \ldots, k \} $, 
		$ x \in \mc O $ 
	that 
		\begin{equation}
		X^{ \theta, k, x }_{ l }
		= 
		\begin{cases}
		x  & \colon l = k  \\
		\phi_{ l } ( X^{ \theta, k, x }_{ l + 1 }, W^{ \theta }_{ l } ) & \colon l<k,  
		\end{cases}
		\end{equation}	
	assume for all 
		$ k \in \{ 0, 1, \ldots, K \} $, 
		$ x	\in \mc O $
	that 
		$ \EXP{ | g ( X^{ 0, k, x }_{ 0 } ) |^2 + \sum_{ l = 0 }^{ k - 1 } | f_{ l } ( X^{ 0, k, x }_{ l }, 0 ) |^2 } < \infty $, 
	let 
		$ V^{ \theta }_{ k, n } \colon \mc O \times \Omega \to \R $, $ n \in \N_0 $, $ k \in \{ 0, 1, \ldots, K \} $, $ \theta \in \Theta $,
	satisfy for all 
		$ \theta \in \Theta $, 
		$ k \in \{ 0, 1, \ldots, K \} $, 
		$ n \in \N_{ 0 } $, 
		$ x \in \mc O $ 
	that 
		\begin{equation} 
		\begin{split}
		V^{ \theta }_{ k, n } ( x ) 
		& = 
		\frac{ \mathbbm{ 1 }_{ \N }(n) }{ M^n } \sum_{ m = 1 }^{ M^n } \left[ g ( X^{ ( \theta, 0, -m ), k, x }_{ 0 } ) + \sum_{ l = 0 }^{ k - 1 } f_{ l } ( X^{ ( \theta, 0, m ), k, x }_{ l }, 0 ) \right] 
		+
		\sum_{ j = 1 }^{ n - 1 } 
		\frac{ 1 }{ M^{ n - j } }
		\sum_{ m = 1 }^{ M^{ n - j } } 
		\frac{ \mathbbm{1}_{\N}(k) }{ \mf p_{ k, \mc{R}^{ ( \theta, j, m ) }_{ k } } }
		\\  
		& 
		\cdot 
		\bigg[ 
		\Big( f_{ \mc{R}^{ ( \theta, j, m ) }_{ k } } \big( X^{ ( \theta, j, m ), k, x }_{ \mc{R}^{ ( \theta, j, m ) }_{ k } },  V^{ ( \theta, j, m ) }_{ \mc{R}^{ ( \theta, j, m ) }_{ k }, j } ( X^{ ( \theta, j, m ), k, x }_{ \mc{R}^{ ( \theta, j, m ) }_{ k } } ) \big)  
		-
		V^{ ( \theta, j, m ) }_{ \mc{R}^{ ( \theta, j, m ) }_{ k }, j } ( X^{ ( \theta, j, m ), k, x }_{ \mc{R}^{ ( \theta, j, m ) }_{ k } } )
		\Big) 
		\\[1ex] 
		& - 
		\Big( f_{ \mc{R}^{ ( \theta, j, m ) }_{ k } } \big( X^{ ( \theta, j, m ), k, x }_{ \mc{R}^{ ( \theta, j, m ) }_{ k } },  V^{ ( \theta, j, -m ) }_{ \mc{R}^{ ( \theta, j, m ) }_{ k }, j - 1 } ( X^{ ( \theta, j, m ), k, x }_{ \mc{R}^{ ( \theta, j, m ) }_{ k } } ) \big) 
		- V^{ ( \theta, j, -m ) }_{ \mc{R}^{ ( \theta, j, m ) }_{ k }, j - 1 } ( X^{ ( \theta, j, m ), k, x }_{ \mc{R}^{ ( \theta, j, m ) }_{ k } } ) 
		\Big)
		\bigg],
		\end{split}
		\end{equation}
	and let $ v_{ 0 }, v_{ 1 }, \ldots, v_{ K } \colon \mc O \to \R$ be $ \Borel( \mc O ) $/$ \Borel( \R ) $-measurable functions which satisfy for all 
		$ k \in \{ 1, 2, \ldots, K \} $, 
		$ x \in \mc O $ 
	that 
		$ v_{ 0 } ( x ) = g ( x ) $ 
	and 
		\begin{equation} 
		v_{ k } ( x )  
		= \EXPP{ f_{ k - 1 } ( X^{ 0, k, x }_{ k - 1 }, v_{ k - 1 } ( X^{ 0, k, x }_{ k - 1 } ) ) }
		= \EXPP{ f_{ k - 1 } ( \phi_{ k - 1 } ( x, W^{ 0 }_{ k - 1 } ), v_{ k - 1 }( \phi_{ k - 1 } ( x, W^{ 0 }_{ k - 1 } ) ) ) }
		\end{equation} 
	(cf.~\cref{lem:welldefinedness_exact_solution}). 
	Then it holds for all 
		$ N \in \N_0 $, 
		$ x \in \mc O $ 
	that
		\begin{equation}\label{error_estimate:claim} 
		\begin{split}
			\left( \EXPP{ | V^{ 0 }_{ K, N } ( x ) - v_{ K } ( x ) |^2 } \right)^{\!\nicefrac12}
			& \leq \frac{ \exp( \tfrac{ M }{ 2 } + \sum_{ j = 0 }^{ K - 1 } L_{ j } ) }{ M^{ \nicefrac{ N }{ 2 } } }
			\left[ 1 + 2 \max_{ k,l \in \N_0, l < k \leq K }  \left( \frac{ L_{ l } }{ \mf p_{ k, l } } \right)  \right]^{N}
			\\
			& \quad \cdot 
			\left[ \left( \Exp{ \big| g ( X^{ 0, K, x }_{ 0 } ) \big|^2 } \right)^{\!\nicefrac12} 
			+ \sum_{ l = 0 }^{ K - 1 } \left( \Exp{ \big| f_{ l } ( X^{ 0, K, x }_{ l }, 0 ) \big|^2 } \right)^{ \!\nicefrac{1}{2} } \right]\!.
		\end{split}
		\end{equation}
\end{prop}

\begin{proof}[Proof of \cref{prop:error_estimate}] 
	We first prove \eqref{error_estimate:claim} in the case that $\sum_{k=0}^{K-1} L_k > 0$. 
	For this let
	 	$ x \in \mc O $, 
 		$ N \in \N_0 $, 
	let $ e_{ n, q } \in [ 0, \infty ) $, $ n,q \in \N_0 $, 
	let $ \alpha_{ k }^{ ( q ) } \in [ 0, \infty ) $, $ q \in \N_0 $, $ k \in \{ 0, 1, \ldots, K \} $, satisfy for all  
		$ k \in \{ 0, 1, \ldots, K - 1 \} $, 
		$ q \in \N $ 
	that 
		\begin{equation} \label{error_estimate:definition_a_q_k}
		\begin{split}
		& 
		\alpha_{ K }^{ ( 0 ) } = 1, 
		\qquad 
		\alpha_{ k }^{ ( 0 ) } = 0 = \alpha_{ K }^{ ( q ) }, 
		\qquad\text{and}\qquad
		\alpha_{ k }^{ ( q ) } = \frac{ \left[ \sum_{ j = k }^{ K - 1 } L_{ j } \right]^{ q } - \left[ \sum_{ j = k + 1 }^{ K - 1 } L_{ j } \right]^{ q } }{ q! \left[ \sum_{ j = 0 }^{ K - 1 } L_{ j } \right]^q }, 
		\end{split} 
		\end{equation} 
	and assume for all 
	 	$ n,q \in \N_0 $
	that 
		\begin{equation} \label{error_estimate:definition_e_n_q}
		e_{ n, q }^2 = \sum_{ k = 0 }^{ K } \alpha^{ ( q ) }_{ k } \, \Exp{ \big| V^{ 0 }_{ k, n } ( X^{ 0, K, x }_{ k } ) - v_{ k } ( X^{ 0, K, x }_{ k } ) \big|^2 }\!. 
		\end{equation} 
	Observe that \eqref{error_estimate:definition_a_q_k}, \eqref{error_estimate:definition_e_n_q}, \cref{lem:error_recursion}, and \cref{lem:special_alphas} ensure that for all 
	 	$ n \in \N $, 
	 	$ q \in \N_0 $ 
	it holds that 
		\begin{equation} 
		\begin{split} 
		& e_{ n, q } \leq \frac{ 1 }{ \sqrt{ M^n } } \left[ \sum_{ k = 0 }^{ K } \alpha_{ k }^{ ( q ) } \right]^{ \nicefrac12 }
		\left[ \left( \Exp{ \big| g ( X^{ 0, K, x }_{ 0 } ) \big|^2 } \right)^{ \!\nicefrac12 }
		+ \sum_{ l = 0 }^{ K - 1 } \left( \Exp{ \big| f_{ l } ( X^{ 0, K, x }_{ l }, 0 ) \big|^2 } \right)^{\!\nicefrac12} \right] 
		\\
		& + 
		2 \left[ \max_{ k,l \in \N_0, l < k \leq K }  \left( \frac{ L_{ l } }{ \mf p_{ k, l } } \right) \right]
		\sum_{ j = 0 }^{ n - 1 } \frac{ 1 }{ \sqrt{ M^{ n - j - 1 } } } \left( \sum_{ k = 0 }^{ K - 1 } \alpha_{ k }^{ ( q + 1 ) }\,\Exp{ \big| V^{ 0 }_{ k, j } ( X^{ 0, K, x }_{ k } ) - v_{ k } ( X^{ 0, K, x }_{ k } ) \big|^2 } 
		\right)^{\!\nicefrac12}
		.
		\end{split}
		\end{equation} 
	The fact that for all 
		$ q \in \N_0 $ 
	it holds that 
		$ \sum_{ k = 0 }^{ K } \alpha^{ ( q ) }_{ k } = \frac{1}{q!} $ 
	and \eqref{error_estimate:definition_e_n_q} therefore ensure that for all 
		$ n \in \N $, 
		$ q \in \N_0 $ 
	it holds that 
		\begin{equation} \label{error_estimate:recursive_inequality_n_geq_one}
		\begin{split} 
		  e_{n,q}
		  & 
		  \leq 
		  \frac{ 1 }{ \sqrt{ q! M^{ n } } } \left[ \left( \Exp{ \big| g ( X^{ 0, K, x }_{ 0 } ) \big|^2 } \right)^{ \!\nicefrac12 } 
		  + \sum_{ l = 0 }^{ K - 1 } \left( \Exp{ \big| f_{ l } ( X^{ 0, K, x }_{ l }, 0 ) \big|^2 } \right)^{ \!\nicefrac{1}{2} }
		  \right] 
		  \\
		  & + 
		  2\left[ \max_{ k,l \in \N_0, l < k \leq K }  \left( \frac{ L_{ l } }{ \mf p_{ k, l } } \right) \right]
		  \sum_{ j = 0 }^{ n - 1 } \frac{ e_{ j, q + 1 } }{ \sqrt{ M^{ n - j - 1 } } } .
		 \end{split}
		 \end{equation}
	Moreover, note that \cref{lem:a_priori_estimates} ensures that
		\begin{equation} 
		\begin{split}
		& \sup_{ k \in \{ 0, 1, \ldots, K \} } \left( \Exp{ \big| v_{ k } ( X^{ 0, K, x }_{ k } ) \big|^2 } \right)^{ \!\nicefrac12 }
		\\
		& \leq \exp\!\left( \sum_{ j = 0 }^{ K - 1 } L_{ j } \right) \!
		\left[ \left( \Exp{ \big| g ( X^{ 0, K, x }_{ 0 } ) \big|^2 } \right)^{ \!\nicefrac12 } 
		+ 
		\sum_{l=0}^{K-1}
		\left( \Exp{ \big| f_{ l } ( X^{ 0, K, x }_{ l }, 0 ) \big|^2 } \right)^{ \!\nicefrac{ 1 }{ 2 } }
		\right]\!.
		\end{split} 
		\end{equation} 
	This, the fact that for all 
		$ q \in \N_0 $ 
	it holds that 
		$ \sum_{k=0}^{K} \alpha_k^{(q)} = \frac{1}{q!} $, 
	and \eqref{error_estimate:definition_e_n_q} ensure that for all 
		$ q \in \N_0 $ 
	it holds that 
		\begin{equation} 
		\begin{split}
		e_{ 0, q } 
		& 
		= \left[ \sum_{ k = 0 }^{ K } \alpha_{ k }^{ ( q ) } \, \Exp{ \big| v_{ k } ( X^{ 0, K, x }_{ k } ) \big|^2 } \right]^{\nicefrac12} 
		\leq \left[ \sum_{ k = 0 }^{ K } \alpha_{ k }^{ ( q ) } \right]^{ \nicefrac{ 1 }{ 2 } } \left[ \sup_{ k \in \{ 0, 1, \ldots, K \} } \left( \Exp{ \big| v_{ k } ( X^{ 0, K, x }_{ k } ) \big|^2 } \right)^{\!\nicefrac12} \right] 
		\\
		& \leq \frac{ 1 }{ \sqrt{ q! } }  \exp\!\left( \sum_{ j = 0 }^{ K - 1 } L_{ j } \right) \!
		\left[ \left( \Exp{ \big| g ( X^{ 0, K, x }_{ 0 } ) \big|^2 } \right)^{ \!\nicefrac12 } 
		+ \sum_{ l = 0 }^{ K - 1 }
		\left( \Exp{ \big| f_{ l } ( X^{ 0, K, x }_{ l }, 0 ) \big|^2 } \right)^{ \!\nicefrac{1}{2} }
		\right]\!.
		\end{split}
		\end{equation}
	Combining this with \eqref{error_estimate:recursive_inequality_n_geq_one} proves that for all
		$ n, q \in \N_0 $ 
	it holds that 
		\begin{equation} 
		\begin{split} 
		e_{n,q}
		& 
		\leq 
		\frac{ \exp\!\big( \sum_{ j = 0 }^{ K - 1 } L_{ j } \big) \! }{ \sqrt{ q! M^{ n } } } \left[ \left( \Exp{ \big| g ( X^{ 0, K, x }_{ 0 } ) \big|^2 } \right)^{ \!\nicefrac12 } 
		+ \sum_{ l = 0 }^{ K - 1 } \left( \Exp{ \big| f_{ l } ( X^{ 0, K, x }_{ l }, 0 ) \big|^2 } \right)^{ \!\nicefrac{1}{2} }
		\right] 
		\\
		& + 
		2\left[ \max_{ k,l \in \N_0, l < k \leq K }  \left( \frac{ L_{ l } }{ \mf p_{ k, l } } \right) \right]
		\sum_{ j = 0 }^{ n - 1 } \frac{ e_{ j, q + 1 } }{ \sqrt{ M^{ n - j - 1 } } } .
		\end{split}
		\end{equation} 
	\cref{lem:gronwall_type_inequality}, \eqref{error_estimate:definition_a_q_k}, and \eqref{error_estimate:definition_e_n_q} hence demonstrate that 
		\begin{equation} \label{error_estimate:first_estimate}
		\begin{split}
		& \left( \Exp{ | V^{ 0 }_{ K, N } ( x ) - v_{ K } ( x ) |^2 } \right)^{\!\nicefrac12} 
		= e_{ N, 0 } 
		\leq \left[ 1 + 2 \max_{ k,l \in \N_0, l < k \leq K }  \left( \frac{ L_{ l } }{ \mf p_{ k, l } } \right)  \right]^{N}
		\\ 
		&  
		\quad \cdot \exp\!\left( \sum_{ j = 0 }^{ K - 1 } L_{ j } \right) \left[ \left( \Exp{ \big| g ( X^{ 0, K, x }_{ 0 } ) \big|^2 } \right)^{ \!\nicefrac12 } 
		+ \sum_{ l = 0 }^{ K - 1 } \left( \Exp{ \big| f_{ l } ( X^{ 0, K, x }_{ l }, 0 ) \big|^2 } \right)^{ \!\nicefrac{1}{2} } \right]  
		\frac{ e^{ \nicefrac{ M }{ 2 } } }{ M^{ \nicefrac{ N }{ 2 } } } . 
		\end{split} 
		\end{equation}
	This establishes \eqref{error_estimate:claim} in the case that $\sum_{k=0}^{K-1} L_k > 0$. 
	Next we prove \eqref{error_estimate:claim} in the case that $\sum_{k=0}^{K-1} L_k \geq 0$.
	For this, we note that \eqref{error_estimate:claim} in the case that $ \sum_{k=0}^{K-1} L_k > 0$ ensures that for all 
		$ N \in \N_0 $, 
		$ x \in \mc O $, 
		$ \eta \in (0,\infty) $ 
	it holds that 
		\begin{equation} 
		\begin{split}
		\left( \Exp{ | V^{ 0 }_{ K, N } ( x ) - v_{ K } ( x ) |^2 } \right)^{ \!\nicefrac12 }
		& \leq 
		\exp\!\left( K \eta + \sum_{ j = 0 }^{ K - 1 } L_{ j } \right) \!
		\frac{ e^{ \nicefrac{ M }{ 2 } } }{ M^{ \nicefrac{ N }{ 2 } } } \left[ 1 + 2 \max_{ k,l \in \N_0, l < k \leq K }  \left( \frac{ L_{ l } + \eta }{ \mf p_{ k, l } } \right) 
		\right]^{N}\\
		& 
		\cdot 
		\left[  
		\left( \Exp{ \big| g( X^{ 0, K, x }_{ 0 } ) \big|^2 } \right)^{\!\nicefrac12} 
		+ \sum_{ l = 0 }^{ K - 1 } \left( \Exp{ \big| f_{ l } ( X^{ 0, K, x }_{ l }, 0 ) \big|^2 } \right)^{ \!\nicefrac{ 1 }{ 2 } }
		\right]\!.
		\end{split}
		\end{equation} 
	Taking the infimum over $ \eta \in (0,\infty) $ hence establishes \eqref{error_estimate:claim}. 
	This completes the proof of \cref{prop:error_estimate}.
\end{proof}

\begin{cor} \label{prop:exponential_Euler_error_estimate}
	Let $ d,K,M \in \N $, 
		$ \Theta = \bigcup_{ k \in \N } \! \Z^{ k } $, 
		$ L,T, t_{ 0 }, t_{ 1 }, \ldots, t_{ K } \in \R $ 
	satisfy $ 0 = t_{ 0 } < t_{ 1 } < \ldots < t_{ K } = T $, 
	let $ f \in C( [ 0, T ] \times \R^{ d } \times \R, \R) $, 
	let $ v_{ k } \colon \R^d \to \R $, $ k \in \{ 0, 1, \ldots, K \} $, be $ \Borel( \R^d ) $/$ \Borel( \R ) $-measurable, 
	let $ ( \Omega, \mc F, \P ) $ be a probability space, 
	let $ W^{ \theta } \colon [ 0, T ] \times \Omega \to \R^d $, $ \theta \in \Theta $, be i.i.d.\ standard Brownian motions, 
	let $ R^{ \theta } \colon \Omega \to ( 0, 1 ) $, $ \theta \in \Theta $, be i.i.d.\,random variables, 
	assume for all 
		$ r \in (0, 1) $ 
	that 
		$ \P( R^{ 0 } \leq r ) = r $, 
	assume that $ ( R^{ \theta } )_{ \theta \in \Theta } $ and $ ( W^{ \theta } )_{ \theta \in \Theta } $ are independent, 
	let $ \mc{R}^{ \theta } = ( \mc{R}^{ \theta }_{ k } )_{ k \in \{ 0, 1, \ldots, K \} } \colon \{ 0, 1, \ldots, K \} \times \Omega \to \N_{ 0 } $, $ \theta \in \Theta $, satisfy for all  
		$ \theta \in \Theta $, 
		$ k \in \{ 0, 1, \ldots, K \} $ 
	that 
		$ \mc{R}^{ \theta }_{ k } = \max\{ n \in \N_0 \colon t_{ n } \leq t_{ k } R^{ \theta } \} $, 
	assume for all 
		$ k \in \{ 1, 2, \ldots, K \} $, 
		$ t   \in [ 0, T ] $, 
		$ x \in \R^d $, 
		$ a,b \in \R $ 
	that  
		$ | f ( t, x, a ) - f ( t, x, b ) | \leq L | a - b | $, 
		$ \EXP{ | v_0 ( x + W^{ 0 }_{ T } ) |^2 + \sum_{ l = 0 }^{ K - 1 } | f ( t_{ l }, x + W^{ 0 }_{ T - t_{ l } }, 0 ) |^2 } < \infty $, 
	and 
		\begin{equation}	\label{toy_example_final_error_estimate:exact_solution_ass}
		v_{ k } ( x ) = \EXPP{ v_{ k - 1 } ( x + W^{ 0 }_{ t_{ k } - t_{ k - 1 } } ) + ( t_{ k } - t_{ k - 1 } )
			f ( t_{ k - 1 }, x + W^{ 0 }_{ t_{ k } - t_{ k - 1 } }, v_{ k - 1 } ( x + W^{ 0 }_{ t_{ k } - t_{ k - 1 } } ) ) }
		\end{equation}
	(cf.~\cref{cor:well_definedness_exact_solution_brownian_case}),
	and let $ V^{ \theta }_{ k, n } \colon \R^d \times \Omega \to \R $, $ n \in \N_0 $, $ k \in \{ 0, 1, \ldots, K \} $, $ \theta \in \Theta $, satisfy for all
		$ \theta \in \Theta $, 
		$ k	\in \{ 0, 1, \ldots, K \} $, 
		$ n \in \N_0 $, 
		$ x \in \R^d $ 
	that 
		\begin{equation}
		\begin{split}
		& V^{ \theta }_{ k, n } ( x ) 
		= 
		\frac{ \mathbbm{ 1 }_{ \N }( n ) }{ M^{ n } } \sum_{ m = 1 }^{ M^{ n } } 
		\left[ v_0( x + W^{ ( \theta, 0, -m ) }_{ t_{ k } } )
		+ \sum_{ l = 0 }^{ k - 1 } ( t_{ l + 1 } - t_{ l } ) f( t_{ l }, x + W^{ ( \theta, 0, m ) }_{ t_{ k } } - W^{ ( \theta, 0, m ) }_{ t_{ l } }, 0 )	\right]  
		\\ 
		& + \sum_{ j = 1 }^{ n - 1 } \frac{ t_{ k } }{ M^{ n - j } } \sum_{ m = 1 }^{ M^{ n - j } }
		\bigg[
		f \big( t_{ \mc{R}^{ ( \theta, j, m ) }_{ k } }, 
		x + W^{ ( \theta, j, m ) }_{ t_{ k } } - W^{ ( \theta, j, m ) }_{ t_{ \mc{R}^{ ( \theta, j, m ) }_{ k } } }, 
		V^{ ( \theta, j, m ) }_{ \mc{R}^{ ( \theta, j, m ) }_{ k }, j } ( x + W^{ ( \theta, j, m ) }_{ t_{ k } } - W^{ ( \theta, j, m ) }_{ t_{ \mc{R}^{ ( \theta, j, m ) }_{ k } } } ) \big)
		\\
		& -
		f \big( t_{ \mc{R}^{ ( \theta, j, m ) }_{ k } }, 
		x + W^{ ( \theta, j, m ) }_{ t_{ k } } - W^{ ( \theta, j, m ) }_{ t_{ \mc{R}^{ ( \theta, j, m ) }_{ k } } },
		V^{ ( \theta, j, -m ) }_{ \mc{R}^{ ( \theta, j, m ) }_{ k }, j - 1 } ( x + W^{ ( \theta, j, m ) }_{ t_{ k } } - W^{ ( \theta, j, m ) }_{ t_{ \mc{R}^{ ( \theta, j, m ) }_{ k } } } ) \big) 
		\bigg]. 
		\end{split}
		\end{equation}
	Then it holds for all 
		$ N \in \N_0 $, 
		$ x \in \R^d $ 
	that 
		\begin{multline}\label{toy_example_final_error_estimate:claim}
		\left( \Exp{ \big| V^{ 0 }_{ K, N } ( x ) - v_{ K } ( x ) \big|^2 } \right)^{ \nicefrac12 }   
		\leq 
		\frac{ \exp( LT + \tfrac{ M }{ 2 } ) [ 1 + 2LT ]^N }{ M^{ \nicefrac{ N }{ 2 } } }
		\\ 
		\cdot   
		\left[ \left( \Exp{ \big| v_0 ( x + W^{ 0 }_{ T } ) \big|^2 } \right)^{ \!\nicefrac12 }
		+ \sum_{ l = 0 }^{ K - 1 } ( t_{ l + 1 } - t_{ l } ) \left( \Exp{ \big| f ( t_{ l }, x+W^0_{ T  - t_{ l } }, 0 ) \big|^2}\right)^{\!\nicefrac12} \right]\!.
		\end{multline}
\end{cor}

\begin{proof}[Proof of \cref{prop:exponential_Euler_error_estimate}]
	Throughout this proof let $ \norm{\cdot}\colon \R^d \to [0,\infty) $ be the standard norm on $ \R^d $ 
	and let $ \mf p_{ k, l } \in [ 0, \infty ) $, $ k \in \{ 1, 2, \ldots, K \} $, $ l \in \N_{ 0 } $, satisfy for all 
		$ k \in \{ 1, 2, \ldots, K \} $ 
	that 
		\begin{equation} 
		\mf p_{ k, l } = 
		\begin{cases} 
		\frac{ t_{ l + 1 } - t_{ l } }{ t_{ k } } & \colon l < k \\
		1 & \colon l \geq k. 
		\end{cases} 
		\end{equation}  
	Observe that the assumption that $ W^0 $ is a Brownian motion and the assumption that for all 
		$ x \in \R^d $ 
	it holds that 
		$ \EXP{ | v_0 ( x + W^{ 0 }_{ T } ) |^2 + \sum_{l=0}^{K-1} |f(t_l, x+W^{0}_{T-t_l}, 0)|^2 } < \infty $  
	show that for all 
		$ k \in \{ 0, 1, \ldots, K \} $, 
		$ x \in \R^d $
	it holds that 
		$ \EXP{ | v_0( x + W^{ 0 }_{ t_{ k } } ) |^2 + \sum_{ l = 0 }^{ k - 1 } | f ( t_{ l }, x + W^{ 0 }_{ t_{ k } } - W^{ 0 }_{ t_{ l } }, 0 ) |^2 } < \infty $. 
	In addition, note that \eqref{toy_example_final_error_estimate:exact_solution_ass} and  the fact that $W^0\colon [0,T]\times\Omega \to \R^d$ is a standard Brownian motion ensure that for all 
		$k \in \{1,2,\ldots,K\}$, 
		$x \in \R^d$ 
	it holds that 
		\begin{multline}  
		v_k(x) 
		= 
		\EXP{ 
			v_{k-1}(x + W^0_{t_k}-W^0_{t_{k-1}}) 
		\\  
		+ 
		(t_k-t_{k-1})\,
		f(t_{k-1},x + W^0_{t_k}-W^0_{t_{k-1}},v_{k-1}(x 
			+ 
			W^0_{t_k}-W^0_{t_{k-1}}))
		}.
		\end{multline} 
	Moreover, note that \cref{prop:error_estimate} (applied with 
		$ d \is d $, 
		$ K \is K $, 
		$ M \is M $,
		$ \Theta \is \Theta $,
		$ \mc O \is \R^d $, 
		$ ( f_{ k } )_{ k \in \{ 0, 1, \ldots, K \} } \is ( ( \R^d \times \R \ni ( x, a ) \mapsto a + ( t_{ \min\{ k + 1, K \} } - t_{ k } ) f( t_{ k }, x, a ) \in \R ) )_{ k \in \{ 0, 1, \ldots, K \} } $, 
		$ g \is v_0 $, 
		$ c \is LT $,
		$ ( L_{ k } )_{ k \in \{ 0, 1, \ldots, K \} } \is ( ( t_{ \min\{ k + 1, K \} } - t_{ k } ) L )_{ k \in \{ 0, 1, \ldots, K \} } $, 
		$ ( S, \mc S ) \is ( \R^d, \Borel( \R^d ) ) $, 
		$ ( \phi_{ k } )_{ k \in \{ 0, 1, \ldots, K \} } \is ( \R^d\times\R^d \ni (x,w) \mapsto x + w \in \R^d )_{ k \in \{ 0, 1, \ldots, K \} } $,
		$ ( \Omega, \mc F, \P ) \is ( \Omega, \mc F, \P ) $, 
		$ ( W^{ \theta }_{ k } )_{ ( \theta, k ) \in \Theta \times \{ 0, 1, \ldots, K \} } \is ( W^{ \theta }_{ t_{ \min\{ k + 1, K \} } } - W^{ \theta }_{ t_{ k } } )_{ ( \theta, k ) \in \Theta \times \{ 0, 1, \ldots, K \} } $, 
		$ ( \mc{R}^{ \theta } )_{ \theta \in \Theta } \is ( \mc{R}^{ \theta } )_{ \theta \in \Theta } $, 
		$ ( \mf p_{ k, l } )_{ ( k, l ) \in \{ 1, 2, \ldots, K \} \times \N_{ 0 } } \allowbreak \is ( \mf p_{ k, l } )_{ ( k, l ) \in \{ 1, 2, \ldots, K \} \times \N_{ 0 } } $
	in the notation of \cref{prop:error_estimate}) and the fact that $ W^{ 0 } $ is a standard Brownian motion therefore demonstrate that for all 
		$ N \in \N_0 $, 	
		$ x \in \R^d $ 
	it holds that 
		\begin{align}
		& \nonumber
		\left( \Exp{ \big| V^{ 0 }_{ K, N } ( x ) - v_{  K } ( x ) \big|^2 } \right)^{ \!\nicefrac12 }
		\leq 
		\frac{ e^{ \nicefrac{ M }{ 2 } } }{ M^{ \nicefrac{ N }{ 2 } } } 
		\exp\!\left( \sum_{ l = 0 }^{ K - 1 } ( t_{ l + 1 } - t_{ l } ) L \right) 
		\left[ 1 + 2 \max_{ k, l \in \{ 0, 1, \ldots, K \}, k > l } \frac{ ( t_{ l + 1 } - t_{ l } ) L }{ \mf p_{ k, l } } \right]^{ N }
		\\
		& \quad \cdot \left[ 
		\left( \Exp{ \big| v_0 ( x + W^{ 0 }_{ T } ) \big|^2 } \right)^{\!\nicefrac12} 
		+ \sum_{ l = 0 }^{ K - 1 } \left( \Exp{ \big| ( t_{ l + 1 } - t_{ l } ) f ( t_{ l }, x + W^{ 0 }_{ T } - W^{ 0 }_{ t_{ l } }, 0 ) \big|^2 } \right)^{\!\nicefrac12} \right] 
		\\ \nonumber
		& = 
		\frac{ e^{ \nicefrac{ M }{ 2 } } e^{ LT } [ 1 + 2 L T ]^{ N } }{ M^{ \nicefrac{ N }{ 2 } } }
		\left[ \left( \Exp{ \big| v_0 ( x + W^{ 0 }_{ T } ) \big|^2 } \right)^{\!\nicefrac12} + \sum_{ l = 0 }^{ K - 1 } ( t_{ l + 1 } - t_{ l } ) \!\left( \Exp{ \big| f ( t_{ l }, x + W^{ 0 }_{ T - t_{ l } }, 0 ) \big|^2 } \right)^{ \!\nicefrac12 }
		\right]\!.
		\end{align} 
	This establishes \eqref{toy_example_final_error_estimate:claim}. 
	This completes the proof of \cref{prop:exponential_Euler_error_estimate}. 
\end{proof}

\begin{cor} \label{corollary_at_the_end_of_the_day}
	Let $ d,K,M \in \N $, 
		$ \Theta = \bigcup_{ n \in \N } \! \Z^n $, 
		$ L,T \in [ 0, \infty ) $,  
		$ f \in C( \R, \R ) $
	satisfy for all 
		$ a,b \in \R $  
	that 
		$ | ( f ( a ) - a ) - ( f ( b ) - b ) | \leq L | a - b | $,  
	let $ ( \Omega, \mc F, \P ) $ be a probability space, 
	let $ W^{ \theta } \colon [0,T] \times \Omega \to \R^d $, $ \theta \in \Theta $, be i.i.d.\ standard Brownian motions, 
	let $ R^{ \theta } \colon \Omega \to ( 0, 1 ) $, $ \theta \in \Theta $, be i.i.d.\,random variables, 
	assume for all 
		$ r \in (0, 1) $ 
	that 
		$ \P( R^0 \leq r ) = r $, 
	assume that 
		$ ( R^{ \theta } )_{ \theta \in \Theta } $ 
	and 
		$ ( W^{ \theta } )_{ \theta \in \Theta } $ 
	are independent, 
	let $ \mc{R}^{ \theta }_{ k } \colon \Omega \to \{ 0, 1, \ldots, K \} $, $ \theta \in \Theta $, $ k \in \{ 0, 1, \ldots, K \} $,  
	satisfy for all 
		$ \theta \in \Theta $, 
		$ k \in \{ 0, 1, \ldots, K \} $ 
	that 
		$ \mc{R}^{ \theta }_{ k } = \max\{ n \in \N_0 \colon n \leq k R^{ \theta } \} $, 
	let 
		$ v_{ k } \colon \R^d \to \R $, $ k \in \{ 0, 1, \ldots, K \} $,  
	satisfy for all 
		$ k \in \{ 1, 2, \ldots, K \} $, 
		$ x \in \R^d $ 
	that 
		$ v_{ 0 } \in C( \R^d, \R ) $, 
		$ \EXP{ | v_{ 0 } ( x + W^{ 0 }_{ T } ) |^2 } < \infty $, 
	and 
		\begin{equation} 
		v_{ k } ( x ) = \Exp{ f ( v_{ k - 1 } ( x + W^{ 0 }_{ \nicefrac{ T }{ K } } ) ) }
		\end{equation} 
	(cf.~\cref{cor:well_definedness_exact_solution_brownian_case}), 
	and let 
		$ V^{ \theta }_{ k, n } \colon \R^d \times \Omega \to \R $, $ n \in \N_0 $, $ k \in \{ 0, 1, \ldots, K \} $, $ \theta \in \Theta $,
	satisfy for all 
		$ \theta \in \Theta $, 
		$ k \in \{ 0, 1, \ldots, K \} $, 
		$ n \in \N_{ 0 } $, 
		$ x \in \R^d $ 
	that 
		\begin{align}
		& V^{ \theta }_{ k, n } ( x ) 
		= \frac{ \mathbbm{ 1 }_{ \N }( n ) }{ M^n } 
		\sum_{ m = 1 }^{ M^n }
		\left[ v_{ 0 } ( x+ W^{ ( \theta, 0, -m ) }_{ \nicefrac{ kT }{ K } } ) + k f(0) \right]  
		+ \sum_{ j = 1 }^{ n - 1 } \frac{ k }{ M^{ n - j } }
		\\ \nonumber
		&
		\cdot \sum_{ m = 1 }^{ M^{ n - j } } 
		\bigg[  
		\Big( f \big( V^{ ( \theta, j, m ) }_{ \mc{R}^{ ( \theta, j, m ) }_{ k }, j } ( x + 
		W^{ ( \theta, j, m ) }_{ \nicefrac{ kT }{ K } } - 
		W^{ ( \theta, j, m ) }_{ \nicefrac{ \mc{R}^{ ( \theta, j, m ) }_{ k } T }{ K } } ) \big)
		-
		V^{ ( \theta, j, m ) }_{ \mc{R}^{ ( \theta, j, m ) }_{ k }, j } ( x + W^{ ( \theta, j, m ) }_{ \nicefrac{ k T }{ K } } - 
		W^{ ( \theta, j, m ) }_{ \nicefrac{ \mc{R}^{ ( \theta, j, m ) }_{ k } T }{ K } } ) \Big)
		\\ \nonumber
		& - \Big( f \big( V^{ ( \theta, j, -m ) }_{ \mc{R}^{ ( \theta, j, m ) }_{ k }, j - 1 } ( x + 
		W^{ ( \theta, j, m ) }_{ \nicefrac{ kT }{ K } } - 
		W^{ ( \theta, j, m ) }_{ \nicefrac{ \mc{R}^{ ( \theta, j, m ) }_{ k } T }{ K } } ) \big)
		-
		V^{ ( \theta, j, -m ) }_{ \mc{R}^{ ( \theta, j, m ) }_{ k }, j - 1 } ( x + W^{ ( \theta, j, m ) }_{ \nicefrac{ kT }{ K } } - 
		W^{ ( \theta, j, m ) }_{ \nicefrac{ \mc{R}^{ ( \theta, j, m ) }_{ k } T }{ K } } ) \Big)
		\bigg].
		\end{align}
	Then it holds for all 
		$ N \in \N_0 $, 
		$ x \in \R^d $ 
	that
		\begin{equation} 		
		\label{corollary_at_the_end_of_the_day:claim}
		\begin{split}
		\left( \Exp{ \big| V^{ 0 }_{ K, N } ( x ) - v_{ K } ( x ) \big|^2 } \right)^{ \!\nicefrac12 }
		& \leq
		\frac{ e^{ \nicefrac{ M }{ 2 } + KL } \left[ 1 + 2KL \right]^{ N } }{ M^{ \nicefrac{ N }{ 2 } } }
		\left[ \left( \Exp{ | v_{ 0 } ( x + W^{ 0 }_{ T } ) |^2 } \right)^{ \!\nicefrac12 } 
		+ K | f ( 0 ) |	\right]\!.
		\end{split}
		\end{equation} 
\end{cor} 
\begin{proof}[Proof of \cref{corollary_at_the_end_of_the_day}]
	Observe that \cref{prop:exponential_Euler_error_estimate} (applied with 
		$ d \is d $, 
		$ K \is K $, 
		$ M \is M $, 
		$ L \is \frac{ K L }{ T } $, 
		$ T \is T $,  
		$ \Theta \is \Theta $, 
		$ ( t_k )_{ k \in \{ 0, 1, \ldots, K \} } \is ( \frac{ k T }{ K } )_{ k \in \{ 0, 1, \ldots, K \} } $, 
		$ f \is ( [ 0, T ] \times \R^d \times \R \ni ( t, x, a ) \mapsto \frac{ K }{ T } ( f( a ) - a ) \in \R ) $, 
		$ (v_k)_{ k \in \{0, 1, \ldots, K \} } \is (v_k)_{ k \in \{0, 1, \ldots, K \} } $, 
		$ ( \Omega, \mc F, \P ) \is ( \Omega, \mc F, \P ) $, 
		$ ( W^{ \theta } )_{ \theta \in \Theta } \is ( W^{ \theta } )_{ \theta \in \Theta } $, 
		$ ( R^{ \theta } )_{ \theta \in \Theta } \is ( R^{ \theta } )_{ \theta \in \Theta } $ 
	in the notation of \cref{prop:exponential_Euler_error_estimate}) ensures that for all 
		$ N \in\N_0 $, 
		$ x \in \R^d $ 
	it holds that 
		\begin{equation} 
		\begin{split}
		&
		\left( \Exp{ | V^{ 0 }_{ K, N } ( x ) - v_{ K } ( x ) |^2 } \right)^{ \!\nicefrac12 }
		\leq 
		\frac{ e^{ \nicefrac{ M }{ 2 } } }{ M^{ \nicefrac{ N }{ 2 } } }
		\left[ 1 + 2 ( \tfrac{ K L }{ T } ) T \right]^{ N } 
		\\
		& \cdot e^{ \frac{ K L }{ T } T } \left[ \left( \Exp{ \left| v_{ 0 } ( x + W^{ 0 }_{ T } ) \right|^2 } \right)^{\!\nicefrac12}
		+ 
		\sum_{ l = 0 }^{ K - 1 } \tfrac{ T }{ K } 
		\left( \Exp{ \big| \tfrac{ K }{ T } ( f ( 0 ) - 0 ) \big|^2}\right)^{\!\nicefrac12}
		\right]
		\\
		& = 
		\frac{ e^{ \nicefrac{ M }{ 2 } } }{ M^{ \nicefrac{ N }{ 2 } } }
		e^{ K L } ( 1 + 2 K L )^{ N }
		\left[  \left( \Exp{ | v_{ 0 } ( x + W^{ 0 }_{ T } ) |^2 } \right)^{ \!\nicefrac12 } 
		+ K | f ( 0 ) | \right]\!.
		\end{split}
		\end{equation} 	
	This establishes \eqref{corollary_at_the_end_of_the_day:claim}. 
	This completes the proof of \cref{corollary_at_the_end_of_the_day}.
\end{proof}

\section[Complexity analysis for MLP approximations for nested expectations]{Complexity analysis for MLP approximations for iterated nested expectations}
\label{sec:complexity_analysis}

In this section we combine the error analysis for the proposed MLP approximation schemes in \cref{prop:error_estimate} in \cref{subsec:full_error_analysis} with a computational cost analysis for the proposed MLP approximation schemes to obtain in \cref{cor:complexity_general_dynamics} in \cref{subsec:complexity} below a complexity analysis for the proposed MLP approximation schemes. 
In \cref{cor:complexity_mlp_exponential_euler} in \cref{subsec:application} below we illustrate \cref{cor:complexity_general_dynamics} by means of a sample application to exponential Euler approximations.

\subsection{Complexity analysis for MLP approximations}
\label{subsec:complexity}

\begin{theorem} \label{cor:complexity_general_dynamics}
	Let $ \alpha,\gamma,c,\kappa,p \in [ 0, \infty ) $, 
		$ \Theta = \bigcup_{ n \in \N } \! \Z^n $, 
	let $ \xi_{ d } \in \R^d $, $ d \in \N $, 
	let $ L^{ d, K }_{ k } \in \R $, $ k \in \{0, 1, \ldots, K \} $, $ d, K \in \N $, 
	and $ f^{ d, K }_{ k } \in C( \R^d \times \R, \R) $, $ k \in \{0, 1, \ldots, K\} $, $ d, K \in \N $,   
	satisfy for all 
		$ d,K \in \N $, 
		$ k \in \{ 0, 1, \ldots, K \} $, 
		$ x \in \R^d $, 
		$ a,b \in \R $ 
	that  
		$ | ( f^{ d, K }_{ k } ( x, a ) - a ) - ( f^{ d, K }_{ k } ( x, b ) - b ) | \leq L^{ d, K }_{ k } | a - b | $, 	
	for every 
		$ d,K \in \N $, 
		$ k \in \{ 0, 1, \ldots, K \} $
	let $ \phi^{ d, K }_{ k } \colon \R^d \times \R^d \to \R^d$ be $ ( \Borel( \R^d \times \R^d ) ) $/$ \Borel( \R^d ) $-measurable, 
	let $ ( \Omega, \mc F, \P ) $ be a probability space, 
	for every 
		$ d,K \in \N $ 
	let $ W^{ d, K, \theta } = ( W^{ d, K, \theta }_{ k } )_{ k \in \{ 0, 1, \ldots, K \} } \colon \{ 0, 1, \ldots, K \} \times \Omega \to \R^d $, $ \theta \in \Theta $, be i.i.d.\,stochastic processes, 
	assume for all 
		$ d,K \in \N $, 
		$ \theta \in \Theta $ 
	that 
		$ W^{ d, K, \theta }_{ 0 }, W^{ d, K, \theta }_{ 1 }, 
		\ldots, W^{ d, K, \theta }_{ K - 1 } $ 
	are independent, 
	for every 
		$ d,K \in \N $ 
	let $ \mc{R}^{ d, K, \theta } = ( \mc{R}^{ d, K, \theta }_{ k } )_{ k \in \{ 0, 1, \ldots, K \} } \colon \{ 0, 1, \ldots, K \} \times \Omega \to \N_{ 0 } $, $ \theta \in \Theta $, be i.i.d.\,stochastic processes, 
	let $ \mf p^{ d, K }_{ k, l } \in ( 0, \infty ) $, $ l \in \N_0 $, $ k \in \{ 1, 2, \ldots, K \} $, $ d,K \in \N $, 
	assume for all 
		$ d,K \in \N $, 
		$ k \in \{ 1, 2, \ldots, K \} $, 
		$ l \in \{ 0, 1, \ldots, k - 1 \} $
	that 
		$ \mc{R}^{ d, K, \theta }_{ k } \leq k - 1 $, 
		$ \mf p^{ d, K }_{ k, l } \P( \mc{R}^{ d, K, 0 }_{ k } = l ) = | \P( \mc{R}^{ d, K, 0 }_{ k } = l ) |^2 $, 
	and  
		$ L^{ d, K }_{ l } \leq c \P( \mc{R}^{ d, K, 0 }_{ k } = l ) $, 
	assume for all
		$ d,K \in \N $ 
	that 
		$ ( R^{ d, K, \theta } )_{ \theta \in \Theta } $ 
	and 
		$ ( W^{ d, K, \theta } )_{ \theta \in \Theta } $
	are independent, 
	for every 
		$ d, K \in \N $, 
		$ k \in \{ 0, 1, \ldots, K \} $, 
		$ \theta \in \Theta $ 
	let $ X^{ d, K, \theta, k } = ( X^{ d, K, \theta, k, x }_{ l } )_{ ( l, x ) \in \{ 0, 1, \ldots, k \} \times \R^d } \colon 	\{ 0, 1, \ldots, k \} \times \R^d \times \Omega \to \R^d $ 
	satisfy for all 
		$ l \in \{ 0, 1, \ldots, k \} $, 
		$ x \in \R^d $ 
	that 
		\begin{equation}
		X^{ d, K, \theta, k, x }_{ l }
		= 
		\begin{cases}
		x & \colon l = k \\
		\phi^{ d, K }_{ l } ( X^{ d, K, \theta, k, x }_{ l + 1 }, W^{ d, K, \theta }_{ l } ) & \colon l < k,  
		\end{cases}
		\end{equation}	 
	for every 
		$ d,K \in \N $
	let $ v^{ d, K }_{ k } \colon \R^d \to \R $, $ k \in \{ 0, 1, \ldots, K \} $, be $ \Borel( \R^d ) $/$ \Borel( \R ) $-measurable, 
	assume for all 
		$ d, K \in \N $, 
		$ k \in \{ 1, 2, \ldots, K \} $, 
		$ x \in \R^d $ 
	that 
		$ \EXP{ | v^{ d, K }_{ 0 } ( X^{ d, K, 0, k, x }_{ 0 } ) |^2 + \sum_{ l = 0 }^{ k - 1 } | f^{ d, K }_{ l } ( X^{ d, K, 0 , k, x }_{ l }, 0 ) |^2 } < \infty $,   
		$ ( \EXP{ | v^{ d, K }_{ 0 } ( X^{ d, K, 0, K, \xi_d }_{ 0 } ) |^2 } )^{ \nicefrac12 } 
		+ \sum_{ l = 0 }^{ K - 1 } ( \EXP{ | f^{ d, K }_{ l } ( X^{ d, K, 0, K, \xi_d }_{ l }, 0 ) |^2 } )^{ \nicefrac12 } \leq  \kappa $,  		
	and 	
		\begin{equation} 
		\begin{split}
		v^{ d, K }_{ k } ( x ) 
		&
		= \EXPP{ f^{ d, K }_{ k - 1 } ( X^{ d, K, 0, k, x }_{ k - 1 }, v^{ d, K }_{ k - 1 } ( X^{ d, K, 0, k, x }_{ k - 1 } ) ) }
		\end{split}
		\end{equation} 
	(cf.~\cref{lem:welldefinedness_exact_solution}),    
		let $ V^{ d, K, \theta }_{ k, n, M } \colon \R^d \times \Omega \to \R $, $ \theta \in \Theta $, $ k \in \{ 0, 1, \ldots, K \} $, $ n \in \N_0 $, $ d,K,M \in \N $, 
		satisfy for all 
			$ d,K,M \in \N $, 
			$ n \in \N_0 $,
			$ \theta \in \Theta $, 
			$ k \in \{ 0, 1, \ldots, K \} $, 
			$ x \in \R^d $ 
		that 
			\begin{align}
			& 
			V^{ d, K, \theta }_{ k, n, M } ( x ) 
			\\& \nonumber 
			= 
			\frac{ \mathbbm{1}_{\N}(n) }{ M^n } \sum_{ m = 1 }^{ M^n }
			\left[ v^{ d, K }_{ 0 } ( X^{ d, K, ( \theta, 0, -m ), k, x }_{ 0 } ) + \sum_{ l = 0 }^{ k - 1 } f^{ d, K }_{ l } ( X^{ d, K, ( \theta, 0, m ), k, x }_{  l} ,0 ) \right] 
			+ \sum\limits_{ j = 1 }^{ n - 1 } 
		\frac{ 1 }{ M^{ n - j } }
		\sum\limits_{ m = 1 }^{ M^{ n - j } } 
		\frac{ \mathbbm{1}_{\N}(k) }{ \mf p^{ d, K }_{ k, \mc{R}^{ d, K, ( \theta, j, m ) }_{ k } } }
		\\ \nonumber
		& 
		\cdot \bigg[ 
		f^{ d, K }_{ \mc{R}^{ d, K, ( \theta, j, m ) }_{ k } } \big( X^{ d, K, ( \theta, j, m ), k, x }_{ \mc{R}^{ d, K, ( \theta, j, m ) }_{ k } },  V^{ d, K, ( \theta, j, m ) }_{ \mc{R}^{ d, K, ( \theta, j, m ) }_{ k }, j, M } ( X^{ d, K, ( \theta, j, m ) , k, x }_{ \mc{R}^{ d, K, ( \theta, j, m ) }_{ k } } ) \big) 
		- 
		V^{ d, K, ( \theta, j, m ) }_{ \mc{R}^{ d, K, ( \theta, j, m ) }_{ k }, j, M } ( X^{ d, K, ( \theta, j, m ), k, x }_{ \mc{R}^{ d, K, ( \theta, j, m ) }_{ k } } )
		\\ \nonumber
		&   
		- f^{ d, K }_{ \mc{R}^{ d, K, ( \theta, j, m ) }_{ k } } \big( X^{ d, K, ( \theta, j, m ), k, x }_{ \mc{R}^{ d, K, ( \theta, j, m ) }_{ k } },  V^{ d, K, ( \theta, j, -m ) }_{ \mc{R}^{ d, K, ( \theta, j, m ) }_{ k }, j - 1, M } ( X^{ d, K, ( \theta, j, m ), k, x }_{ \mc{R}^{ d, K, ( \theta, j, m ) }_{ k } } ) \big) 
		+
		V^{ d, K, ( \theta, j, -m ) }_{ \mc{R}^{ d, K, ( \theta, j, m ) }_{ k }, j - 1, M } ( X^{ d, K, ( \theta, j, m ), k, x }_{ \mc{R}^{ d, K, ( \theta, j, m ) }_{ k } } ) \bigg], 
		\end{align}
	and let 
		$ \mf C^{ d, K }_{ M, n } \in \R $, $ d, K, M, n \in \N_0 $, 
	satisfy for all 
		$ d,K,M \in \N $, 
		$ n \in \N_0 $ 
	that 
		\begin{equation} \label{complexity_general_dynamics:comp_cost}
		\mf C^{d,K}_{M,n} \leq \gamma K^{\alpha} d^p M^n + \sum\limits_{j=1}^{n-1} M^{n-j}(1+\gamma K^{\alpha} d^p + \mf C^{d,K}_{M,j} + \mf C^{d,K}_{M,j-1}).
		\end{equation}  
	Then 
		\begin{enumerate}[(i)]
		\item \label{complexity_general_dynamics:item1}
	it holds for all 
		$ d,K,M \in \N $,
		$ N \in \N_0 $ 
	that
		\begin{equation} \label{complexity_general_dynamics:claim1}
		\begin{split}
		\left( \EXPP{ | V^{ d, K, 0 }_{ K, N, M }( \xi_d ) - v^{ d, K }_{ K } ( \xi_d ) |^2 } \right)^{\!\nicefrac12}
		& \leq 
		\kappa ( 1 + 2 c )^N M^{ -\nicefrac{ N }{ 2 } } \exp( \nicefrac{ M }{ 2 } + c )
		\end{split}
		\end{equation}
	and 
		\item \label{complexity_general_dynamics:item2}
	there exist 
		$\mf N = ( \mf N_{\varepsilon} )_{ \varepsilon \in (0,1] } \colon ( 0, 1 ] \to \N $ 
	and 
		$\mf c = ( \mf c_{\delta} )_{ \delta \in (0,1] } \colon ( 0, 1]  \to [ 0, \infty ) $
	such that for all 
		$ d,K \in \N $, 
		$ \delta, \varepsilon \in ( 0, 1 ] $
	it holds that 
		$ \mf C^{ d, K }_{ \mf N_{ \varepsilon }, \mf N_{ \varepsilon } } 
		\leq \mf c_{ \delta } K^{ \alpha } d^p \varepsilon^{ - ( 2 + \delta ) } 
		$
	and 
		\begin{equation} \label{complexity_general_dynamics:claim2}
		\left( \EXPP{ | V^{ d, K, 0 }_{ K, \mf N_{ \varepsilon }, \mf N_{ \varepsilon } } ( \xi_d ) - v^{ d, K }_K ( \xi_d ) |^2 } \right)^{ \!\nicefrac12 } \leq \varepsilon . 
		\end{equation}
	\end{enumerate}
\end{theorem} 

\begin{proof}[Proof of \cref{cor:complexity_general_dynamics}] 
	First, observe that \cref{prop:error_estimate} (applied with 
		$ d 							\is d $, 
		$ K 							\is K $, 
		$ M 							\is M $, 
		$ \Theta 						\is \Theta $, 
		$ \mc O 						\is \R^d $, 
		$ ( f_{ k } )_{ k \in \{ 0, 1, \ldots, K \} } \is ( f^{ d, K }_{ k } )_{ k \in \{ 0, 1, \ldots, K \} } $, 
		$ g 							\is v^{ d, K }_{ 0 } $, 
		$ c \is c $,
		$ ( L_{ k } )_{ k \in \{ 0, 1, \ldots, K \} } \is ( L^{ d, K }_{ k } )_{ k \in \{ 0, 1, \ldots, K \} } $, 
		$ ( \phi_k )_{ k \in \{ 0, 1, \ldots, K \} } \is ( \phi^{ d, K }_{ k } )_{ k \in \{ 0, 1, \ldots, K \} } $, 
		$ ( \mc{R}^{ \theta } )_{ \theta \in \Theta } \is ( \mc{R}^{ d, K, \theta } )_{ \theta \in \Theta } $, 
		$ ( S,\mc S ) 					\is ( \R^d, \Borel( \R^d ) ) $,
		$ ( \Omega, \mc F, \P ) \is ( \Omega, \mc F, \P ) $, 	
		$ ( \mf p_{ k, l } )_{ ( k, l ) \in \{ 1, 2, \ldots, K \} \times \N_{ 0 } } \is ( \mf p^{ d, K }_{ k, l } )_{ ( k, l ) \in \{ 1, 2, \ldots, K \} \times \N_0 } $, 
		$ ( W^{ \theta } )_{ \theta \in \Theta }\is ( W^{ d, K, \theta } )_{ \theta \in \Theta  } $		
	for $ d,K,M \in \N $ in the notation of \cref{prop:error_estimate}) ensures that for all 
		$ d,K,M \in \N $, 
		$ N \in \N_0 $, 
		$ x \in \R^d $ 
	it holds that 
		\begin{equation} 
		\begin{split}
		& 
		\left( \Exp{ \big| V^{ d, K, 0 }_{ K, N, M } ( x ) - v^{ d, K }_{ K } ( x ) \big|^2 } \right)^{ \!\nicefrac12 } 
		\\
		&
		\leq 
		\frac{ e^{ \nicefrac{ M }{ 2 } } }{ M^{ \nicefrac{ N }{ 2 } } } \exp\!\left( \sum_{ j = 0 }^{ K - 1 } L^{ d, K }_{ j } \right) 
		\!\left[ 1 + 2 \max_{ k, j \in [ 0, K ] \cap \N_0, k > j  } \left( \frac{ L^{ d, K }_{ j } }{ \mf p^{ d, K }_{ k, j } } \right) \right]^N 
		\\
		& 
		\cdot 
		\left[ \left( \Exp{ \big| v^{ d, K }_{ 0 } ( X^{ d, K, 0, K, x }_{ 0 } ) \big|^2 } \right)^{\!\nicefrac12} 
		+ \sum_{ l = 0 }^{ K - 1 } \left( \Exp{ \big| f^{ d, K }_{ l } ( X^{ d, K, 0, K, x }_{ l }, 0 ) \big|^2 } \right)^{ \!\nicefrac12 } \right]\!. 
		\end{split}
		\end{equation} 
	The fact that for all 
		$ d,K \in \N $, 
		$ k \in \{ 1, 2, \ldots, K \} $, 
		$ l \in \{ 0, 1, \ldots, k - 1 \} $  
	it holds that 
		$ L^{ d, K }_{ l } \leq c \P( \mc{R}^{ d, K, 0 }_{ k } = l ) \leq c \mf p^{ d, K }_{ k, l } $ 
	hence implies that for all 
		$d,K,M\in\N$, 
		$N\in\N_0$, 
		$x\in\R^d$ 
	it holds that 
		\begin{equation} 
		\begin{split}
		& 
		\left( \Exp{ \big| V^{ d, K, 0 }_{ K, N, M } ( x ) - v^{ d, K }_{ K } ( x ) \big|^2} \right)^{ \!\nicefrac12 } 
		\\&
		\leq \frac{ e^{ \nicefrac{ M }{ 2 } } }{ M^{ \nicefrac{ N }{ 2 } } } e^c
		\!\left( 1 + 2c \right)^N 
		\left[ \left( \Exp{ \big| v^{ d, K }_{ 0 } ( X^{ d, K, 0, K, x }_{ 0 } ) \big|^2 } \right)^{ \!\nicefrac12 } 
		+ \sum_{ l = 0 }^{ K - 1 } \left( \Exp{ \big| f^{ d, K }_{ l } ( X^{ d, K, 0, K, x }_{ l }, 0 ) \big|^2 }	\right)^{\!\nicefrac12} \right]\!.
		\end{split}
		\end{equation} 
	Combining this with the assumption that for all 
		$ d,K \in \N $ 
	it holds that 
		$ ( \EXP{ | v^{ d, K }_{ 0 } ( X^{ d, K, 0, K, \xi_{ d } }_{ 0 } ) |^2 } )^{\nicefrac12} + \sum_{ l = 0 }^{ K - 1 } ( \EXP{ | f^{ d, K }_{ l } ( X^{ d, K, 0 , K, \xi_{ d } }_{ l }, 0 ) |^2 } )^{ \nicefrac12 } \leq \kappa $
	shows that for all 
		$d,K,M\in\N$, 
		$N\in\N_0$ 
	it holds that 
		\begin{equation} \label{complexity_general_dynamics:error_estimate}
		\begin{split}
		& 
		\left( \Exp{ \big| V^{ d, K, 0 }_{ K, N, M } ( \xi_{ d } ) - v^{ d, K }_{ K } ( \xi_{ d } ) \big|^2 } \right)^{\!\nicefrac12}  
		\leq 
		\frac{ e^{ \nicefrac{ M }{ 2 } } }{ M^{ \nicefrac{ N }{ 2 } } } e^c ( 1 + 2c )^N \kappa . 
		\end{split}
		\end{equation} 
	This establishes item~\eqref{complexity_general_dynamics:item1}. 
	Next we prove item~\eqref{complexity_general_dynamics:item2}. 
	Observe that \cite[Lemma 3.14]{MLPElliptic} (applied with 
		$ \alpha \is  \gamma K^{ \alpha }  d^p $, 
		$ \beta \is 1 + \gamma K^{ \alpha } d^p $, 
		$ M \is M $, 
		$ ( C_n )_{ n \in \N_0 } \is (\max \{ \mf C^{ d, K }_{ M, n }, 0 \} )_{ n \in \N_0 } $
	for $ d,K,M \in \N $ in the notation of \cite[Lemma 3.14]{MLPElliptic}) and \eqref{complexity_general_dynamics:comp_cost} show that for all 
		$ d,K,M,n \in \N $ 
	it holds that 
		\begin{equation} \label{complexity_general_dynamics:cost_estimate}
		\begin{split}
		\mf C^{ d, K }_{ M, n } 
		& 
		\leq 
		\frac{ 2 \gamma K^{ \alpha } d^p + ( 1 + \gamma K^{ \alpha } d^p ) }{ 2 } ( 3M )^{ n } 
		\leq 
		\left[ \frac{ 1 + 3 \gamma }{2} \right] d^p K^{ \alpha } (3M)^n . 
		\end{split}
		\end{equation} 
	In addition, note that \cite[Lemma 4.2]{AllenCahnApproximation2019} (applied with 
		$ \alpha \is 3 $, 
		$ \beta \is 2 $, 
		$ c \is \kappa + 1 $, 
		$ \kappa \is \sqrt{ e } $, 
		$ \rho \is 1 $, 
		$ K \is 0 $, 
		$ ( \gamma_{ n } )_{ n \in \N } \is ( ( 3 n )^n )_{ n \in \N } $, 
		$ ( \epsilon_{ n, r } )_{ n \in \N, r \in [ \rho, \infty) } \is ( \kappa e^{ c } [ \sqrt{ e } ( 1 + 2 c )]^n n^{ -\nicefrac{ n }{ 2 } } )_{ n \in \N, r \in [ 1, \infty) } $, 
		$ L \is ( ( 0, \infty ) \ni r \mapsto c \in [ 0, \infty ) ) $, 
		$ \varrho \is ( \N \ni n \mapsto n \in ( 0, \infty ) ) $
	in the notation of \cite[Lemma 4.2]{AllenCahnApproximation2019}) ensures that there exist
		$ \mf N = ( \mf N_{\varepsilon} )_{ \varepsilon \in (0, 1] } \colon ( 0, 1 ] \to \N $ 
	and 
		$ \mf c = ( \mf c_{\delta} )_{ \delta \in (0, 1] } \colon ( 0, 1 ] \to [ 0, \infty ) $
	which satisfy for all 
		$ \delta, \varepsilon \in ( 0, 1 ] $  
	that 
		\begin{equation}
		( 3 \mf N_{ \varepsilon } )^{ \mf N_{ \varepsilon } } \leq \mf c_{ \delta } \varepsilon^{ - ( 2 + \delta ) } 
		\qandq 
		\kappa e^{ c } [ \sqrt{ e } ( 1 + 2 c ) ]^{ \mf N_{ \varepsilon } } \mf N_{ \varepsilon }^{ -\nicefrac{ \mf N_{ \varepsilon } }{ 2 } } \leq \varepsilon.  
		\end{equation} 
	This, \eqref{complexity_general_dynamics:error_estimate}, and \eqref{complexity_general_dynamics:cost_estimate} imply that for all 
		$ \delta \in ( 0, \infty ) $, 
		$ \varepsilon \in ( 0, 1 ] $, 
		$ d,K \in \N $ 
	it holds that 
		\begin{equation} 
		\mf C^{ d, K }_{ \mf N_{ \varepsilon }, \mf N_{ \varepsilon } }
		\leq 
		\mf c_{ \delta } \left[ \frac{ 1 + 3 \gamma }{2} \right] d^p K^{ \alpha } \varepsilon^{ - ( 2 + \delta ) }
		\qandq
		\left( \EXPP{ | V^{ d, K, 0 }_{ K, \mf N_{ \varepsilon }, \mf N_{ \varepsilon } } ( \xi_{ d } ) - v^{ d, K }_{ K } ( \xi_{ d } ) |^2 } \right)^{\!\nicefrac12} \leq \varepsilon.   
		\end{equation} 
	This establishes item~\eqref{complexity_general_dynamics:item2}. 
	This completes the proof of \cref{cor:complexity_general_dynamics}. 
\end{proof} 

\subsection{Exponential Euler approximations for semilinear heat equations} 
\label{subsec:application}

\begin{lemma} \label{lem:and_one_more_well_definedness_result}
	Let $ d \in \N $, 
		$ p \in (0, 2) $,
		$ r,T \in [0, \infty) $, 
	let $ \lVert\cdot\rVert \colon \R^d \to [0, \infty) $ be the standard norm on $ \R^d $, 
	let $ (\Omega, \mc F, \P) $ be a probability space, 
	let $ W \colon [0, T]\times\Omega \to \R^d $ be a standard Brownian motion, 
	and  
	let $ h \in C( \R^d, \R ) $ satisfy that 
		$ \EXP{ | h( W_T ) |^2 } < \infty $. 
	Then $ \sup_{ x \in \R^d, \norm{x} \leq r } \EXP{ | h( x + W_T ) |^p } < \infty $.  
\end{lemma} 

\begin{proof}[Proof of \cref{lem:and_one_more_well_definedness_result}]
	Throughout this proof assume without loss of generality that $ \min\{ r, T \} > 0 $, 
	let 
		$ \alpha \in (0, 1) $ 
	satisfy that 
		$ ( 1 - \alpha )^2 > \frac{p}{2} $, 
	and let 
		$ R = \frac{r}{\alpha} $.  
	Note that the assumption that $ h \in C(\R^d, \R) $ and the fact that $ \{ z \in \R^d \colon \norm{z} \leq r \} $ is compact ensure that 
		\begin{equation} \label{and_one_more_well_definedness_result:sup_bounded_set}
		\sup_{ x \in \R^d, \norm{x} \leq r } \EXPP{ \mathbbm{1}_{ \{\norm{W_T} \leq 2R \} } | h( x + W_T) |^p } < \infty . 
		\end{equation} 
	Next note that the triangle inequality proves that for all 
		$ x, z \in \R^d $ 
	with 
		$ \norm{x} \leq r $ 
	and 
		$ \norm{z-x}\geq 2R $ 
	it holds  
		\begin{enumerate}[(i)]
		\item \label{and_one_more_well_definedness_result:proof_item1} that $ \norm{ z } \geq 2R - \norm{ x } \geq 2R - r = 2 R - \alpha R > R $, 
		\item that $ \norm{ x } \leq r = \alpha R \leq \alpha \norm{z} $, and 
		\item \label{and_one_more_well_definedness_result:proof_item3} that $ \norm{ z - x } \geq \norm{z } - \norm{ x } \geq ( 1 - \alpha ) \norm{ z } $. 
		\end{enumerate}
	Observe that item~\eqref{and_one_more_well_definedness_result:proof_item3} implies that for all 
		$ x \in \R^d $ 
	with 
		$ \norm{ x } \leq r $ 
	it holds that 
		\begin{equation} 
		\begin{split}
		& 
		\EXPP{\mathbbm{1}_{\{\norm{W_T} > 2R \} } | h( x + W_T ) |^p } 
		\\[1ex]
		& = 
		(2 \pi T)^{-\nicefrac{d}{2}} \int_{ \R^d } \mathbbm{1}_{ \{z \in \R^d \colon \norm{z} > 2 R \}}(y) \, | h( x + y ) |^p \exp(-\tfrac{\norm{y}^2}{2T})\,dy 
		\\
		& = 
		(2 \pi T)^{-\nicefrac{d}{2}} 
		\int_{ \R^d } \mathbbm{1}_{ \{z \in \R^d \colon \norm{z} > 2 R \}}(y - x) \,  | h( y ) |^p \exp(-\tfrac{\norm{y - x}^2}{2T})\,dy 
		\\
		& \leq 
		(2 \pi T)^{-\nicefrac{d}{2}} 
		\int_{ \R^d } | h( y ) |^p \exp( -(1-\alpha)^2 \tfrac{\norm{y}^2}{2T} ) \,dy. 
		\end{split}
		\end{equation} 
	The fact that $ ( 1 - \alpha )^2 > \frac{p}{2} $ and H\"older's inequality therefore show that for all 
		$ x \in \R^d $ 
	with 
		$ \norm{ x } \leq r  $
	it holds that 
		\begin{equation}
		\begin{split} 
		& 
		\EXPP{\mathbbm{1}_{\{\norm{W_T} > 2R \} } | h( x + W_T ) |^p } 
		\\
		& \leq 
		(2 \pi T)^{-\nicefrac{d}{2}} 
		\int_{ \R^d } \left[ | h( y ) |^2 \exp( -\tfrac{\norm{y}^2}{2T} ) \right]^{\nicefrac{p}{2}} 
		\exp( -( (1-\alpha)^2 - \tfrac{p}{2} ) \tfrac{\norm{y}^2}{2T} ) \,dy
		\\
		& \leq 
		( 2 \pi T)^{-\nicefrac{d}{2}}
		\left( \EXPP{ | h( W_T ) |^2 } \right)^{\!\nicefrac{p}{2}}
		\left[ \int_{\R^d} \exp( - \tfrac{ ( 1- \alpha)^2 - \nicefrac{p}{2} }{ 1 - \nicefrac{p}{2} } \tfrac{\norm{y}^2}{2T} ) \,dy \right]^{1-\nicefrac{p}{2}}. 
		\end{split} 
		\end{equation}
	Hence, we obtain that 
		$ \sup_{ x \in \R^d, \norm{x}\leq r } 
		\EXP{\mathbbm{1}_{\{\norm{W_T} \geq 2R \} } | h( x + W_T ) |^p } < \infty $. 
	Combining this with \eqref{and_one_more_well_definedness_result:sup_bounded_set} shows that 
		$ \sup_{ x \in \R^d, \norm{x}\leq r } \EXP{ | h( x + W_T ) |^p } < \infty $. 
	This completes the proof of \cref{lem:and_one_more_well_definedness_result}. 
\end{proof} 

\begin{lemma} \label{lem:yet_another_well_definedness_result}
	Let $ d \in \N $,
		$ f \in C( \R^d \times \R, \R ) $, 
		$ g \in C( \R^d, \R ) $, 
		$ L, T \in [0, \infty) $,  
	let $ ( \Omega, \mc F, \P ) $ be a probability space, 
	let $ W \colon [0, T]\times\Omega \to \R^d $ be a standard Brownian motion, 
	assume for all 
		$ x \in \R^d $, 
		$ a, b \in \R $ 
	that 
		$ | f( x, a ) - f( x, b ) | \leq L | a - b | $
	and  	
		$ \EXP{ | f( x + W_T, 0 ) |^2 + | g( x + W_T ) |^2 } < \infty $, 
	and	let $ u \colon \R^d \to \R $ satisfy for all 
		$ x \in \R^d $ 
	that 
		$ u( x ) = \EXP{ f(x + W_T, g( x + W_T ) ) } $
	(cf.\ \cref{lem:welldefinedness_induction_step}). 
	Then $ u \in C(\R^d, \R) $. 
\end{lemma}

\begin{proof}[Proof of \cref{lem:yet_another_well_definedness_result}] 
	Throughout this proof let $ \lVert\cdot\rVert \colon \R^d \to [0, \infty) $ be the standard norm on $ \R^d $	
	and let 
		$ \mf x_n \in \R^d $, $ n \in \N_0 $, 
	satisfy 
		$ \limsup_{ n \to \infty } \norm{ \mf x_n - \mf x_0 } = 0 $. 
	The assumption that $ f \in C( \R^d \times \R, \R ) $ and the assumption that $ g \in C( \R^d, \R ) $ ensure that for all 
		$ \omega \in \Omega $ 
	it holds that 
		\begin{equation} \label{yet_another_well_definedness_result:pointwise_convergence}
		\limsup_{ n \to \infty } | f( \mf x_n + W_T(\omega), g( \mf x_n + W_T(\omega) ) ) - f( \mf x_0 + W_T(\omega), g( \mf x_0 + W_T(\omega) ) ) | = 0 .
		\end{equation}  
	Moreover, note that the assumption that for all 
		$ x \in \R^d $, 
		$ a,b \in \R $ 
	it holds that 
		$ | f( x, a ) - f( x, b ) | \leq L | a - b | $ 
	and the assumption that for all 
		$ x \in \R^d $ 
	it holds that 
		$ \EXP{ | f( x + W_T, 0 ) |^2 + | g( x + W_T ) |^2 } < \infty $
	imply that 
		\begin{equation} 
		\EXPP{ | f( x + W_T, g( x + W_T ) ) |^2 } 
		\leq 
		2 \, \EXPP{ | f( x + W_T, 0 ) |^2 } + 2 \, \EXPP{ | g( x + W_T ) |^2 } < \infty. 
		\end{equation} 
	\cref{lem:and_one_more_well_definedness_result} hence ensures that for all 
		$ p \in (1, 2) $ 
	it holds that 	
		$ \sup_{ n \in \N } \EXP{ | f( \mf x_n + W_T, g( \mf x_n + W_T ) ) |^p } < \infty $. 
	Vitali's convergence theorem (see, e.g., Klenke~\cite[Theorem 6.25]{Klenke2014}), the de la Vall\'ee Poussin theorem (see, e.g., Klenke~\cite[Theorem 6.19]{Klenke2014}), and \eqref{yet_another_well_definedness_result:pointwise_convergence} therefore demonstrate that
		\begin{equation} 
		\limsup_{ n \to \infty } \EXPP{ | f( \mf x_n + W_T, g( \mf x_n + W_T ) ) - f( \mf x_0 + W_T, g( \mf x_0 + W_T ) ) | } = 0.  
		\end{equation}  
	This completes the proof of \cref{lem:yet_another_well_definedness_result}. 		
\end{proof} 

\begin{cor} \label{cor:complexity_mlp_exponential_euler}
	Let	$ L,T \in \R $, 
		$ \Theta = \bigcup_{ k \in \N } \! \Z^{ k } $, 
		$ f \in C( \R, \R ) $ 
	satisfy for all 
		$ a,b \in \R $ 
	that 
		$| f ( a ) - f ( b ) | \leq L | a - b |$, 
	let 
		$ \xi_{ d } \in \R^d $, $ d \in \N $, 
	for every 
		$ K \in \N $ 
	let $ t^{ K }_{ 0 }, t^{ K }_{ 1 }, \ldots, t^{ K }_{ K } \in \R $
	satisfy 
		$ 0 = t^{ K }_{ 0 } < t^{ K }_{ 1 } < \ldots < t^{ K }_{ K } = T $,
	let $( \Omega, \mc F, \P ) $ be a probability space, 
	let $ R^{ \theta } \colon \Omega \to ( 0, 1 ) $, $ \theta \in \Theta $, be i.i.d.\,random variables, 
	assume for all 
		$ r \in (0, 1) $ 
	that 
		$ \P ( R^{ 0 } \leq r ) = r $, 
	let $ W^{ d, \theta } \colon [0,T] \times \Omega \to \R^d $, $ d \in \N $, $ \theta \in \Theta $, be i.i.d.~standard Brownian motions, 
	assume that 
		$ ( R^{ \theta } )_{ \theta \in \Theta } $ 
	and 
		$ ( W^{ d, \theta } )_{ ( d, \theta ) \in \N \times \Theta } $ 
	are independent, 
	let $ v^{ d, K }_{ k } \in C( \R^d, \R ) $, $ k \in \{ 0, 1, \ldots, K \} $, $ d, K \in \N $,
	satisfy for all
		$ d,K \in \N $,
		$ k \in \{ 1, 2, \ldots, K \} $,
		$ x \in \R^d $	
	that 
		$ \EXP{ | v^{d,K}_{ 0 } ( \xi_{ d } + W^{ d, 0 }_{ T } ) |^2 } \leq L $, 
		$ \EXP{ | v^{d,K}_{ 0 } ( x + W^{ d, 0 }_{ T } ) |^2 } < \infty $, 
	and 
		\begin{equation} 
		v^{ d, K }_{ k } ( x ) 
		= \Exp{ v^{ d, K }_{ k - 1 } ( x + W^{ d, 0 }_{ t^{ K }_{ k } - t^{ K }_{ k - 1 } } ) + ( t^{ K }_{ k } - t^{ K }_{ k - 1 } ) f \big( v^{ d, K }_{ k - 1 } ( x + W^{ d, 0 }_{ t^{ K }_{ k } - t^{ K }_{ k - 1 } } ) \big) }
		\end{equation}
	(cf.\,\cref{lem:yet_another_well_definedness_result}), 	
	let $ \mc{R}^{ K, \theta } = ( \mc{R}^{ K, \theta }_{ k } )_{ k \in \{ 0, 1, \ldots, K \} } \colon \{ 0, 1, \ldots, K \} \times \Omega \to \N_0 $, $ K \in \N $, $ \theta \in \Theta $, satisfy for all 
	$ K \in \N $, 
	$ \theta \in \Theta $,
	$ k \in \{ 0, 1, \ldots, K \} $ 
	that 
	$ \mc{R}^{ K, \theta }_{ k } = \max \{ n \in \N_0 \colon t^{ K }_{ n } \leq t^{ K }_{ k } R^{ \theta } \} $, 		 
	let $ V^{ d, K, \theta }_{ k, n, M } \colon \R^d \times \Omega \to \R $, $ \theta \in \Theta $, $ k \in \{ 0, 1, \ldots, K \} $, $ n \in \N_0 $, $ d,K,M \in \N $,
	satisfy for all
		$ \theta \in \Theta $, 
		$ d, K, M \in \N $, 
		$ n \in \N_0 $,
		$ k \in \{ 0, 1, \ldots, K \} $, 
		$ x \in \R^d $ 
	that 
		\begin{multline}
		V^{ d, K, \theta }_{ k, n, M } ( x ) 
		= 
		\frac{ \mathbbm{1}_{\N}(n) }{ M^{ n } } \sum_{ m = 1 }^{ M^{ n } } 
		\left[ v^{ d, K }_{ 0 } ( x + W^{ d, ( \theta, 0, -m ) }_{ t^{ K }_{ k } } ) + t^{ K }_{ k } f ( 0 ) \right]  
		\\
		+ 
		\sum_{ j = 1 }^{ n - 1 } 
		\frac{ t^{ K }_{ k } }{ M^{ n - j } }
		\sum_{ m = 1 }^{ M^{ n - j } } \bigg[  
		f \big( V^{ d, K, ( \theta, j, m ) }_{ \mc{R}^{ K, ( \theta, j, m ) }_{ k }, j, M } ( x + W^{ d, ( \theta, j, m ) }_{ t^{ K }_{ k } } - W^{ d, ( \theta, j, m ) }_{ t^{ K }_{ \mc{R}^{ K, ( \theta, j, m ) }_{ k } } } ) \big) 
		\\
		- f \big( V^{ d, K, ( \theta, j, -m ) }_{ \mc{R}^{ K, ( \theta, j, m ) }_{ k }, j - 1, M } ( x + W^{ d, ( \theta, j, m ) }_{ t^{ K }_{ k } } - W^{ d, ( \theta, j, m ) }_{ t^{ K }_{ \mc{R}^{ K, ( \theta, j, m ) }_{ k } } } ) \big) \bigg], 
		\end{multline}
	and let $ \mf C^{ d, K }_{ M, n } \in \R $, $ d, K, M, n \in \N_0 $, satisfy for all 
		$ d,K,M \in \N $, 
		$ n \in \N_0 $ 
	that
		\begin{equation}\label{cor:complexity_mlp_exponential_euler:cost_inequality}
		\mf C^{ d, K }_{ M, n } \leq d M^{ n } + \sum_{ j = 1 }^{ n - 1 } M^{ n - j } ( d + 1 + \mf C^{d,K}_{M,j} + 
		\mf C^{ d, K }_{ M, j - 1 } ). 
		\end{equation}
	Then there exist 
		$ \mf N = ( \mf N_{ \varepsilon } )_{ \varepsilon \in (0,1] } \colon ( 0, 1 ] \to \N $
	and 
		$ \mf c = ( \mf c_{ \delta } )_{ \delta \in (0,1] } \colon ( 0, 1 ] \to \R $ 
	such that for all 
		$ \delta, \varepsilon \in (0,1] $,
		$ d,K \in \N $
	it holds that 
		$ \mf C^{ d, K }_{ \mf N_{ \varepsilon }, \mf N_{ \varepsilon } } \leq \mf c_{ \delta } d \varepsilon^{ - ( 2 + \delta ) } $ 
	and 
		$
		\big( \EXP{ | V^{ d, K, 0 }_{ K, \mf N_{ \varepsilon }, \mf N_{ \varepsilon } } ( \xi_{ d } ) - v^{ d, K }_{ K } ( \xi_{ d } ) |^2 } \big)^{\! \nicefrac12 } \leq \varepsilon 
		$. 	
\end{cor} 

\begin{proof}[Proof of \cref{cor:complexity_mlp_exponential_euler}]
	Throughout this proof let 
		$ \mf p^{ K }_{ k, l } \in ( 0, \infty ) $, $ k \in \{ 1, 2, \ldots, K \} $, $ l \in \N_{ 0 } $, $ K \in \N $, 
	satisfy for all 
		$ K \in \N $, 
		$ k \in \{ 1, 2, \ldots, K \} $, 
		$ l \in \N_{ 0 } $ 
	that 
		\begin{equation} 
		\mf p^{ K }_{ k, l } = 
		\begin{cases} 
		\frac{ t^{ K }_{ l + 1 } - t^{ K }_{ l } }{ t^{ K }_{ k } } & \colon l < k \\[1ex]
		1 & \colon l \geq k.
		\end{cases} 
		\end{equation}  
	Observe that item~\eqref{complexity_general_dynamics:item2} of \cref{cor:complexity_general_dynamics} (applied with 
		$ \alpha \is 0 $, 
		$ \gamma \is 1 $, 
		$ c \is LT $, 
		$ \kappa \is L + T | f ( 0 ) | $, 
		$ p \is 1 $, 
		$ \Theta \is \Theta $, 
		$ \xi_{ d } \is \xi_{ d } $,
		$ ( L^{ d, K }_{ k } )_{ k \in \{ 0, 1, \ldots, K \} } \is ( L ( t^{ K }_{ \min\{ k + 1, K \} } - t^{ K }_{ k } ) )_{ k \in \{ 0, 1, \ldots, K \} } $, 
		$ ( f^{ d, K }_{ k } )_{ k \in \{ 0, 1, \ldots, K \} } \is ( ( \R^d \times \R \ni ( x, a ) \mapsto a + ( t^{ K }_{ \min\{ k + 1, K \} } - t^{ K }_{ k } ) f ( a ) \in \R ) )_{ k \in \{ 0, 1, \ldots, K \} } $, 
		$ ( \phi^{ d, K }_{ k } )_{ k \in \{ 0, 1, \ldots, K \} } \is ( \R^d \times \R^d \ni ( x, w ) \mapsto x + w \in \R^d )_{ k \in \{ 0, 1, \ldots, K \} }$, 
		$ ( \Omega, \mc F, \P ) \is ( \Omega,\mc F, \P ) $, 
		$ ( \mf C^{ d, K }_{ M, n } )_{ ( M, n ) \in \N \times \N_0 } \is ( \mf C^{d,K}_{M,n} )_{(M,n)\in\N\times\N_0}$,
		$ ( \mc{R}^{ d, K, \theta } )_{ \theta \in \Theta } \is ( \mc R^{ K, \theta } )_{ \theta \in \Theta } $,    
		$ ( W^{ d, K, \theta }_{ k } )_{ ( k, \theta ) \in \{ 0, 1, \ldots, K \} \times \Theta } \is ( W^{ d, \theta }_{ t^{ K }_{ \min\{ k + 1, K \} } } - W^{ d, \theta }_{ t^{ K }_{ k } } )_{ ( k, \theta ) \in \{ 0, 1, \ldots, K \} \times \Theta } $, 
		$ ( \mf p^{ d, K }_{ k, l } )_{ ( k, l ) \in \{ 1, 2, \ldots, K \} \times \N_{ 0 } } \is ( \mf p^{ K }_{ k, l } )_{ ( k, l ) \in \{ 1, 2, \ldots, K \} \times \N_{ 0 } } $, 
		$ ( v^{ d, K }_{ k } )_{ k \in \{0, 1, \ldots, K\}} \is ( v^{ d, K }_{ k } )_{ k \in \{ 0, 1, \ldots, K \} } $
	for 
		$d,K\in\N$ 
	in the notation of \cref{cor:complexity_general_dynamics}) ensures that there exist 	
		$ \mf N = ( \mf N_{\varepsilon } )_{\varepsilon \in (0,1]} \colon (0,1] \to \N$ 
	and 
		$ \mf c = ( \mf c_{\delta} ) \colon (0, 1] \to [0,\infty)$ 
	such that for all 
		$ \delta,\varepsilon \in (0,1] $, 
		$ d,K \in \N $ 
	it holds that 
		$\mf C^{d,K}_{\mf N_{\varepsilon},\mf N_{\varepsilon}} 
		\leq \mf c_{\delta} d \varepsilon^{- ( 2 + \delta ) } $ 
	and 
		\begin{equation} 
		\left( \Exp{ \big| V^{ d, K, 0 }_{ K, \mf N_{ \varepsilon }, \mf N_{ \varepsilon } } ( \xi_{ d } ) - v^{ d, K }_{ K } ( \xi_{ d } ) \big|^2 } \right)^{ \!\nicefrac12 } 	\leq \varepsilon. 
		\end{equation} 
	This completes the proof of \cref{cor:complexity_mlp_exponential_euler}. 
\end{proof}

\subsection*{Acknowledgements}

The second author acknowledges funding by the Deutsche Forschungsgemeinschaft (DFG, German Research Foundation) under Germany's Excellence Strategy EXC 2044-390685587, Mathematics Muenster: Dynamics-Geometry-Structure. 

\bibliographystyle{acm}
\bibliography{multilevel_bibfile}

\begin{thebibliography}{10}

\bibitem{AdesLuClaxton2004ExpectedValueOfSampleInfo}
{\sc Ades, A., Lu, G., and Claxton, K.}
\newblock Expected value of sample information calculations in medical decision
  modeling.
\newblock {\em Medical decision making 24}, 2 (2004), 207--227.

\bibitem{Andersen1999SimpleApproachToPricingBermudanSwaptions}
{\sc Andersen, L.~B.}
\newblock A simple approach to the pricing of bermudan swaptions in the
  multi-factor libor market model.
\newblock {\em Available at SSRN 155208\/} (1999).

\bibitem{MLPElliptic}
{\sc Beck, C., Gonon, L., and Jentzen, A.}
\newblock Overcoming the curse of dimensionality in the numerical approximation
  of high-dimensional semilinear elliptic partial differential equations.
\newblock {\em arXiv:2003.00596\/} (2020), 50 pages.

\bibitem{AllenCahnApproximation2019}
{\sc Beck, C., Hornung, F., Hutzenthaler, M., Jentzen, A., and Kruse, T.}
\newblock Overcoming the curse of dimensionality in the numerical approximation
  of {A}llen--{C}ahn partial differential equations via truncated full-history
  recursive multilevel {P}icard approximations.
\newblock {\em Accepted by the {J}ournal of {N}umerical {M}athematics,
  arXiv:1907.06729\/} (2019), 31 pages.

\bibitem{BeckerEtAl2020MLPSimulations}
{\sc Becker, S., Braunwarth, R., Hutzenthaler, M., Jentzen, A., and von
  Wurstemberger, P.}
\newblock Numerical simulations for full history recursive multilevel picard
  approximations for systems of high-dimensional partial differential
  equations.
\newblock {\em arXiv preprint arXiv:2005.10206\/} (2020), 21 pages.

\bibitem{Belomestny2011PricingBermudanOptions}
{\sc Belomestny, D.}
\newblock Pricing {B}ermudan options by nonparametric regression: optimal rates
  of convergence for lower estimates.
\newblock {\em Finance Stoch. 15}, 4 (2011), 655--683.

\bibitem{BenderGaertnerSchweizer2018PathwiseDynamicProgramming}
{\sc Bender, C., G\"{a}rtner, C., and Schweizer, N.}
\newblock Pathwise dynamic programming.
\newblock {\em Math. Oper. Res. 43}, 3 (2018), 965--995.

\bibitem{BenderSchweizerZhuo2017PrimalDualAlgorithmForBSDEs}
{\sc Bender, C., Schweizer, N., and Zhuo, J.}
\newblock A primal-dual algorithm for {BSDE}s.
\newblock {\em Math. Finance 27}, 3 (2017), 866--901.

\bibitem{BouchardTouzi2004ApproximationOfBSDEs}
{\sc Bouchard, B., and Touzi, N.}
\newblock Discrete-time approximation and {M}onte-{C}arlo simulation of
  backward stochastic differential equations.
\newblock {\em Stochastic Process. Appl. 111}, 2 (2004), 175--206.

\bibitem{BratvoldBickelLohne2009ValueOfInformationOilGas}
{\sc Bratvold, R.~B., Bickel, J.~E., Lohne, H.~P., et~al.}
\newblock Value of information in the oil and gas industry: past, present, and
  future.
\newblock {\em SPE Reservoir Evaluation \& Engineering 12}, 04 (2009),
  630--638.

\bibitem{BrennanEtAl2007CalculatingPartialEVPI}
{\sc Brennan, A., Kharroubi, S., O'Hagan, A., and Chilcott, J.}
\newblock Calculating partial expected value of perfect information via {M}onte
  {C}arlo sampling algorithms.
\newblock {\em Medical Decision Making 27}, 4 (2007), 448--470.

\bibitem{BroadieDuMoallemi2011EfficientRiskEstimation}
{\sc Broadie, M., Du, Y., and Moallemi, C.~C.}
\newblock Efficient risk estimation via nested sequential simulation.
\newblock {\em Management Science 57}, 6 (2011), 1172--1194.

\bibitem{BroadieDuMoallemi2015RiskRegression}
{\sc Broadie, M., Du, Y., and Moallemi, C.~C.}
\newblock Risk estimation via regression.
\newblock {\em Oper. Res. 63}, 5 (2015), 1077--1097.

\bibitem{BroadieGlasserman1997PricingAmericans}
{\sc Broadie, M., and Glasserman, P.}
\newblock Pricing {A}merican-style securities using simulation.
\newblock {\em J. Econom. Dynam. Control 21}, 8-9 (1997), 1323--1352.
\newblock Computational financial modelling.

\bibitem{BujokHamblyReisinger2015MultilevelSimulation}
{\sc Bujok, K., Hambly, B.~M., and Reisinger, C.}
\newblock Multilevel simulation of functionals of {B}ernoulli random variables
  with application to basket credit derivatives.
\newblock {\em Methodol. Comput. Appl. Probab. 17}, 3 (2015), 579--604.

\bibitem{Carriere1996ValuationEarlyExercise}
{\sc Carriere, J.~F.}
\newblock Valuation of the early-exercise price for options using simulations
  and nonparametric regression.
\newblock {\em Insurance Math. Econom. 19}, 1 (1996), 19--30.

\bibitem{DimitsEtAl2013CoulombCollisions}
{\sc Dimits, A.~M., Cohen, B.~I., Caflisch, R.~E., Rosin, M.~S., and Ricketson,
  L.~F.}
\newblock Higher-order time integration of {C}oulomb collisions in a plasma
  using {L}angevin equations.
\newblock {\em J. Comput. Phys. 242\/} (2013), 561--580.

\bibitem{hutzenthaler2016multilevel}
{\sc E, W., Hutzenthaler, M., Jentzen, A., and Kruse, T.}
\newblock Multilevel {P}icard iterations for solving smooth semilinear
  parabolic heat equations.
\newblock {\em arXiv:1607.03295\/} (2016), 19 pages.

\bibitem{EHutzenthalerJentzenKruse2019MLP}
{\sc E, W., Hutzenthaler, M., Jentzen, A., and Kruse, T.}
\newblock On multilevel {P}icard numerical approximations for high-dimensional
  nonlinear parabolic partial differential equations and high-dimensional
  nonlinear backward stochastic differential equations.
\newblock {\em J. Sci. Comput. 79}, 3 (2019), 1534--1571.

\bibitem{Egloff2005_MonteCarloForOptimalStoppingAndStatisticalLearning}
{\sc Egloff, D.}
\newblock Monte {C}arlo algorithms for optimal stopping and statistical
  learning.
\newblock {\em Ann. Appl. Probab. 15}, 2 (2005), 1396--1432.

\bibitem{ElKarouiKapoudjianPardouxPeng1997ReflectedBSDEsAndObstacleProblems}
{\sc El~Karoui, N., Kapoudjian, C., Pardoux, E., Peng, S., and Quenez, M.~C.}
\newblock Reflected solutions of backward {SDE}'s, and related obstacle
  problems for {PDE}'s.
\newblock {\em Ann. Probab. 25}, 2 (1997), 702--737.

\bibitem{Giles2018MLMCNestedExpectations}
{\sc Giles, M.~B.}
\newblock {MLMC} for {N}ested {E}xpectations.
\newblock In {\em Contemporary Computational Mathematics-A Celebration of the
  80th Birthday of Ian Sloan}. Springer, 2018, pp.~425--442.

\bibitem{GilesGoda2018DecisionMakingUnderUncertainty}
{\sc Giles, M.~B., and Goda, T.}
\newblock Decision-making under uncertainty: using {MLMC} for efficient
  estimation of {EVPPI}.
\newblock {\em Statistics and Computing\/} (2018), 1--13.

\bibitem{GilesHaji-Ali2018MultilevelNestedSimulationForRiskEstimation}
{\sc Giles, M.~B., and Haji-Ali, A.-L.}
\newblock Multilevel nested simulation for efficient risk estimation.
\newblock {\em arXiv:1802.05016\/} (2018), 28 pages.

\bibitem{GilesWeltiJentzen2019GeneralisedMLP}
{\sc Giles, M.~B., Jentzen, A., and Welti, T.}
\newblock Generalised multilevel {P}icard approximations.
\newblock {\em arXiv:1911.03188\/} (2019), 61 pages.

\bibitem{GlassermanHeidelbergerShahabudding2000VarianceReductionForVaR}
{\sc Glasserman, P., Heidelberger, P., and Shahabuddin, P.}
\newblock Variance reduction techniques for estimating {V}alue-at-{R}isk.
\newblock {\em Management Science 46}, 10 (2000), 1349--1364.

\bibitem{GlassermanHeidelbergerShahabuddin2002PortfolioVaRHeavyTailedRisk}
{\sc Glasserman, P., Heidelberger, P., and Shahabuddin, P.}
\newblock Portfolio {V}alue-at-{R}isk with heavy-tailed risk factors.
\newblock {\em Mathematical Finance 12}, 3 (2002), 239--269.

\bibitem{GobetLemorWarin2005_ARegressionBasedMonteCarloMethodForBSDEs}
{\sc Gobet, E., Lemor, J.-P., and Warin, X.}
\newblock A regression-based {M}onte {C}arlo method to solve backward
  stochastic differential equations.
\newblock {\em Ann. Appl. Probab. 15}, 3 (2005), 2172--2202.

\bibitem{GodaHironakaIwamoto2020MultilevelMonteCarloForExpectedInformationGain}
{\sc Goda, T., Hironaka, T., and Iwamoto, T.}
\newblock Multilevel {M}onte {C}arlo estimation of expected information gains.
\newblock {\em Stochastic Analysis and Applications 38}, 4 (2020), 581--600.

\bibitem{GodaMurakamiTanakaSato2018DecisionTheoreticSensitivityAnalysis}
{\sc Goda, T., Murakami, D., Tanaka, K., and Sato, K.}
\newblock Decision-theoretic sensitivity analysis for reservoir development
  under uncertainty using multilevel quasi-{M}onte {C}arlo methods.
\newblock {\em Computational Geosciences 22}, 4 (2018), 1009--1020.

\bibitem{GoldbergChen2018BeatingTheCurseOfDimensionalityInOptionsPricingAndOptimalStopping}
{\sc Goldberg, D.~A., and Chen, Y.}
\newblock Beating the curse of dimensionality in options pricing and optimal
  stopping.
\newblock {\em arXiv:1807.02227\/} (2018), 62 pages.

\bibitem{GordyJuneja2010NestedSimulationPortfolioRisk}
{\sc Gordy, M.~B., and Juneja, S.}
\newblock Nested simulation in portfolio risk measurement.
\newblock {\em Management Science 56}, 10 (2010), 1833--1848.

\bibitem{HajiAliEtAl2018MCMcKeanVlasov}
{\sc Haji-Ali, A.-L., and Tempone, R.}
\newblock Multilevel and multi-index {M}onte {C}arlo methods for the
  {M}c{K}ean-{V}lasov equation.
\newblock {\em Stat. Comput. 28}, 4 (2018), 923--935.

\bibitem{hironaka2019multilevel}
{\sc Hironaka, T., Giles, M.~B., Goda, T., and Thom, H.}
\newblock Multilevel {M}onte {C}arlo estimation of the expected value of sample
  information.
\newblock {\em arXiv:1909.00549\/} (2019), 26 pages.

\bibitem{HutzenthalerJentzenKruse2019MLPGradient}
{\sc Hutzenthaler, M., Jentzen, A., and Kruse, T.}
\newblock Overcoming the curse of dimensionality in the numerical approximation
  of parabolic partial differential equations with gradient-dependent
  nonlinearities.
\newblock {\em arXiv:1912.02571\/} (2019), 33 pages.

\bibitem{Overcoming}
{\sc Hutzenthaler, M., Jentzen, A., Kruse, T., Nguyen, T.~A., and von
  Wurstemberger, P.}
\newblock Overcoming the curse of dimensionality in the numerical approximation
  of semilinear parabolic partial differential equations.
\newblock {\em Accepted by Proc. Roy. Soc. London A, arXiv:1807.01212\/}
  (2018), 30 pages.

\bibitem{hutzenthaler2019overcoming}
{\sc Hutzenthaler, M., Jentzen, A., and von Wurstemberger, P.}
\newblock Overcoming the curse of dimensionality in the approximative pricing
  of financial derivatives with default risks.
\newblock {\em Electronic Journal of Probability 25\/} (2020).

\bibitem{HutzenthalerKruse2020MLPApproximationsOfHighDimensionalSemilinearPDEsWithGradientDependentNonlinearities}
{\sc Hutzenthaler, M., and Kruse, T.}
\newblock Multilevel {P}icard {A}pproximations of {H}igh-{D}imensional
  {S}emilinear {P}arabolic {D}ifferential {E}quations with
  {G}radient-{D}ependent {N}onlinearities.
\newblock {\em SIAM J. Numer. Anal. 58}, 2 (2020), 929--961.

\bibitem{Klenke2014}
{\sc Klenke, A.}
\newblock {\em Probability theory}, second~ed.
\newblock Universitext. Springer, London, 2014.
\newblock A comprehensive course.

\bibitem{KornKornKroisandt2010MonteCarloMethodsAndModels}
{\sc Korn, R., Korn, E., and Kroisandt, G.}
\newblock {\em Monte {C}arlo methods and models in finance and insurance}.
\newblock Chapman \& Hall/CRC Financial Mathematics Series. CRC Press, Boca
  Raton, FL, 2010.

\bibitem{LongstaffSchwartz2001ValuingAmericanOptions}
{\sc Longstaff, F.~A., and Schwartz, E.~S.}
\newblock Valuing {A}merican options by simulation: a simple least-squares
  approach.
\newblock {\em The review of financial studies 14}, 1 (2001), 113--147.

\bibitem{NakayasuGodaTanaka2016EvaluatingValueSinglePointData}
{\sc Nakayasu, M., Goda, T., Tanaka, K., Sato, K., et~al.}
\newblock Evaluating the value of single-point data in heterogeneous reservoirs
  with the expectation-maximization algorithm.
\newblock {\em SPE Economics \& Management 8}, 01 (2016), 1--10.

\bibitem{PardouxPeng1990AdaptionSolutionOfBSDEs}
{\sc Pardoux, {\'E}., and Peng, S.}
\newblock Adapted solution of a backward stochastic differential equation.
\newblock {\em Systems Control Lett. 14}, 1 (1990), 55--61.

\bibitem{PardouxPeng1992}
{\sc Pardoux, E., and Peng, S.}
\newblock Backward stochastic differential equations and quasilinear parabolic
  partial differential equations.
\newblock In {\em Stochastic partial differential equations and their
  applications ({C}harlotte, {NC}, 1991)}, vol.~176 of {\em Lect. Notes Control
  Inf. Sci.} Springer, Berlin, 1992, pp.~200--217.

\bibitem{Peng2004NonlinearExpectations}
{\sc Peng, S.}
\newblock Nonlinear expectations, nonlinear evaluations and risk measures.
\newblock In {\em Stochastic methods in finance}, vol.~1856 of {\em Lecture
  Notes in Math.} Springer, Berlin, 2004, pp.~165--253.

\bibitem{RosinRicketsonDimits2014MultilevelMCForCoulombCollisions}
{\sc Rosin, M., Ricketson, L., Dimits, A.~M., Caflisch, R.~E., and Cohen,
  B.~I.}
\newblock Multilevel {M}onte {C}arlo simulation of {C}oulomb collisions.
\newblock {\em Journal of Computational Physics 274\/} (2014), 140--157.

\bibitem{SzpruchTanTse2017IterativeParticleApproximationForMcKeanVlasov}
{\sc Szpruch, L., Tan, S., and Tse, A.}
\newblock Iterative particle approximation for {M}c{K}ean-{V}lasov sdes with
  application to multilevel {M}onte {C}arlo estimation.
\newblock {\em arXiv preprint arXiv:1706.00907\/} (2017).

\bibitem{TsitsiklisVanRoy1999OptimalStoppingOfMarkovProcesses}
{\sc Tsitsiklis, J.~N., and Van~Roy, B.}
\newblock Optimal stopping of {M}arkov processes: {H}ilbert space theory,
  approximation algorithms, and an application to pricing high-dimensional
  financial derivatives.
\newblock {\em IEEE Trans. Automat. Control 44}, 10 (1999), 1840--1851.

\bibitem{TsitsiklisVanRoy2001RegressionMethodForAmericanOptions}
{\sc Tsitsiklis, J.~N., and Van~Roy, B.}
\newblock Regression methods for pricing complex american-style options.
\newblock {\em IEEE Transactions on Neural Networks 12}, 4 (2001), 694--703.

\end{thebibliography}

\end{document}